\newif\ifproofs\proofsfalse\ifproofs\RequirePackage[mathlines]{lineno}\fi

\newif\ifsubsections
\subsectionsfalse

\documentclass[]{pcmi}

\usepackage[sc]{mathpazo}          
\usepackage{eulervm}               
\usepackage[mathscr]{eucal}		   
\usepackage[scaled=0.86]{berasans} 
\usepackage[scaled=1]{inconsolata} 
\usepackage{float}                 
\usepackage[font=normalsize,labelfont=bf]{caption} 
\usepackage{enumitem}              
\usepackage{mathtools}             
\usepackage{xspace}					

\usepackage[%
	protrusion=true,
	expansion=false,
	auto=false
	]{microtype}

\usepackage{xcolor}
\usepackage{graphicx}
\graphicspath{{./figures/}}

\ifdraft
    \definecolor{linkred}{rgb}{0.7,0.2,0.2}
    \definecolor{linkblue}{rgb}{0,0.2,0.6}
\else
    \definecolor{linkred}{cmyk}{0.0,0.0,0.0,1.0}
    \definecolor{linkblue}{cmyk}{0,0.0,0.0,1.0}
\fi

\usepackage{verbatim}
\usepackage{tikz-cd}
\usepackage{amssymb}
\usetikzlibrary{positioning}
\usetikzlibrary{shapes.geometric}

\PassOptionsToPackage{hyphens}{url} 
\usepackage[
    setpagesize=false,
    pagebackref,
	pdfpagelabels,
	plainpages=false,
    pdfstartview={FitH 1000},
    bookmarksnumbered=false,
    linktoc=all,
    colorlinks=true,
    anchorcolor=black,
    menucolor=black,
    runcolor=black,
    filecolor=black,
    linkcolor=linkblue,
	citecolor=linkblue,
	urlcolor=linkred,
]{hyperref}
\usepackage[backrefs,msc-links,initials,nobysame]{amsrefs}

\customizeamsrefs 

\theoremstyle{plain}
\newtheorem{lemma}[equation]{Lemma}
\newtheorem{proposition}[equation]{Proposition}
\newtheorem{corollary}[equation]{Corollary}
\newtheorem{theorem}[equation]{Theorem}

\theoremstyle{definition}
\newtheorem{exercise}[equation]{Exercise}

\newtheorem{definition}[equation]{Definition}
\newtheorem{remark}[equation]{Remark}
\newtheorem{example}[equation]{Example}
\newtheorem{Algorithm}[equation]{Algorithm}
\newtheorem{method}[equation]{Method}

\newlist{enumex}{enumerate}{1}
\setlist[enumex,1]{label=\textbf{\alph*.},leftmargin=*,nosep}

\newcommand{\N}		{\ensuremath{\mathbb{N}}}
\newcommand{\Z}		{\ensuremath{\mathbb{Z}}}
\newcommand{\Q}		{\ensuremath{\mathbb{Q}}}
\newcommand{\R}		{\ensuremath{\mathbb{R}}}
\newcommand{\C}		{\ensuremath{\mathbb{C}}}
\newcommand{\F}		{\ensuremath{\mathbb{F}}}

\newcommand\tensor\otimes

\newcommand{\fm}		{\ensuremath{\mathfrak{m}}}
\newcommand{\fn}		{\ensuremath{\mathfrak{n}}}
\newcommand{\fp}		{\ensuremath{\mathfrak{p}}}
\newcommand{\fq}		{\ensuremath{\mathfrak{q}}}

\newcommand{\fb}	{\mathfrak{b}}

\newcommand{\genlegendre}[4]{%
  \genfrac{(}{)}{}{#1}{#3}{#4}%
  \if\relax\detokenize{#2}\relax\else_{\!#2}\fi%
}
\newcommand{\leg}[3][]{\genlegendre{}{#1}{#2}{#3}}

\DeclareMathOperator{\Hom}{Hom}

\DeclareMathOperator{\Aut}{Aut}
\DeclareMathOperator{\End}{End}

\DeclareMathOperator{\rad}{rad}
\DeclareMathOperator{\Mat}{Mat}
\DeclareMathOperator{\Tr}{Tr}
\DeclareMathOperator{\Trad}{Trad}
\DeclareMathOperator{\coker}{coker}
\DeclareMathOperator{\im}{im}

\DeclareMathOperator{\ord}{ord}
\DeclareMathOperator{\sgn}{sgn}
\DeclareMathOperator{\id}{id}
\DeclareMathOperator{\rk}{rk}
\DeclareMathOperator{\nil}{nil}

\tikzset{choose/.style={"\bullet",outer sep=-.9ex}}
\makeatletter
\pgfqkeys{/tikz/commutative diagrams}{
  row sep/.code={\tikzcd@sep{row}{#1}{}},
  column sep/.code={\tikzcd@sep{column}{#1}{}},
  bo row sep/.code={\tikzcd@sep{row}{#1}{between origins}},
  bo column sep/.code={\tikzcd@sep{column}{#1}{between origins}},
  bo column sep/.default=normal,
  bo row sep/.default=normal,
}
\def\tikzcd@sep#1#2#3{
  \pgfkeysifdefined{/tikz/commutative diagrams/#1 sep/#2}%
    {\pgfkeysalso{/tikz/#1 sep={\ifx\\#3\\1*\else1.7*\fi\pgfkeysvalueof{/tikz/commutative diagrams/#1 sep/#2},#3}}}%
    {\pgfkeysalso{/tikz/#1 sep={#2,#3}}}}
\makeatother

\newcommand*{\qedclaim}{\null\nobreak\hfill\ensuremath{\blacksquare}}
\newcommand{\mycomment}[1]{%
}%

\begin{document}

%
%
%
%
%
%

\title[Polynomial-time algorithms in algebraic number theory]{Polynomial-time algorithms in algebraic number theory\\\vspace{.3em}\textmd{\large Based on lectures by Hendrik Lenstra}}

%
%
\author{Dani\"el M. H. van Gent}
\address{Universiteit Leiden, Niels Bohrweg 1, 2333 CA Leiden}
\email{d.m.h.van.gent@math.leidenuniv.nl}
%
%

\begin{abstract}
There are many problems in algebraic number theory which one would like to solve algorithmically, for example computation of the maximal order \(\mathcal{O}\) of a number field, and the many problems that are most often stated only for \(\mathcal{O}\), such as inverting ideals and unit computations. However, there is no known fast, i.e.\ polynomial-time, algorithm to compute \(\mathcal{O}\), which we motivate by a reduction to elementary number theory. We will instead restrict to polynomial-time algorithms, and work around this inaccessibility of \(\mathcal{O}\).
\end{abstract}  

%
%
\maketitle

%
%

\setcounter{tocdepth}{3}
\tableofcontents

\section{Introduction}

These notes are based on lectures by Hendrik Lenstra at the PCMI summer school 2022.

In these notes we study methods for solving algebraic computational problems, mainly those involving number rings, which are fast in a mathematically precise sense.
Our motivation is to prove the following theorem. We will however prove something stronger.

\begin{theorem}\label{thm:main_intro}
There exists a polynomial-time algorithm that, given a number field \(K\), non-zero elements \(\alpha_1,\dotsc,\alpha_m\in K\) and \(n_1,\dotsc,n_m\in\Z\), decides whether \(\prod_i \alpha_i^{n_i}=1\).
\end{theorem}

As we will see in Section~\ref{sec:coprime}, even for \(K=\Q\) this problem is not entirely trivial.

In Section~\ref{sec:abgp} we work up from basic operations such as addition and multiplication of integers to computation in finitely generated abelian groups. 
Here we encounter lattices and the LLL-algorithm. 
Then, in Section~\ref{sec:fractional} we give a polynomial-time algorithm to compute the blowup of an ideal. This allows us to generalize the proofs in Section~\ref{sec:coprime} to decide whether \(\prod_i \alpha_i^{n_i}\) is a unit in the maximal order of \(K\).
The final step in the proof of Theorem~\ref{thm:main_intro} brings us back to a lattice algorithm.

In Section~\ref{sec:rings} we exhibit some algorithmic number-theoretic problems which are deemed to be hard. We support this assertion by a reduction to the prime factorization problem. In Section~\ref{sec:symbols} we give algorithms to compute Jacobi symbols, which are of independent interest.

\subsection{Algorithms}

To be able to talk about polynomial-time algorithms, we should first define what an algorithm is. 
We equip the set of natural numbers \(\N=\Z_{\geq 0}\) with a length function \(l:\N\to\N\) that sends \(n\) to the number of digits of \(n\) in base 2, with \(l(0)=1\). 

\begin{definition}\label{def:algorithm}
A \emph{problem} is a function \(f:I\to\N\) for some set of inputs \(I\subseteq \N\) and we call \(f\) a \emph{decision problem} if \(f(I)\subseteq\{0,1\}\).
An \emph{algorithm} for a problem \(f:I\to\N\) is a `non-probabilistic method' to compute \(f(x)\) for all \(x\in I\).
An algorithm for \(f\) is said to run in \emph{polynomial time} if there exist \(c_0,c_1,c_2\in \R_{>0}\) such that for all \(x\in I\) the time required to compute \(f(x)\), called the \emph{run-time}, is at most \((c_0+c_1 l(x))^{c_2}\).
We say a problem \(f\) is \emph{computable} if there exists an algorithm for \(f\).
\end{definition}

This definition is rather empty: we have not specified what a `non-probabilistic method' is, nor have we explained how to measure run-time. 
We will briefly treat this more formally.
The reader for whom the above definition is sufficient can freely skip the following paragraph.
The main conclusion is that we will not heavily rely on the formal definition of run-time in these notes.

In these notes, the word algorithm will be synonymous with the word \emph{Turing machine}.
For an extensive treatment of Turing machines, see Section 8.2 in \cite{Turing}.
A Turing machine is a model of computation described by Alan Turing in 1936 that defines an abstract machine which we these days think of as a computer. The main differences between a Turing machine and a modern day computer is that the memory of a Turing machine is a tape as opposed to random-access memory, and that a Turing machine has infinite memory.
The run-time of a Turing machine is then measured as the number of elementary tape operations: reading a symbol on the tape, writing a symbol on the tape and moving the tape one unit forward or backward.
It is then immediately clear that it is expensive for a Turing machine to move the tape around much to look up data, as opposed to the random-access memory model where the cost of a memory lookup is constant, regardless of where the data are stored in memory.
This also poses a problem for our formal treatment of run-time, as it may depend on our model of computation. 
However, both models of computation are able to emulate each other in such a way that it preserves the property of computability in polynomial time, even though the constants \(c_0\), \(c_1\) and \(c_2\) as in Definition~\ref{def:algorithm} may increase drastically.
We use this as an excuse to be informal in these notes about determining the run-time of an algorithm.

\subsection{Basic computations}\label{sec:basic_comp}

In these notes we build up our algorithms from basic building blocks. 
First and foremost, we remark that the basic operations in \(\Z\) and \(\Q\) are fast.
Addition, subtraction, multiplication and division (with remainder in the case of \(\Z\)) can be done in polynomial time, as well as checking the sign of a number and whether numbers are equal.
We assume here that we represent a rational number by a pair of integers, a numerator and a denominator.
We may even assume the numerator and denominator are coprime:
Given \(a,b\in\Z\) we can compute their greatest common divisor \(\gcd(a,b)\) and solve the B\'ezout equation \(ax+by=\gcd(a,b)\) for some \(x,y\in\Z\) using the extended Euclidean algorithm (see Exercise~\ref{ex:mod_gcd}) in polynomial time. 
Applying these techniques in bulk we can also do addition, subtraction and multiplication of integer and rational matrices in polynomial time.
Least trivially of our building blocks, using the theory of lattices we can compute bases for the kernel and the image of an integer matrix in polynomial time, which is the topic of Section~\ref{sec:abgp}.

\subsection{Commutative algebra and number theory}

We will go through some basic definitions from commutative algebra and number theory which we assume the reader to be familiar with in the coming sections. 
We will use \cite{Atiyah} and \cite{AlgebraicNumberTheory} as a reference. 

\begin{definition}\label{def:basic_ring}
All rings by definition have a multiplicative identity element denoted \(1\). Let \(R\) be a commutative ring. 
\begin{enumerate}
\item A {\em unit} of \(R\) is an \(x\in R\) for which there exists an {\em inverse} \(y\in R\) such that \(xy=1\), and we write \(R^*\) for the group of units of \(R\). 
A {\em field} is a non-zero commutative ring of which every element but \(0\) is a unit.
\item A {\em regular element} of \(R\) is an \(x\in R\) for which multiplication by \(x\) is injective on \(R\).
A {\em domain} is a non-zero commutative ring of which every element but \(0\) is regular.
\item A {\em nilpotent} of \(R\) is an element \(x\in R\) for which there exists an \(n\in\Z_{>0}\) such that \(x^n=0\), and we write \(\textup{nil}(R)\) for the ideal of nilpotents of \(R\). We say \(R\) is {\em reduced} if its only nilpotent is \(0\).
\end{enumerate}
\end{definition}

\begin{definition}\label{def:basic_ideal}
Let \(R\) be a commutative ring. 
We say ideals \(\mathfrak{a}\) and \(\mathfrak{b}\) of \(R\) are {\em coprime} if \(\mathfrak{a}+\mathfrak{b}=R\).
We say an ideal \(\mathfrak{p}\) of \(R\) is {\em prime} if \(R/\mathfrak{p}\) is a domain and say an ideal \(\mathfrak{m}\) of \(R\) is {\em maximal} if \(R/\mathfrak{m}\) is a field.
We say \(R\) is {\em local} if it has precisely one maximal ideal, and say \(R\) is {\em semi-local} if it has only finitely many maximal ideals.
For a prime ideal \(\mathfrak{p}\) of \(R\) we write \(R_\mathfrak{p}\) for the {\em localization of \(R\) at \(\mathfrak{p}\)} (see p.~38 of \cite{Atiyah}). 
Note that if \(R\) is a domain, then the zero ideal \((0)\) is prime, the ring \(R_{(0)}\) is a field and the natural map \(R\to R_{(0)}\) is injective.
In this case we refer to \(R_{(0)}\) as the {\em field of fractions of \(R\)}.
\end{definition}

\begin{definition}\label{def:intro_nf}
A \emph{number field} is a field \(K\) containing the field of rational numbers \(\Q\) such that the dimension of \(K\) over \(\Q\) as a vector space is finite.
A \emph{number ring} is a ring isomorphic to a subring of a number field.
An \emph{order} is a domain whose additive group is isomorphic to \(\Z^n\) for some \(n\in\Z_{\geq 0}\).
Note that some authors do not require an order to be a domain, but simply a commutative ring.
Every order \(R\) is a number ring in the number field \(R_{(0)}\cong \Q\tensor_\Z R\).
Conversely, every number field \(K\) has a {\em maximal order} denoted \(\mathcal{O}_K\) (Theorem~I.1 in \cite{AlgebraicNumberTheory}). 
\end{definition}

\begin{exercise}\label{ex:nilradical_prime_intersection}
Let $R$ be a commutative ring. Show that \(\nil(R)\) is equal to the intersection of all prime ideals of \(R\).
Moreover, show that if \(\nil(R)\) is finitely generated, then \(\nil(R)\) is {\em nilpotent}, i.e.\ \(\nil(R)^n=0\) for some \(n>0\).
\end{exercise}

\begin{exercise}\label{ex:finite_prime_maximal}
Show that any finite commutative domain is a field. 
Conclude that in a general commutative ring prime ideals of finite index are maximal.
\end{exercise}

\begin{exercise}\label{ex:maximal_ideal_coprime}
Let \(R\) be a commutative ring and let \(I_1,\dotsc,I_m,J_1,\dotsc,J_n\subseteq R\) be ideals such that for all \(i, j\) we have \(I_i + J_j = R\). 
Show that \(I_1 \dotsm I_m + J_1 \dotsm J_n = R\). 
Conclude that for any two distinct maximal ideals \(\fm,\fn \subseteq R\) and any \(m,n\in\Z_{\geq 0}\) we have \(\fm^m + \fn^n = R\). Here \(\fm^0=\fn^0=R\).
\end{exercise}

\begin{exercise}[Chinese remainder theorem for ideals]\label{ex:crt}
Let \(R\) be a commutative ring and let \(I_1,\dotsc, I_n\subseteq R\) be pairwise coprime ideals.
Show that \(\bigcap_{i=1}^n I_i = \prod_{i=1}^n I_i\) and prove that the natural homomorphism 
\[  R / \Big( \bigcap_{i=1}^n I_i \Big)  \to  \prod_{i=1}^n (R/I_i) \]
is an isomorphism.
\end{exercise}

\section{Coprime basis factorization}\label{sec:coprime}

In this section we treat the following problem, which will be the motivation for the \emph{coprime basis algorithm}.

\begin{theorem}\label{thm:unit_product_1}
There is a polynomial-time algorithm that on input \(t\in\N\), \(q_1,\dotsc,q_t\in\Q^*\) and \(n_1,\dotsc,n_t\in\Z\) decides whether 
\[\prod_{i=1}^t q_i^{n_i}=1. \tag{2.1}\]
\end{theorem}

It is clear we can determine whether such a product has the correct sign: Simply take the sum of all \(n_i\) for which \(q_i<0\) and check whether the result is even. It is then sufficient to prove the following theorem instead.

\begin{theorem}\label{thm:unit_product_2}
There is a polynomial-time algorithm that on input \(s,t\in\N\), \(a_1,\dotsc,a_s,\) \(b_1,\dotsc,b_t\in\Z_{>0}\) and \(m_1,\dotsc,m_s,n_1,\dotsc,n_t\in\Z_{\geq 0}\) decides whether 
\[\prod_{i=1}^s a_i^{m_i}=\prod_{i=1}^t b_i^{n_i}. \tag{2.2}\]
\end{theorem}

In this form, the problem looks deceptively easy. 
Consider for example the most straightforward method to decide (2.2).

\begin{Algorithm}\label{alg:explict_computation}
Compute \(\prod_{i=1}^s a_i^{m_i}\) and \(\prod_{i=1}^t b_i^{n_i}\) explicitly and compare the results.
\end{Algorithm}
This method is certainly correct in that it is able to decide (2.2).
However, it fails to run in polynomial time even when \(s=t=1\).
For \(n\in\Z_{>0}\) the length \(l(2^n)\) of \(2^n\) as a string encoded in binary equals \(n+1\). 
Hence \(l(2^n)\) is not bounded by any polynomial in \(l(n)\approx \log_2 n\).
We wouldn't even have enough time to write down the number regardless of our proficiency in multiplication because the number is too long.

Another method uses the fundamental theorem of arithmetic, also known as unique prime factorization in \(\Z\).

\begin{Algorithm}\label{alg:factoring}
Factor \(a_1,\dotsc,a_s,b_1,\dotsc,b_t\) into primes and for each prime that occurs compute the number of times it occurs in the products \(\prod_{i=1}^s a_i^{m_i}\) and \(\prod_{i=1}^t b_i^{n_i}\) and compare the results.
\end{Algorithm}

It is true that once we have factored all integers into primes only a polynomial number of steps remains. If we write \(x_{ip}\) for the exponent of the prime \(p\) in \(a_i\), then we may compute \(\sum_{i=1}^t m_i x_{ip}\), the exponent of \(p\) in \(\prod_{i=1}^t a_i^{m_i}\), in polynomial time.
Moreover, the number of prime factors of \(n\in\Z_{>0}\) is at most \(l(n)\), so the number of primes occurring is at most \(\sum_{i=1}^s l(a_i)+\sum_{i=1}^t l(b_i)\), which is less than the length of the combined input.
The problem lies in the fact that we have not specified how to factor integers into primes. 
As of \today, nobody has been able to show that we can factor integers in polynomial time. Until this great open problem is solved, Algorithm~\ref{alg:factoring} is out the window.

An interesting observation is that the main obstruction in Algorithm~\ref{alg:explict_computation} lies in the exponents being large, while for Algorithm~\ref{alg:factoring} the obstruction is in the bases. 
Our proof for Theorem~\ref{thm:unit_product_2} will be to slightly tweak Algorithm~\ref{alg:factoring}. 
Namely, observe that we do not need to factor into prime elements but that it suffices to factor into pairwise coprime elements.
The following lemma follows readily from unique prime factorization.

\begin{lemma}[Unique coprime factorization]\label{lem:UCPF}
Let \(k\in\N\) and let \(c_1,\dotsc,c_k\in\Z_{>1}\) be pairwise coprime.
If for \(m_1,\dotsc,m_k,n_1,\dotsc,n_k\in\Z\) we have 
\[ \prod_{i=1}^k c_i^{m_i} = \prod_{i=1}^k c_i^{n_i}, \tag{2.3} \]
then \(m_i=n_i\) for all \(i\). \qed
\end{lemma}

\begin{exercise}
Suppose \(c_1,\dotsc,c_k\in\Z_{>1}\) are pairwise coprime and let \(m_1,\dotsc,m_k\in\Z\).
Then \(\prod_i c_i^{m_i} \in \Z_{>0}\) if and only if \(m_1,\dotsc,m_k\geq 0\).
\end{exercise}

\begin{definition}
For a subset \(A\subseteq \Z_{\geq1}\), a {\em coprime basis} is a subset \(C\subseteq\Z_{>1}\) such that the elements of \(C\) are pairwise coprime and such that \(\langle C\rangle\subseteq\Q^*\), the group generated by \(C\), contains \(A\).
The set of coprime bases for \(A\) is partially ordered (Exercise~\ref{ex:partial_order}) by the relation where \(C\leq D\) if \(\langle C\rangle \subseteq \langle D\rangle\).
\end{definition}

\begin{example}
For \(A=\{30,42\}\) a coprime basis is \(C=\{5,6,7\}\).
For \(a,m,n\in\Z_{>1}\) the set \(\{a^m,a^n\}\) has a coprime basis \(C_d=\{a^d\}\) for any \(d\) dividing \(\gcd(m,n)\), and \(C_e\leq C_d\) if and only if \(d \mid e\).
For arbitrary \(A\subseteq\Z_{\geq1}\) the set of all primes is the unique maximal coprime basis. 
\end{example}

We now propose the following algorithm for deciding (2.2).

\begin{method}\label{met:coprimefactoring}
Factor \(a_1,\dotsc,a_s,b_1,\dotsc,b_t\) into pairwise coprime \(c_1,\dotsc,c_k\in\Z_{>1}\). 
For each \(c_i\) compute the number of times it occurs in \(\prod_{i=1}^s a_i^{m_i}\) and \(\prod_{i=1}^t b_i^{n_i}\) and compare the results.
\end{method}

Now to prove Theorem~\ref{thm:unit_product_2} and in turn Theorem~\ref{thm:unit_product_1} it suffices to give a po\-ly\-no\-mial-time algorithm for computing coprime bases.

\begin{exercise}\label{ex:partial_order}
Verify that the relation \(\leq\) on the set of coprime bases is a partial order. 
In particular, show that \(\langle C\rangle = \langle D\rangle\) implies \(C=D\) for all coprime bases \(C\) and \(D\).
\end{exercise}

\begin{lemma}\label{lem:closure}
Let \(A\subseteq\Z_{\geq1}\). We write \(\overline{A}\) for the closure of \(A\cup\{1\}\) under multiplication, integer division (i.e.\ a division with integral result), and taking gcd's. Then 
\begin{enumerate}
\item \(C\subseteq\Z_{>1}\) is a coprime basis for \(A\) if and only if it is a coprime basis for \(\overline{A}\);
\item the set \(\textup{atom}(\overline{A})\) of minimal elements of \(\overline{A}\setminus\{1\}\) with respect to the divisibility relation is the unique minimal (wrt. \(\leq\)) coprime basis for \(\overline{A}\).
\item if \(C\subseteq\Z_{>1}\) satisfies \(C\subseteq\overline{A}\subseteq\langle C\rangle\) and the elements of \(C\) are pairwise coprime, then \(C=\textup{atom}(\overline{A})\).
\end{enumerate}
\end{lemma}
\begin{proof}
(1) Suppose \(C\) is a coprime basis for \(A\). For \(a=\prod_{c\in C} c^{a_c}\) and \(b=\prod_{c\in C} c^{b_c}\) in \(\langle C\rangle\cap \Z\) we have
\[ ab = \prod_{c\in C} c^{a_c+b_c}, \quad a/b = \prod_{c\in C} c^{a_c-b_c} \quad\text{and}\quad \gcd(a,b) = \prod_{c\in C} c^{\min\{a_c,b_c\}}. \]
Hence \(ab, a/b,\gcd(a,b),\in \langle C\rangle\).
As \(A\subseteq\langle C\rangle\), we have \(\overline{A}\subseteq \overline{\langle C\rangle}=\langle C\rangle\), so \(C\) is a coprime basis for \(\overline{A}\).

Suppose \(C\) is a coprime basis for \(\overline{A}\). Then \(A\subseteq\overline{A}\subseteq\langle C\rangle\). If \(C'\subseteq C\) is a coprime basis for \(A\), then it is a coprime basis for \(\overline{A}\) by the previous, so \(C'=C\) by minimality. Hence \(C\) is a coprime basis for \(A\).

(2) Firstly note that the elements of \(C=\textup{atom}(\overline{A})\) are pairwise coprime: For \(c,d\in C\) we have \(\textup{gcd}(c,d)\in\overline{A}\), so by minimality of \(c\) and \(d\) we must have \(\textup{gcd}(c,d)=1\) or \(c=d\).

Secondly, suppose \(a\in \overline{A}\) is minimal such that \(a \not\in \langle C\rangle\). 
Clearly \(1\in\langle C\rangle\), so \(a\neq 1\). 
Then some \(c\in C\) satisfies \(c \mid a\) by definition of \(C\), so \(a/c\in\overline{A}\). 
By minimality of \(a\) we have \(a/c\in\langle C\rangle\), so \(a=(a/c)\cdot c \in \langle C\rangle\).
Hence \(\overline{A}\subseteq\langle C\rangle\).

Thus \(C\) is a coprime basis for \(\overline{A}\). 
Suppose \(D\subseteq\Z_{>1}\) is a coprime basis for \(\overline{A}\). 
Then \(C\subseteq\overline{A}\subseteq \langle D\rangle\), so \(\langle C\rangle\subseteq\langle D\rangle\). 
Hence \(C\) is the unique minimal coprime basis.

(3) Clearly \(C\) is a coprime basis, so by (2) it suffices to show that \(C\) is minimal.
As \(C\subseteq \overline{A}\) we have \(\langle C\rangle \subseteq \langle \overline{A}\rangle \subseteq \langle C\rangle \), so \(\langle C\rangle=\langle\overline{A}\rangle\). Any coprime basis \(D\) should satisfy \(\langle C\rangle = \langle\overline{A}\rangle \subseteq \langle D \rangle\), so \(C\) is minimal.
\end{proof}

\begin{theorem}[Coprime basis factorization]\label{thm:graph_algorithm}
There is a polynomial-time algorithm that on input \(s\in\N\) and \(a_1,\dotsc,a_s\in\Z_{>0}\) computes \(k\in\N\) and the minimal coprime basis 
\(\{c_1,\dotsc,c_k\}\) of \(\{a_1,\dotsc,a_s\}\), as well as \((n_{ij})\in\Z_{\geq0}^{s\times k}\) such that \(a_i=\prod_{j=1}^k c_j^{n_{ij}}\) for all \(i\).
\end{theorem}

We state the algorithm first and prove the theorem later. Note that we do not require the \(a_i\) in the theorem to be distinct.

\begin{Algorithm}[cf.\ Algorithm R in \cite{FactorRefinement}]\label{alg:graph_algorithm}
Remove all \(a_i\) for which \(a_i=1\).
Construct a complete simple graph \(G\) on \(s\) vertices and label the vertices with \(a_1,\dotsc,a_s\). 
We call it a labeling because the map sending a vertex to its label need not be injective.
While there are edges in \(G\), iterate the following 5 steps:
\begin{enumerate}[topsep=0pt,itemsep=-1ex,partopsep=1ex,parsep=1ex]
\item Choose an edge \(\{U,V\}\) of \(G\), delete it from the graph, and let \(u\) and \(v\) be the labels of \(U\) and \(V\) respectively.
\item Compute \(w=\gcd(u,v)\) using the Euclidean algorithm.
\item Add a vertex \(W\) labeled \(w\) to \(G\) and connect it to \(U\), \(V\) and those vertices which are neighbors of both \(U\) and \(V\).
\item Update the labels of \(U\) and \(V\) to \(u/w\) and \(v/w\) respectively.
\item For each \(S\in\{U,V,W\}\), if the label of \(S\) is \(1\), then delete \(S\) and its incident edges from \(G\). 
\end{enumerate}
Once \(G\) has no more edges, the labels of the vertices in the graph are the required pairwise coprime elements. The remaining output can be computed in polynomial time.
\end{Algorithm}

In the graph that we construct and update, the edges represent the pairs of numbers of which we do not yet know whether they are coprime, while a missing edge denotes that we know the pair to be coprime.

\begin{example} We apply Algorithm~\ref{alg:graph_algorithm} to \((a_1,a_2)=(4500,5400)\). Since there are only two vertices, our graphs will fit on a single line. We denote the edge we choose in each iteration with a bullet and edges we have to erase are dotted. On the right we show how to keep track of the factorization of \(4500\) with minimal bookkeeping by writing it as a product of vertices in the graph.  
\begin{center}
\begin{tikzcd}[bo column sep,column sep=.5em,row sep=0em]
\text{Iteration 1:} &[1pt] 4500 \arrow[dash]{rrrrrrrrrrrrrrrr} &&&&&&&&\bullet&&&&&&&& 5400 &[10pt] 4500\\
\text{Iteration 2:} & 5 \arrow[dash]{rrrrrrrr} &&&&\bullet&&&& 900 \arrow[dash]{rrrrrrrr} &&&&&&&& 6 & 5\cdot 900 \\
\text{Iteration 3:} & 1 \arrow[dash,dotted]{rrrr} &&&& 5 \arrow[dash]{rrrr} &&\bullet&& 180 \arrow[dash]{rrrrrrrr} &&&&&&&& 6 & 5^2\cdot 180 \\
\text{Iteration 4:} &&&&& 1 \arrow[dash,dotted]{rr} && 5 \arrow[dash]{rr} &\bullet& 36 \arrow[dash]{rrrrrrrr} &&&&&&&& 6 & 5^3\cdot 36 \\
\text{Iteration 5:} &&&&&&& 5 \arrow[dash,dotted]{r} & 1 \arrow[dash,dotted]{r} & 36 \arrow[dash]{rrrrrrrr} &&&&\bullet&&&& 6 & 5^3\cdot 36 \\
\text{Iteration 6:} &&&&&&& 5 & & 6 \arrow[dash]{rrrr} &&\bullet&& 6 \arrow[dash,dotted]{rrrr} &&&& 1 & 5^3\cdot 6\cdot 6 \\
\text{Iteration 7:} &&&&&&& 5 & & 1 \arrow[dash,dotted]{rr} && 6 \arrow[dash,dotted]{rr} && 1 &&&& & 5^3\cdot 6^2\\
\text{Iteration 8:} &&&&&&& 5 &&&& 6&&&&&& & 5^3\cdot 6^2
\end{tikzcd}
\end{center}
We obtain \((c_1,c_2)=(5,6)\) and \(4500=5^3\cdot6^2\). By trial division we obtain \(5400=5^2\cdot6^3\).
\end{example}

\begin{example}
We apply Algorithm~\ref{alg:graph_algorithm} to \((a_1,a_2,a_3)=(15,21,35)\).
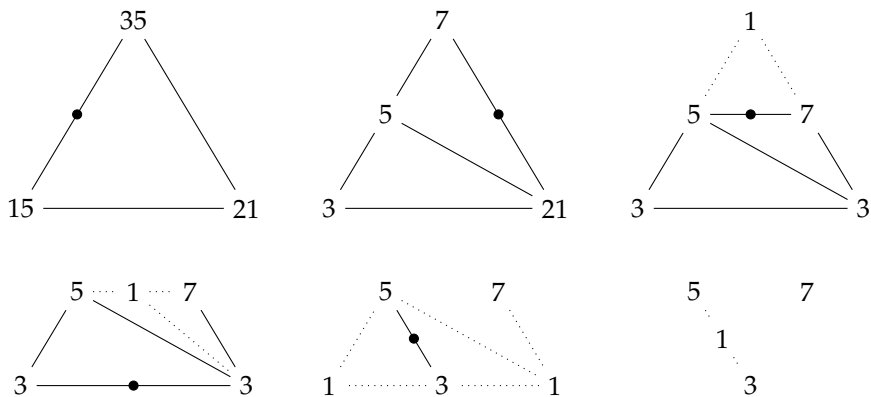
\begin{center}
\begin{tikzpicture}[scale=.5]
\path (-1,-.5) -- (7,5.5);
\node (L) at (0,0) {$15$};
\node (R) at (6,0) {$21$};
\node (T) at (3,5) {$35$};
\draw (L) -- (R) -- (T) -- (L);
\filldraw (1.5,2.5) circle (.12); 
\end{tikzpicture}
\begin{tikzpicture}[scale=.5]
\path (-1,-.5) -- (7,5.5);
\node (L) at (0,0) {$3$};
\node (R) at (6,0) {$21$};
\node (T) at (3,5) {$7$};
\node (LT) at (1.5,2.5) {$5$};
\draw (L) -- (R) -- (T) -- (LT) -- (L);
\draw (LT) -- (R);
\filldraw (4.5,2.5) circle (.12); 
\end{tikzpicture}
\begin{tikzpicture}[scale=.5]
\path (-1,-.5) -- (7,5.5);
\node (L) at (0,0) {$3$};
\node (R) at (6,0) {$3$};
\node (T) at (3,5) {$1$};
\node (LT) at (1.5,2.5) {$5$};
\node (RT) at (4.5,2.5) {$7$};
\draw (LT) -- (L) -- (R) -- (RT) -- (LT) -- (R);
\draw[dotted] (LT) -- (T) -- (RT);
\filldraw (3,2.5) circle (.12); 
\end{tikzpicture}
\end{center}
\begin{center}
\begin{tikzpicture}[scale=.5]
\path (-1,-.5) -- (7,4);
\node (L) at (0,0) {$3$};
\node (R) at (6,0) {$3$};
\node (LT) at (1.5,2.5) {$5$};
\node (RT) at (4.5,2.5) {$7$};
\node (C) at (3,2.5) {$1$};
\draw (LT) -- (L) -- (R) -- (RT);
\draw (LT) -- (R);
\draw[dotted] (LT) -- (C) -- (RT);
\draw[dotted] (C) -- (R);
\filldraw (3,0) circle (.12); 
\end{tikzpicture}
\begin{tikzpicture}[scale=.5]
\path (-1,-.5) -- (7,4);
\node (L) at (0,0) {$1$};
\node (R) at (6,0) {$1$};
\node (LT) at (1.5,2.5) {$5$};
\node (RT) at (4.5,2.5) {$7$};
\node (B) at (3,0) {$3$};
\draw (LT) -- (B);
\draw[dotted] (RT) -- (R) -- (B) -- (L) -- (LT) -- (R);
\filldraw (2.25,1.25) circle (.12); 
\end{tikzpicture}
\begin{tikzpicture}[scale=.5]
\path (-1,-.5) -- (7,4);
\node (LT) at (1.5,2.5) {$5$};
\node (RT) at (4.5,2.5) {$7$};
\node (B) at (3,0) {$3$};
\node (F) at (2.25,1.25) {$1$};
\draw[dotted] (LT) -- (F) -- (B);
\end{tikzpicture}
\end{center}
The resulting coprime basis is \((c_1,c_2,c_3)=(3,5,7)\). In the fifth graph something interesting happens. Vertex \(7\) suddenly becomes disconnected from the graph because we know it is coprime to one of the \(3\)'s.
\end{example}

\begin{proof}[Proof of Theorem~\ref{thm:graph_algorithm}]
We claim Algorithm~\ref{alg:graph_algorithm} is correct and runs in polynomial time.
One can show inductively that throughout the algorithm two vertices in the graph \(G\) have coprime labels when there is no edge between them.
When the algorithm terminates because there are no edges in the graph, we may conclude that \(c_1,\dotsc,c_s\) are coprime.
Additionally, one shows inductively that the numbers \(a_1,\dotsc,a_t\) can be written as some product of labels.
Hence \(c_1,\dotsc,c_s\) is the minimal coprime basis of \(a_1,\dotsc,a_t\) by Lemma~\ref{lem:closure}.3, as \(c_1,\dotsc,c_s\) are in the closure of \(\{a_1,\dotsc,a_t\}\).
Thus Algorithm~\ref{alg:graph_algorithm} is correct. 
It remains to show that it is fast.

Write \(P_n\) for the product of all labels of the vertices in the graph after iteration \(n\).
Note that \(P_0=a_1\dotsm a_s\) and that \(P_{n+1}\mid P_n\) for all \(n\in\N\).
Since \(P_0\) has at most \(B:=\log_2 P_0\) prime factors counting multiplicities, there are at most \(B\) iterations \(n\) for which \(P_{n+1}<P_n\), and at most \(B\) vertices in \(G\) after every iteration in the algorithm. 
The iterations for which \(P_n=P_{n+1}\) are those where the edge we chose is between coprime integers, meaning no vertices or edges are added to the graph and one edge is deleted. As the number of edges is at most \(B^2\), then so is the number of consecutive iterations for which \(P_n=P_{n+1}\).
Hence the total number of iterations is at most \(B^3\), which is polynomial in the length of the input. 
Lastly, note that each iteration takes only polynomial time because the labels of the vertices are bounded from above by \(2^B\) and the Euclidean algorithm runs in polynomial time. Hence Algorithm~\ref{alg:graph_algorithm} runs in polynomial time.
\end{proof}

The speed of this algorithm heavily depends on how fast the product of all vertices decreases.
In each step we want to choose our edge \(\{u,v\}\) such that \(\gcd(u,v)\gg 1\). 
A heuristic for this could be to choose edges between large numbers.
For those interested in the efficiency of coprime basis factorization we refer to \cite{CoprimeBase} for a provably faster algorithm.

\begin{exercise}
Show that there exists a polynomial-time algorithm that, given \(a,b,c,d\in\Z_{>0}\) such that \(ab=cd\), computes \(w,x,y,z\in\Z_{>0}\) such that \((a,b,c,d)=(wx,yz,wz,xy)\).
\end{exercise}

\begin{exercise}[Modified Euclidean algorithm] \label{ex:mod_gcd} 
Recall that \(\gcd(0,0)=0\). 
\begin{enumex}
\item Show that for all \(a,b\in\Z\) with \(b\neq 0\) there exist \(r,q\in\Z\) with \(a=qb+r\) and \(|r|\leq |b|/2\).
\item Show that for all \(q,b,r\in\Z\) with \(a=qb+r\) we have \(\gcd(a,b)=\gcd(b,r)\) and \(\gcd(a,0)=|a|\).
\item Prove that there exists a polynomial-time algorithm that, given \(a,b\in\Z\), computes \(\gcd(a,b)\) as well as \(x,y\in\Z\) such that \(ax+by=\gcd(a,b)\).
\item Conclude that there exists a polynomial-time algorithm that, given \(a,n\in\Z\) with \(n\geq1\), decides whether \((a\bmod n)\in(\Z/n\Z)^*\) and if so computes some \(a'\in\Z\) such that \(aa'\equiv 1 \text{ mod } n\).
\item Prove that there exists a polynomial-time algorithm that, given \(a,b,m,n\in\Z\) with \(n,m>1\), decides whether there exists some \(c\in\Z\) such that \(c\equiv a \text{ mod } m\) and \(c\equiv b \text{ mod } n\) and if so computes such a \(c\).
\end{enumex}
\end{exercise}

\begin{exercise}
Show that there exists a polynomial-time algorithm that, given \(k,n\in\Z_{>0}\) and \(a_1,\dotsc,a_k\in\Z\) satisfying \(a_i^2\equiv 1 \bmod n\) for all \(i\), decides whether there exists some non-empty subset \(I\subseteq\{1,\dotsc,k\}\) such that \(\prod_{i\in I} a_i \equiv 1 \bmod n\) and if so computes one such \(I\). \\ 
\emph{Hint:} Non-trivial roots of \(X^2-1\) in \(\Z/n\Z\) factor \(n\).
\end{exercise}

\begin{exercise}
We equip \(\Q^2\setminus\{(0,0)\}\) with an equivalence relation \(\sim\) where \((x_1,y_1)\sim(x_2,y_2)\) if and only if there exists some \(\lambda\in\Q^*\) such that \((\lambda x_1,\lambda y_1)=(x_2,y_2)\).
Write \(\mathbb{P}^1(\Q)=(\Q^2\setminus\{(0,0)\})/{\sim}\) for the \emph{projective line} and write \((x:y)\) for the image of \((x,y)\) in \(\mathbb{P}^1(\Q)\). 
Let \(a,b\in\Z_{>0}\) and let \(c_1,\dotsc,c_n\) be the minimal coprime basis for \(a\) and \(b\).
\begin{enumex}
\item For \(p\mid ab\) prime write \(f(p)=(\ord_p(a):\ord_p(b))\in\mathbb{P}^1(\Q)\). 
Show that every \(c_i\) naturally corresponds to a fiber of \(f\) and give the prime factorization of \(c_i\) in terms of the prime factorization of \(a\) and \(b\).
\item Suppose \(n=7\). Show that \(ab\geq 1485890406000\) and equality holds for exactly 8 pairs \((a,b)\).
\item (difficult, Problem 2023-2/C in \cite{NAW}) Give an asymptotic formula for the minimum of \(\log(ab)\) in terms of \(n\).
\end{enumex}
\end{exercise}

\begin{exercise} \label{ex:square_multiply}
Let \(n\in\Z_{>0}\).
We encode a matrix \(\overline{M}=(\overline{m}_{ij})_{i,j}\) over \(\Z/n\Z\) as a matrix \(M=(m_{ij})_{i,j}\) over \(\Z\) such  that \(0\leq m_{ij} < n\) and \( \overline{m}_{ij} \equiv m_{ij}  \bmod n\) for all \(i,j\).
Show that there exist polynomial-time algorithms for the following problems:
\begin{enumex}
\item given \(n\in\Z_{\geq 0}\) and matrices \(M\) and \(N\) over \(\Z/n\Z\), compute \(M+N\) and \(M\cdot N\) if well-defined;
\item given \(n,k\in\Z_{>0}\) and a square matrix \(M\) over \(\Z/n\Z\), compute \(M^k\); \\ \textit{Note:} An algorithm that takes \(k\) steps is not polynomial-time!
\item given \(n\in\Z_{>0}\) and a matrix \(M\) over \(\Z/n\Z\), compute a row-echelon form of \(M\);
\item given \(n\in\Z_{>0}\) and a square matrix \(M\) over \(\Z/n\Z\), compute \(\det(M)\) and \(\Tr(M)\);
\item given \(n\in\Z_{>0}\) and a matrix \(M\) over \(\Z/n\Z\), decide whether \(M^{-1}\) exists and if so compute it.
\end{enumex}
You may use the following fact: For every \(k,B\in\Z_{>0}\) and matrix \(M=(m_{ij})_{i,j}\in \Z^{k\times k}\) with \(|m_{ij}|\leq B\) for all \(i,j\) it holds that \(|\det(M)|\leq B^k \cdot k^{k/2}\) (see Hadamard's inequality, Exercise~\ref{ex:hadamard}).
\begin{enumex}[resume]
\item Show that there exists a polynomial-time algorithm that, given a square integer matrix \(M\), computes \(\det(M)\) and \(\Tr(M)\).
\end{enumex}
\end{exercise}

\begin{exercise}
Show that there exist polynomial-time algorithms for the following problems:
\begin{enumex}
\item given \(a,p,q\in\Z\) with \(p\) and \(q\) prime and \(\gcd(a,p)=1\), compute \(e\in\Z_{\geq 0}\) such that the order of \(a\) in \((\Z/p\Z)^*\) equals \(uq^e\) for some \(u\in\Z_{>0}\) with \(\gcd(u,q)=1\);
\item given \(a,p\in\Z\) with \(p\) prime, decide whether \(a\) is a square modulo \(p\);
\item given \(a,b,p\in\Z\) with \(p\) prime, \(a\) a square modulo \(p\) and \(b\) not a square modulo \(p\), compute \(c\in\Z\) such that \(c^2\equiv a \bmod p\);
\item given \(a,b,p\in\Z\) with \(p\) prime, compute \(c\in\Z\) such that \(c^2\) equals \(a\), \(b\) or \(ab\) modulo~\(p\).
\end{enumex}
\end{exercise}

\begin{exercise}
Show that there exists a polynomial-time algorithm that, given \(a,b,k,n\in\Z\) with \(k,n >0\) and \(a^k\equiv 1\text{ mod }n\) and \(b^k\equiv -1 \text{ mod } n\), computes some \(c\in\Z\) such that \(a\equiv c^2 \text{ mod }n\). \\
\textit{Hint:} First consider \(n\) odd and \(k\) a power of~\(2\).
\end{exercise}

\section{Finitely generated abelian groups} \label{sec:abgp}

In this section we treat algorithms on finitely generated abelian groups.
This will be used as a basis for algorithmic aspects of ring theory; 
for example, the additive group of orders \(R\), as well as finitely generated modules over \(R\), will be groups to which we can apply the methods of this section.

We begin by specifying a representation for our finitely generated groups.
Recall that every finitely generated abelian group \(A\) fits in some exact sequence
\begin{center}
\begin{tikzcd}[column sep=small]
\Z^m \arrow{r}{\alpha} & \Z^{n} \arrow{r}{f} & A \arrow{r} & 0.
\end{tikzcd}
\end{center}
Namely, we obtain \(n\) and \(f\) by writing down some generators \(a_1,\dotsc,a_n\in A\) for \(A\) and let \(f\) map the \(i\)-th standard basis vector to \(a_i\). 
For \(m\) and \(\alpha\) we repeat the procedure with \(A\) replaced by \(\ker(f)\).
Note that \(\alpha\), being a morphism between free \(\Z\)-modules, has a natural representation as a matrix with integer coefficients.
The isomorphism theorem gives \(A\cong \Z^n/\ker(f) = \Z^n/\im(\alpha) = \coker(\alpha)\), so \(A\) is, up to isomorphism, completely defined by \(n\) and the image of \(\alpha\).
Thus we choose to encode \(A\) as the matrix corresponding to \(\alpha\).
We say representations \(\alpha:\Z^k\to\Z^l\) and \(\beta:\Z^m\to \Z^n\) are {\em equivalent} if \(l=n\) and \(\im(\alpha)=\im(\beta)\).

A morphism \(f:A\to B\) of finitely generated abelian groups in terms of this representation gives a commutative diagram of exact sequences
\begin{equation*}\label{eq:abgp_morphism}
\begin{tikzcd}[column sep=small]
\Z^k \arrow{r}{\alpha} & \Z^{l} \arrow{r} \arrow[dashed]{d}{\varphi} & A \arrow{r} \arrow{d}{f} & 0 \\
\Z^m \arrow{r}{\beta} & \Z^n \arrow{r} & B \arrow{r} &0.
\end{tikzcd} 
\tag{\ref{sec:abgp}.1}
\end{equation*}
Here \(\varphi\) is any morphism that makes the diagram commute. 
We encode \(f\) by the matrix representing \(\varphi\).
Important to note is that not every \(\varphi\) defines a morphism \(f:A\to B\). 
It defines a morphism precisely when \(\im(\varphi\circ \alpha)\subseteq\im(\beta)\), however it is not immediately obvious how to test this.
Computing the composition of morphisms and evaluating morphisms in this form is straightforward, as it is just matrix multiplication.

To work with abelian groups in our algorithms it takes more than just to specify an encoding. 
The following is a list, in no particular order, of operations we would like to be, and later will be, able to perform in polynomial time.
\begin{enumerate}
\item decide whether a matrix encodes a morphism of given abelian groups;
\item compute kernels, images and cokernels of homomorphisms;
\item test if a homomorphism is injective/surjective; 
\item decide if two homomorphisms are equal;
\item decide if there exists an element in the preimage of a group element under a homomorphism and if so compute such an element;
\item compute direct sums, tensor products and homomorphism groups of pairs of abelian groups; 
\item compute the order of a given group element;
\item compute the order/exponent of a finite abelian group;
\item decide whether an exact sequence splits and if so compute a corresponding section and retraction;
\item compute the torsion subgroup of an abelian group;
\item write an abelian group as a direct sum of cyclic groups.
\end{enumerate}
We will spend this section working towards the last entry of this list: An algorithmic version of the structure theorem of finitely generated abelian groups.

All algorithms for the above problems will be straight-forward. 
The only serious complication arises at the fundamentals, where we reduce to linear algebra over \(\Z\).
Here we encounter the LLL-algorithm, a lattice basis reduction algorithm.

As a rule of thumb, all problems on finitely generated abelian groups can be solved in polynomial time, unless the output of the problem is not polynomially bounded or one implicitly requires prime factorization (c.f.\ Algorithm~\ref{alg:explict_computation} and Algorithm~\ref{alg:factoring}). For the first exception, consider the computation of the group \(A^k\) for an abelian group \(A\) and positive integer \(k\).
For the second, we should not expect to be able to factor a finite abelian group as a product of cyclic groups of prime-power order, since this would provide a factorization of \(n\) if we input \(\Z/n\Z\).

Finally, we address an important subtlety that arises from our choice of encoding.
As an exercise, one can prove the following lemma.

\begin{lemma}\label{lem:dlp}
Assuming the above problems have polynomial-time algorithms, we may solve the \emph{discrete logarithm problem} in polynomial time. 
That is, given an abelian group \(A\) and elements \(a,b\in A\), decide whether there exists some positive integer \(n\) such that \(na=b\) and if so compute such \(n\).
\end{lemma}

\noindent It is well known that the discrete logarithm problem for \(\F_q^*\) or elliptic curves over finite fields is difficult, i.e.\ not known to be solvable in polynomial time, even though both are finitely generated abelian groups (see Section~2 of \cite{DiffieHellman}).
The apparent paradox is solved by observing that the difficulty is in representing \(\F_q^*\) and its elements in our encoding. 
For starters, we need to write down generators for \(\F_q^*\) and subsequently write the input to our algorithms in terms of these generators. Doing so is almost equivalent to the discrete logarithm problem.

\begin{exercise}
Prove Lemma~\ref{lem:dlp}.
\end{exercise}

\begin{exercise}[Product, Proposition 2.3.16 in \cite{Iuliana}]
Show that there exists a po\-ly\-no\-mial-time algorithm that, given finitely generated abelian groups \(A\) and \(B\), computes the group \(A\times B\) and the corresponding inclusions and projections.
\end{exercise}

\subsection{Lattices and short bases}

To understand general finitely generated abelian groups, we first need to understand the simplest instances, the free abelian groups. 
It will turn out to be fruitful to consider free abelian groups together with a positive-definite inner product.
Any algorithm that keeps the vectors short with respect to the inner product, also keeps the vectors short with respect to length of the encoding, so that our algorithms have a chance to run in polynomial time.
This will allow us later to compute images and kernels of linear maps.

\begin{definition}
A \emph{Euclidean (vector) space} is a finite-dimensional positive-definite real inner product space.
For elements \(x\) and \(y\) we denote their inner product by \(\langle x,y\rangle\), and will write \(q(x)=\langle x,x\rangle\) for the \emph{square norm} of \(x\).
Note that Euclidean spaces come with a natural metric, and in particular they are topological spaces.
A \emph{lattice} is a discrete subgroup of a Euclidean space, i.e.\ a subgroup \(\Lambda\) for which \(\inf\{ q(x) : x\in\Lambda \setminus\{0\}\}>0\). An isomorphism of lattices is an inner product preserving isomorphism of groups.
\end{definition}

\begin{example}\label{ex:order_lattice}
\textbf{1.} For \(n\in\Z_{\geq 0}\) the vector space \(\R^n\) has an inner product given by
\[\langle(x_1,\dotsc,x_n),(y_1,\dotsc,y_n)\rangle = \sum_{i=1}^n x_i y_i.\]
Any basis (as \(\R\)-module) of \(\R^n\) is a basis (as \(\Z\)-module) of a lattice in \(\R^n\) (Proposition~3.3 in \cite{LatticeIntro}). The standard basis generates the lattice \(\Z^n\).

\textbf{2.} A Euclidean space one naturally encounters in number theory is \(K\tensor_\Z\R\), for any number field \(K\), which we equip with the inner product
\begin{align*}
\langle x,y\rangle = \frac{1}{[K:\Q]} \sum_{\sigma:K\tensor_\Z\R\to\C} \sigma(x) \cdot \overline{\sigma(y)},
\end{align*}
where the sum ranges over all \(\R\)-algebra homomorphisms. In this Euclidean space every order of \(K\) is a lattice (Lemma~10.3 in \cite{Stevenhagen}). 
\end{example}

\begin{exercise}
Let \(V\) be a Euclidean vector space. Show that the inner product is uniquely determined by \(q\), i.e.\ give a formula for \(\langle x,y\rangle\) in terms of \(q\).
\end{exercise}

\begin{definition}\label{def:rk}
For a ring \(R\) and an \(R\)-module \(M\) we say \(M\) is free if \(M\cong \bigoplus_{s\in S} R\) for some set \(S\), i.e.\ it has a basis.
If \(R\) is non-zero and commutative and \(M\) is free, then the cardinality of such \(S\) is unique, and we write \(\rk_R M=\rk M\), the \emph{rank} of \(M\), for this cardinality.
\end{definition}

\begin{proposition}[cf.\ Proposition 4.1 in \cite{LatticeIntro}]\label{prop:lattice_eq}
Any lattice \(\Lambda\) in a Euclidean space \(V\) is a free \(\Z\)-module with \(\rk \Lambda \leq \dim V\) and the restriction of the inner product to \(\Lambda\) is \(\Z\)-bilinear, real-valued, symmetric and satisfies \(\inf\{\langle x,x\rangle\,|\, x\in\Lambda\setminus\{0\}\}>0\).
Conversely, every free \(\Z\)-module \(\Lambda\) of finite rank equipped with a \(\Z\)-bilinear, real-valued, symmetric form \(\varphi\) for which \(\inf\{\varphi(x,x)\mid x\in\Lambda\setminus\{0\}\}>0\) can be embedded in a Euclidean vector space such that the inner product restricted to \(\Lambda\) equals \(\varphi\). \qed
\end{proposition}

We say a lattice \(\Lambda\) in a Euclidean space \(V\) is \emph{full-rank} if \(\rk \Lambda=\dim V\).

Proposition~\ref{prop:lattice_eq} shows that we have an equivalent definition of a lattice that does not require an ambient vector space. 
The bilinear form \(\varphi\) in the proposition is again naturally given by a matrix \((\varphi(b_i,b_j))_{1\leq i,j\leq n}\), the \emph{Gram-matrix}, where \((b_1,\dotsc,b_n)\) is the (\(\Z\)-)basis encoding \(\Lambda\). 
For our computational purposes it is practical to restrict to Gram-matrices with rational entries.
This will be our encoding for lattices.
Unfortunately, the Gram-matrix of the inner product from Example~\ref{ex:order_lattice}.2 is generally not of this form.

An algorithmic problem we will encounter is computing a `short basis' for a lattice \(\Lambda\subseteq\Z^n\).

\begin{definition}
Let \(\Lambda\) be a lattice. 
Let \((b_1,\dotsc,b_n)\) be a basis for \(\Lambda\) and consider its Gram-matrix \(M=(\langle b_i,b_j\rangle)_{1\leq i,j\leq n}\).
We define the \emph{determinant} of \(\Lambda\) to be \(\det(\Lambda)=|\det(M)|^{1/2}\).
\end{definition}

\begin{exercise}
Show that the determinant of a lattice does not depend on the choice of basis and is thus well-defined. \\
\emph{Hint:} Show that any two bases for \(\Lambda\) `differ' by a matrix of unit determinant.
\end{exercise}

\begin{exercise}
Let \(\Lambda\subseteq\R^n\) be a lattice of rank \(n\) and \(S\subseteq \R^n\) a measurable bounded set, i.e.\ a set with a well-defined finite volume.
Prove that if the addition map \(\Lambda \times S \to \R^n\) is injective, then \(\textup{vol}(S)\leq \det(\Lambda)\) holds, and prove that if it is surjective, then \(\textup{vol}(S)\geq \det(\Lambda)\) holds.
\end{exercise}

The determinant \(\det(\Lambda)\) equals the volume of the parallelepiped spanned by a basis for \(\Lambda\).
Since the determinant is an invariant, finding a `shorter' basis is equivalent to finding a `more orthogonal' basis.
In the lattice \(\Z b_1+\Z b_2\) below we can find a better basis.
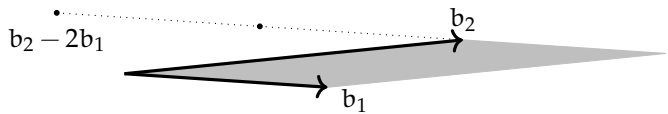
\begin{center}
\begin{tikzpicture}[scale=.9]
\filldraw[lightgray] (0,0) -- (5,.5) -- (8,.3) -- (3,-.2) -- (0,0);
\draw[->,very thick] (0,0) -- (3,-.2);
\draw[->,very thick] (0,0) -- (5,.5);
\draw[dotted] (5,.5) -- (-1,.9);
\filldraw[black] (-1,.9) circle (1pt);
\filldraw[black] (2,.7) circle (1pt);
\node (B2) at (5,.8) {$b_2$};
\node (B1) at (3.4,-.4) {$b_1$};
\node (B3) at (-1,.5) {$b_2-2b_1$};
\end{tikzpicture}
\end{center}
Taking \(c_1=b_1\) and \(c_2=b_2-2b_1\) produces the following basis:
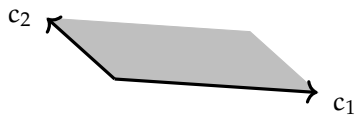
\begin{center}
\begin{tikzpicture}[scale=.9]
\filldraw[lightgray] (11,0) -- (10,.9) -- (13,.7) -- (14,-.2) -- (11,0);
\draw[->,very thick] (11,0) -- (14,-.2);
\draw[->,very thick] (11,0) -- (10,.9);
\node (C2) at (9.6,.9) {$c_2$};
\node (C1) at (14.4,-.4) {$c_1$};
\end{tikzpicture}\phantom{AAAAAAAAAAAAA}
\end{center}
For a Euclidean space the Gram--Schmidt algorithm transforms a basis into an orthogonal one as follows.

\begin{definition}\label{def:Gram-Schmidt}
Let \(V\) be a Euclidean space with basis \(B=(b_1,\dotsc,b_n)\). We iteratively define
\[\mu_{ij} = \frac{\langle b_i,b_j^*\rangle}{\langle b_j^*,b_j^*\rangle} \quad \text{for \(1\leq j<i\) and} \quad b_i^* = b_i-\sum_{j<i} \mu_{ij} b_j^* \quad\text{for \(1\leq i\leq n\).} \]
We call \(B^*=(b_1^*,\dotsc,b_n^*)\) and \((\mu_{ij})_{j<i}\) the \emph{Gram--Schmidt basis} and \emph{Gram--Schmidt coefficients}, respectively, corresponding to \(B\).
\end{definition}

When interpreting \(M=(\mu_{ij})_{j<i}\) as an upper-triangular matrix, we note that that \((\id+M)B^*=B\). 
In particular \(\det(B)=\det(B^*)\).
Since \(b_1^*,\dotsc,b_n^*\) are indeed pairwise orthogonal, they form an orthogonal basis for \(V\).
Sadly the Gram--Schmidt coefficients will generally not be integers, meaning that if \(B\) is a basis for a lattice \(\Lambda\), then generally \(B^*\) will not be. 
This is quite unsurprising, as not every lattice even has an orthogonal basis. 
A possible solution is to round the Gram-Schmidt coefficients to integers in every step, so that we are guaranteed to obtain a basis for \(\Lambda\). 
However, this does not yield the necessary bounds on our basis. 
In the next section we will state the existence of a better algorithm.

\begin{exercise}\label{ex:hadamard}
\begin{enumex}
\item Let \((b_1,\dotsc,b_n)\) be a basis for a lattice \(\Lambda\) in a Euclidean space~\(V\). Write \(\Lambda_k=\sum_{i\leq k} \Z b_i\) for all \(0\leq k\leq n\). 
Show that
\[q(b_i)\geq q(b_i^*) = \Big(\frac{\det(\Lambda_i)}{\det(\Lambda_{i-1})}\Big)^2.\]
\item Conclude \emph{Hadamard's inequality} \cite{Hadamard}: For an invertible matrix \(B\in\R^{n\times n}\) with columns \(b_1,\dotsc,b_n\) we have
\[|\det(B)|\leq \prod_{i=1}^n \|b_i\|,\]
with equality if and only if the \(b_i\) are pairwise orthogonal.
\end{enumex}
\end{exercise}

\begin{exercise}\label{ex:quotient_lattice}
Let \(V\) be a Euclidean space.
For a subspace \(W\subseteq V\) we write \(W^\bot = \{ v \in V \mid \langle v,W\rangle = 0 \}\).
\begin{enumex}
\item Show that for all \(W\subseteq V\) the natural map \(W^\bot \to V/W\) is an isomorphism of vector spaces. 
\end{enumex}
We equip \(V/W\) with the natural Euclidean vector space structure induced by \(W^\bot\). Suppose \(\Lambda\subseteq V\) is a full-rank lattice.
\begin{enumex}[resume]
\item Show that if \(W\subseteq V\) is a subspace and the image \(\overline{\Lambda}\) of \(\Lambda\) in \(V/W\) is a lattice, then \(\det(\Lambda)\geq \det(\overline{\Lambda})\cdot\det(\Lambda\cap W)\).
\end{enumex}
Suppose that \(\Lambda'\subseteq\Lambda\) is a subgroup such that \(\Lambda/\Lambda'\) is torsion free. 
\begin{enumex}[resume]
\item Show that the natural map \(\Lambda/\Lambda' \to V/\R\Lambda'\) is injective and that its image is a lattice.
\item Show that \(\det(\Lambda)=\det(\Lambda/\Lambda')\cdot\det(\Lambda')\).
\end{enumex}
\end{exercise}

\begin{exercise}\label{ex:general_dagger}
Let \(R\) be a commutative ring, \(Z\subseteq R\) a subring, \(V\) a finitely generated free \(R\)-module, and \(\varphi: V\times V\to R\) a symmetric \(R\)-bilinear map such that the natural map \(V\to\Hom_R(V,R)\) given by \(x\mapsto (y\mapsto \varphi(x,y))\) is an isomorphism. 
Let \(b_1,\dotsc,b_n\) be an \(R\)-basis for \(V\) and define \(L=\sum_i Z\cdot b_i\).
Consider \(L^\dagger=\{x\in V \mid \varphi(x,L)\subseteq Z\}\). Show that \(L^\dagger\) is generated as a \(Z\)-module by an \(R\)-basis \(b_1^\dagger,\dotsc,b_n^\dagger\) of \(V\) such that \(\varphi(b_i,b_j^\dagger)\) equals \(1\) or \(0\) according as \(i=j\) or not, and that \(L^{\dagger\dagger}=L\).
\end{exercise}

\begin{exercise}\label{ex:lattice_dual'}
Suppose that \(\Lambda\) is a lattice. Show that \(\Lambda^\dagger\), as in Exercise~\ref{ex:general_dagger} with \(Z=\Z\), \(R=\R\) and \(\varphi\) the inner product on \(\Lambda\), is a lattice with \(\det(\Lambda)\det(\Lambda^\dagger)=1\).
\end{exercise}

We call \(\Lambda^\dagger\) the {\em dual lattice} or {\em polar lattice} of \(\Lambda\).

\begin{exercise}
Let \(\Lambda\) be a lattice. Verify that \(\Hom(\Lambda,\Z)\) is a lattice when equipped with the square norm and inner product
\[ q(f) = \sup_{x\in\Lambda\setminus\{0\}} \frac{f(x)^2}{q(x)} \quad\text{and}\quad \langle f,g\rangle = \frac{q(f+g)-q(f-g)}{4} \]
respectively. Show that \(\Hom(\Lambda,\Z)\) is isomorphic to \(\Lambda^\dagger\), i.e.\ there exists an isomorphism of abelian groups that respects the inner products.
\end{exercise}

\begin{exercise}[Section~5 in \cite{Lattices}]
Let \(V_1\) and \(V_2\) be Euclidean vector spaces and let \(f:V_1\to V_2\) be a linear map.
\begin{enumex}
\item Show that there exists a unique linear map \(f^\dagger:V_2\to V_1\), which we will call the {\em adjoint} of \(f\), such that for all \(x\in V_1\) and \(y\in V_2\) we have \(\langle f(x),y\rangle=\langle x,f^\dagger(y)\rangle\).
\item Suppose \(\Lambda_i\subseteq V_i\) are full-rank lattices and \(f \Lambda_1 \subseteq \Lambda_2\). 
Show that \(f^\dagger \Lambda_2^\dagger \subseteq \Lambda_1^\dagger\). 
\end{enumex}
Suppose that \(g:\Lambda_1\to\Lambda_2\) is a group homomorphism between lattices.
\begin{enumex}[resume]
\item Conclude that \(g\) has a well-defined adjoint.
\item Show that
\[ \det(\ker(g)) \cdot \det(\im(g)) \cdot\det(\Lambda_1^\dagger) = \det(\ker(g^\dagger)) \cdot \det(\im(g^\dagger))\cdot \det(\Lambda_2).\]
\end{enumex}
\end{exercise}

\subsection{The LLL-algorithm}

In this section we state the existence the LLL-algorithm, which efficiently produces a `small' basis for a given lattice.
We will not prove its correctness, nor will we actually describe the algorithm.
What we will do is define what we mean by `small' bases in the context of the LLL-algorithm and derive some of their properties.
In this section we will use \cite{Lattices} as our reference. 
Further references include \cite{LLL} and \cite{Silverman}.

\begin{definition}[Section~10 in \cite{Lattices}]
Let \(\Lambda\) be a lattice with basis \(B=(b_1,\dotsc,b_n)\) and let \((b_1^*,\dotsc,b_n^*)\) and \((\mu_{ij})_{j<i}\) be its corresponding Gram--Schmidt basis and coefficients as defined in Definition~\ref{def:Gram-Schmidt}. 
Let \(c \geq 1\). We say \(B\) is \(c\)-reduced if
\begin{enumerate}[nosep,label={(\arabic*)}]
\item For all \(1\leq j < i \leq n\) we have \(|\mu_{ij}|\leq \tfrac{1}{2}\).
\item For all \(1\leq k < n\) we have \(c q(b_{k+1}^*)\geq q(b_k^*)\).
\end{enumerate} 
\end{definition}

Note that the first condition states that \(B\) is close to the Gram--Schmidt basis. 
The \(b_i^*\) can be interpreted as elements of the lattice \(\Lambda/\Lambda_{i-1}\), with the notation as in Exercise~\ref{ex:hadamard}.a and Exercise~\ref{ex:quotient_lattice}.c.
The second condition, the {\em Lov\'asz condition}, for \(c=1\) implies that \(b_i^*\) is in fact a shortest non-zero vector of \(\Lambda/\Lambda_{i-1}\). For \(c>1\) this condition can be seen as an approximation.
We may compute for \(c>\frac{4}{3}\) a \(c\)-reduced basis, and in particular it always exists.
That we may in fact compute it in polynomial time is non-trivial.

\begin{theorem}[LLL-algorithm \cite{LLLoriginal}]\label{thm:LLL}
There exists an algorithm that, given \(c>\frac{4}{3}\) and a lattice \(\Lambda\), produces a \(c\)-reduced basis for \(\Lambda\). For fixed \(c\) it runs in polynomial time. \qed
\end{theorem}

We should warn the reader that the literature contains various definitions of a `reduced basis', and even in the context of the LLL-algorithm there are at least two. The one we gave is sufficient for our purposes.

\begin{definition}\label{def:succ_min}
Let \(\Lambda\) be a lattice of rank \(n\).
For \(0<i\leq n\) we define the \emph{\(i\)-th successive minimum} to be the value
\[ \lambda_i(\Lambda) = \min\{ r \in \R_{\geq 0}\,|\, \rk\langle x\in \Lambda \mid q(x)\leq r \rangle \geq i \}, \]
a quantity that indicates how large the generating vectors of a rank \(i\) sublattice need to be.
\end{definition}

A basis of vectors attaining the successive minima is the gold standard of `small' bases, although it does not always exist (Exercise~\ref{ex:succ_min_not_attain}). 
The following proposition gives bounds on how far away a reduced basis can be from these minima, which is the justification for calling reduced bases `small'.

\begin{proposition}\label{prop:succ_min}
Let \(c\geq \frac{4}{3}\) and suppose \((b_1,\dotsc,b_n)\) is a \(c\)-reduced basis for a lattice \(\Lambda\).
Then for all \(0<i\leq n\) we have \(c^{1-n} \cdot q(b_i)\leq \lambda_i(\Lambda)\leq c^{i-1} \cdot q(b_i)\).
\end{proposition}
\begin{proof}
See Exercise~\ref{ex:succ_min}, or Section~10 in \cite{Lattices}.
\end{proof}

\begin{exercise}\label{ex:succ_min_not_attain}
Show that there exists a lattice \(\Lambda\) for which no basis \(b_1,\dotsc,b_n\) attains the successive minima, i.e.\ satisfies \(q(b_i)=\lambda_i(\Lambda)\) for all \(i\). \\
\emph{Hint:} Consider \(2\Z^n\subseteq\Lambda\subseteq\Z^n\).
\end{exercise}

\begin{exercise}\label{ex:succ_min}
Let \(c\geq \frac{4}{3}\) and \(0<i\leq n\) and suppose \((b_1,\dotsc,b_n)\) is a \(c\)-reduced basis for \(\Lambda\).
\begin{enumex}
\item Show that \(q(b_j^*)\leq c^{i-j}q(b_i^*)\) for all \(j\leq i\).
\item Recall that \(b_i=b_i^* + \sum_{j<i} \mu_{ij} b_j^*\). Show that \(q(b_i)\leq c^{i-1} q(b_i^*)\).
\item Show that \( q(b_j)\leq c^{i-1} q(b_i^*) \leq c^{i-1} q(b_i)\) for all \(j\leq i\). 
\item Conclude that \(\lambda_i(\Lambda)\leq\max\{q(b_j)\mid j\leq i\}\leq c^{i-1} q(b_i)\).
\end{enumex}
Write \(\Lambda_k=\sum_{j\leq k} \Z b_j\) for all \(0\leq k\leq n\).
\begin{enumex}[resume]
\item Prove that for all \(0<k\leq n\) and \(x\in\Lambda_k\setminus\Lambda_{k-1}\) we have \(q(x)\geq q(b_k^*)\).
\end{enumex}
Write \(S=\{x\in\Lambda\mid q(x)\leq \lambda_i(\Lambda)\}\) and let \(k\) be minimal such that \(S\subseteq\Lambda_k\).
\begin{enumex}[resume]
\item Show that \(k\geq \rk \langle S\rangle \geq i\).
\item Conclude that \(\lambda_i(\Lambda)\geq q(b_k^*) \geq c^{1-n} q(b_i)\).
\end{enumex}
\end{exercise}

\begin{exercise}
Let \(c\geq \frac{4}{3}\) and suppose \((b_1,\dotsc,b_n)\) is a \(c\)-reduced basis for a lattice \(\Lambda\). Show that
\[\det(\Lambda)^2 \leq \prod_{i=1}^n q(b_i) \leq c^{\binom{n}{2}} \det(\Lambda)^2.\]
\end{exercise}

\begin{exercise}[Proposition~1.12 in \cite{LLLoriginal}]\label{ex:Ge10}
Let \(c\geq \tfrac{4}{3}\), let \((b_1,\dotsc,b_n)\) be a \(c\)-reduced basis of some lattice \(\Lambda\) and let \(a_1,\dotsc,a_m\in\Lambda\) be linearly independent. 
\begin{enumex}
\item For \(1\leq i \leq m\) write \(a_i=\sum_{j} x_{ij} b_i\) and let \(j(i)\) be maximal such that \(x_{i j(i)}\neq 0\). Show that \(q(b_{j(i)}^*)\leq q(a_i)\).
\item Suppose the \(a_i\) are ordered so that \(j(1)\leq j(2)\leq \dotsc \leq j(m)\). Show that \(i\leq j(i)\) for all \(1\leq i\leq m\).
\item Show that for all \(1\leq i\leq m\) we have
\[q(b_i) \leq c^{n-1} \cdot \max\{ q(a_1),q(a_2),\dotsc,q(a_m)\}.\]
\end{enumex}
\end{exercise}

\subsection{The kernel-image algorithm}

The LLL-algorithm allows us to construct the kernel-image algorithm, which computes kernels and images of morphisms of finitely generated free abelian groups.
Most of the algorithms for finitely generated abelian groups from the beginning of this section will follow from this.
We first need a theorem from linear algebra.

\begin{theorem}[Cramer's rule, Theorem~4.1 in \cite{LinearAlgebra}]\label{thm:cramer}
Suppose \(A\in\R^{n\times n}\) is an invertible matrix and let \(b\in\R^n\).
Write \(A_i\) for the matrix obtained from \(A\) by replacing the \(i\)-th column with \(b\).
Then there exists a unique \(x\in\R^n\) such that \(Ax=b\), and it is given by \(x=\det(A)^{-1}\cdot (\det(A_1),\dotsc,\det(A_n))\). \qed
\end{theorem}

\begin{exercise}\label{ex:base_extension}
Let \(n\geq 0\). Suppose \(N\subseteq M \subseteq \Z^n\) are subgroups such that \(N\oplus P=\Z^n\) for some \(P\subseteq \Z^n\). Show that \(\rk(M)=\rk(N)\) if and only if \(M=N\). \\
\emph{Hint:} Split the exact sequence \(0\to N\to M\to \pi M \to 0\), for \(\pi:\Z^n\to P\). 
\end{exercise}

\begin{example}
The polynomial ring \(\R[\omega]\) is a totally ordered \(\R\)-algebra if by definition \(f<g\) when the leading term of \(f-g\) is negative. We think of \(\omega\) as being infinity. We can define an {\em extended metric}, generalizing metrics, where the distance function can take values in \(\R[\omega]\). For example the space \(\Z \times\R\), the disjoint union of \(\#\Z\) lines, admits an extended metric where \(d((i,x),(j,y))= |x-y| + |i-j| \cdot \omega\).
If \(\varphi:\Z^n\to\Z^m\) is a linear map, then \(\langle x,x\rangle=\|x\|^2+\|\varphi(x)\|^2\cdot \omega\), where \(\|{-}\|\) denotes the usual norm on \(\Z^n\) and \(\Z^m\), defines an `extended inner product' on \(\Z^n\), and \(\ker(\varphi)=\{x\in\Z^n\,|\, q(x)<\omega\}\).
Morally we use this concept in the following algorithm.
\end{example}

\begin{theorem}[Kernel-image algorithm, Section~14 of \cite{Lattices}]\label{thm:ker_im_alg}
There exists a po\-ly\-no\-mi\-al-time algorithm that, given \(m,n\in\Z_{\geq 0}\) and a linear map \(\varphi:\Z^n\to \Z^m\), computes the rank \(r\) of \(\varphi\) and injective linear maps \(\kappa:\Z^{n-r}\to \Z^n\) and \(\iota:\Z^r\to \Z^n\) such that \(\im(\kappa)=\ker(\varphi)\) and \(\im(\varphi\circ\iota)=\im(\varphi)\) and the induced map \(\Z^{n-r}\times \Z^{r}\to \Z^n\) is an isomorphism.
\end{theorem}

\begin{proof}
Write \(B\) for the largest absolute value of a coefficient of the matrix \(F\) defining \(\varphi\). 
Compute
\[ \omega = 2^{n-1}\cdot n^{n+1} \cdot B^{2n}+1\]
and note that its length is polynomially bounded by the size of the input.
Hence we can consider the lattice \(\Lambda=\Z^n\) together with the bilinear form induced by \(q(x) = \|x\|^2 + \|\varphi(x)\|^2\cdot \omega\), where \(\|{-}\|\) denotes the usual norm on \(\Z^n\) and \(\Z^m\).
Using the LLL-algorithm, compute a \(2\)-reduced basis \((b_1,\dotsc,b_n)\) of \(\Lambda\). 
We will show that this basis has the following properties:
\begin{enumerate}[nosep,label={(\alph*)}]
\item \(q(b_i)<\omega\) for \(0 < i \leq n-r\);
\item \((b_1,\dotsc,b_{n-r})\) is a basis for \(\ker(\varphi)\);
\item \(q(b_i)\geq \omega\) for \(n-r< i \leq n\);
\item \((\varphi(b_{n-r+1}),\dotsc,\varphi(b_n))\) is a basis for \(\im(\varphi)\).
\end{enumerate}
Then, we may compute \(r\) as the maximal \(0\leq r\leq n\) such that \(q(b_{n+1-r})\geq \omega\), and \(\kappa\) and \(\iota\) as the natural maps with images \(\langle b_1,\dotsc,b_{n-r}\rangle\) and \(\langle b_{n-r+1},\dotsc,b_n\rangle\) respectively.

\emph{\underline{Claim:} } For all \(0<i\leq n-r\) we have \(\lambda_i(\Lambda)\leq n^{n+1}B^{2n}\). \\
\emph{Proof.} By Cramer's rule (Theorem~\ref{thm:cramer}), we may take linearly independent vectors \(a_1,\dotsc,a_{n-r}\in\ker(\varphi)\) for which the coefficients are determinants of \(r\times r\) submatrices of \(F\).
Then by Hadamard's inequality (Exercise~\ref{ex:hadamard}), each coefficient is bounded in absolute value by \(n^{n/2} B^n\).
Hence \(q(a_i)=\|a_i\|^2\leq n^{n+1} B^{2n}\) for all \(0 < i \leq n-r\).
The claim now follows from the independence of the \(a_i\). \qedclaim

From Proposition~\ref{prop:succ_min} and the claim it follows that
\[q(b_i)\leq 2^{n-1} \cdot \lambda_i(\Lambda) \leq 2^{n-1} \cdot n^{n+1} \cdot B^{2n} < \omega\]
for all \(0<i\leq n-r\), proving (a). Clearly for all \(x\in \Lambda\) such that \(\omega > q(x)=\|x\|^2+\omega\|\varphi(x)\|^2\) we have \(\|\varphi(x)\|^2=0\) and thus \(x\in\ker(\varphi)\).
In particular, we have linearly independent \(b_1,\dotsc,b_{n-r}\in\ker(\varphi)\). 
From Exercise~\ref{ex:base_extension} we may conclude it is in fact a basis for \(\ker(\varphi)\), proving (b).
It follows from (b) that \(b_i\not\in\ker(\varphi)\) and thus \(q(b_i)\geq \omega\) for all \(n-r<i\leq n\), proving (c).
Lastly, (d) follows from (b) and the homomorphism theorem.
\end{proof}

Note that in the proof of Theorem~\ref{thm:ker_im_alg} we could have constructed a \(c\)-reduced basis for values of \(c\) other than \(2\). Moreover, the exact value of \(\omega\) is not important, as long as it is sufficiently large while still being computable in polynomial time.
Later we will encounter a version of the kernel-image algorithm for general finitely generated abelian groups (Exercise~\ref{ex:ker_im_alg_2}).

\begin{exercise}[Order, cf.\ Proposition 2.6.1 in \cite{Iuliana}]\label{ex:compute_order} Let \(n\in\Z_{\geq 0}\).
\begin{enumex}
\item Let \(M\in\Z^{n\times n}\) be an integer matrix. Show that there exists a \(T\in\Z^{n\times n}\) such that \(\det(T)=1\) and \(MT\) is upper triangular.
\item Show that a linear map \(\varphi:\Z^n\to \Z^n\) satisfies \(\#\coker(\varphi)=|\det(\varphi)|\) if \(\det(\varphi)\neq 0\).
\item Show that there exist a polynomial-time algorithm that, given a finitely generated abelian group \(A\), decides whether \(A\) is finite and if so computes \(\# A\).
\emph{Note:} The linear map representing \(A\) need not be injective.
\end{enumex}
\end{exercise}

\begin{exercise}[cf.\ Theorem~\ref{thm:unit_product_1}]\label{ex:compute_rational_unit_product}
Show that there exists a polynomial-time algorithm that, given \(\alpha_1,\dotsc,\alpha_m\in\Q^*\), computes the kernel of the map
\(\Z^m\to\Q^*\) that is given by \((k_1,\dotsc,k_m)\mapsto \prod_i \alpha_i^{k_i}\). \\
\emph{Hint:} Compute a coprime basis.
\end{exercise}

\subsection{Applications of the kernel-image algorithm}

In this subsection we provide a polynomial-time algorithm for most problems listed in the beginning of this section.
The algorithms in the coming sections, and many more, can be found in Chapter~2 of \cite{Iuliana}.
An immediate consequence of the kernel-image algorithm is the following. 

\begin{corollary}\label{cor:jective}
There exists a polynomial-time algorithm that, given a linear map \(\varphi:A\to B\) of finitely generated free abelian groups, decides whether \(\varphi\) is injective/surjective.
\end{corollary}
\begin{proof}
For injectivity we check whether \(\ker(\varphi)=0\) and for surjectivity we may check using Exercise~\ref{ex:compute_order}.c whether \(\coker(\varphi)=0\).
\end{proof}

\begin{proposition}\label{prop:compute_inclusion_image}
There exists a polynomial-time algorithm that, given linear maps \(\varphi:A\to C\) and \(\psi:B\to C\) of finitely generated free abelian groups, decides whether \(\im(\psi)\subseteq\im(\varphi)\). 
\end{proposition}
\begin{proof}
Using Theorem~\ref{thm:ker_im_alg} compute the kernel \(\kappa:K\to A\times B\) of \(A\times B\to C\) as in the following diagram.
\begin{center}
\begin{tikzcd}[row sep=1em,column sep={3em,between origins}]
& & & A \arrow{dr}{\varphi} & \\
K \arrow{rr}{\kappa} \arrow[bend right=15,dashed]{drrr}& & A\times B \arrow{ur}{\pi_A} \arrow[swap]{dr}{\pi_B}  \arrow{rr}& & C \\
& & & B \arrow[swap]{ur}{\psi} & 
\end{tikzcd}
\end{center}
The image of \(\pi_B\circ \kappa\) is precisely the set of elements \(b\in B\) for which there exist \(a\in A\) such that \(\varphi(a)=\psi(b)\).
Hence it suffices to decide using Corollary~\ref{cor:jective} whether \(\pi_B\circ \kappa\) is surjective.
\end{proof}

\begin{corollary}
There exists a polynomial-time algorithm that, given finitely generated abelian groups \(A\) and \(B\), represented by linear maps \(\alpha:A_0\to A_1\) and \(\beta:B_0\to B_1\) respectively, and a linear map \(\varphi:A_1\to B_1\), decides whether \(\varphi\) represents a morphism \(f:A\to B\).
\end{corollary}
\begin{proof}
Recall \(\varphi\) represents a morphism precisely when \(\im(\varphi\circ\alpha)\subseteq\im(\beta)\).
\end{proof}

\begin{corollary}
There exists a polynomial-time algorithm that, given finitely generated abelian groups \(A\) and \(B\) and morphisms \(f,g:A\to B\), decides whether \(f=g\).
\end{corollary}
\begin{proof}
Considering \(h=f-g\) it suffices to be able to decide whether a morphism is zero.
With \(\eta\) the matrix representing \(h\) and \(\beta\) the matrix representing \(B\), we have \(h=0\) precisely when \(\im(\eta)\subseteq\im(\beta)\).
\end{proof}

\begin{proposition}[Injectivity/surjectivity, Proposition 2.3.6 and Proposition 2.3.7. in \cite{Iuliana}]
There exists a polynomial-time algorithm that, given a morphism \(f:A\to B\) of finitely generated abelian groups, decides whether \(f\) is injective/surjective.
\end{proposition}
\begin{proof}
Write \(\alpha:A_0\to A_1\) and \(\beta:B_0\to B_1\) for the representatives of \(A\) and \(B\) respectively and \(\varphi\) for the representative of \(f\).
Note that \(f\) is surjective if and only if \(B=\im(f)=\im(\varphi)/\im(\beta)\) if and only if \(B_1=\im(\varphi)+\im(\beta)\).
It suffices to decide whether the map \(A_1\times B_0 \to B_1\) induced by \(\varphi\) and \(\beta\) is surjective, for which we have Corollary~\ref{cor:jective}.
Note that \(f\) is injective if and only if \(\ker(\varphi)\subseteq\im(\alpha)\).
Using Theorem~\ref{thm:ker_im_alg} we compute a linear map \(\kappa:K\to A_1\) with \(\im(\kappa)=\ker(\varphi)\), and apply Proposition~\ref{prop:compute_inclusion_image} to decide \(\im(\kappa)\subseteq\im(\alpha)\).
\end{proof}

\begin{proposition}\label{prop:compute_preimage}
There exists a polynomial-time algorithm that, given a linear map \(\varphi:A\to B\) of free abelian groups \(A\) and \(B\) and \(b\in B\), decides whether an element \(a\in A\) exists such that \(\varphi(a)=b\), and if so computes one.
\end{proposition}
\begin{proof}
Consider the linear map \(\psi: A\times \Z \to B\) that sends \((a,x)\) to \(\varphi(a)+xb\).
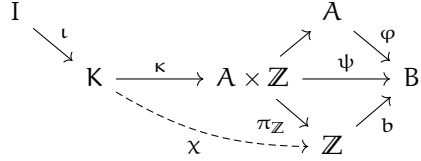
\begin{center}
\begin{tikzcd}[row sep=1em,column sep={3em,between origins}]
I \arrow{dr}{\iota} & & & & A \arrow{dr}{\varphi} & \\
& K \arrow{rr}{\kappa} \arrow[bend right=15,dashed,swap]{drrr}{\chi}& & A\times \Z \arrow{ur} \arrow[swap]{dr}{\pi_\Z}  \arrow{rr}{\psi}& & B \\
& & & & \Z \arrow[swap]{ur}{b} & 
\end{tikzcd}
\end{center}
Compute using Theorem~\ref{thm:ker_im_alg} the kernel \(\kappa: K \to A\times \Z\) of \(\psi\) and in turn \(\iota:I\to K\) a preimage of \(\chi=\pi_\Z\circ \kappa\).
Note that \(\varphi(a)=b\) has a solution precisely when \(-1\) is in the image of \(\chi\).
Moreover, if we find \(k\in K\) such that \(\chi(k)=-1\), then its image under \(K\to A\times\Z\to A\) gives an element \(a\in A\) such that \(\varphi(a)=b\).
Note that \(I\subseteq \Z\). 
If \(I=0\), then no solution to \(\chi(k)=-1\) exists, and otherwise \(k\in\iota(\{\pm1\})\) gives a solution if it exists.
\end{proof}

\begin{corollary}[Preimages]
There exists a polynomial-time algorithm that, given a homomorphism \(f:A\to B\) of finitely generated abelian groups \(A\) and \(B\) and \(b\in B\), decides whether an element \(a\in A\) exists such that \(f(a)=b\), and if so computes one. 
\end{corollary}
\begin{proof}
Let \(\varphi:A_1\to B_1\) be the representative of \(f\) and \(b_1\in B_1\) of \(b\). 
Then there exists some \(a\in A\) such that \(f(a)=b\) if and only if there exists some \(a_1\in A_1\) and \(b_0\in B_0\) such that \(\varphi(a_1)=b_1+\beta(b_0)\).
If we consider again the induced map \(A_1\times B_0 \to B_1\), then it suffices to compute an inverse image of \(b_1\) using Proposition~\ref{prop:compute_preimage}.
\end{proof}

\begin{exercise}
Give a direct proof of Proposition~\ref{prop:compute_inclusion_image} or Proposition~\ref{prop:compute_preimage} by giving an algorithm that applies the LLL-algorithm only once, similar to Theorem~\ref{thm:ker_im_alg}.
\end{exercise}

\begin{exercise}[Kernels and images, cf.\ Proposition 2.3.15 in \cite{Iuliana}]\label{ex:ker_im_alg_2} \hspace{1em}
\begin{enumex}
\item Show that there exists a polynomial-time algorithm that, given morphisms \(\delta:D_0\to D_1\) and \(\chi:C_1\to D_1\) of finitely generated free abelian groups, computes a finitely generated free abelian group \(C_0\) and an injection \(\gamma:C_0\to C_1\) such that \(\im(\delta)=\im(\chi\circ\gamma)\). 
\item Show that there exists a polynomial-time algorithm that, given a morphism \(f:A\to B\) of finitely generated abelian groups, computes injective morphisms \(k:K\to A\) and \(i:I\to B\) such that \(\im(k)=\ker(f)\) and \(\im(i)=\im(f)\). \\
\emph{Hint:} Take \((\delta_i,\chi_i)=(\beta,\varphi)\) and \((\delta_k,\chi_k)=(\alpha,\gamma_i)\).
\end{enumex}
\end{exercise}

\subsection{Homomorphism groups of finitely generated abelian groups}

Although we can algorithmically work with individual morphisms \(f:A\to B\) of finitely generated abelian groups, we have yet to treat the group of homomorphisms \(\Hom(A,B)\) as a whole. Recall that given representations \(\alpha:A_0\to\Z^a\) and \(\beta:B_0\to\Z^b\) of \(A\) and \(B\) respectively, where \(A_0\) and \(B_0\) are free, we represent a morphism \(A\to B\) as a linear map \(\Z^a\to \Z^b\) or, equivalently, as a matrix in \(\Z^{b\times a}\).

\begin{definition}
Let \(\alpha:A_0\to \Z^a\) and \(\beta:B_0\to \Z^b\) be representations of finitely generated abelian groups \(A\) and \(B\) respectively.
We write 
\begin{align*}
\mathcal{H}_1(A,B) &= \{ \varphi\in \Z^{b\times a} : \im(\varphi\circ\alpha)\subseteq\im(\beta) \} \quad\text{and}\\
\mathcal{H}_0(A,B) &= \{ \varphi\in\Z^{b\times a} : \im(\varphi) \subseteq\im(\beta)\}
\end{align*}
for the group of matrices that represent a morphism \(A\to B\) and the group of matrices that represent a zero-morphism \(A\to B\) respectively, so that \(\Hom(A,B)\cong \mathcal{H}_1(A,B)/\mathcal{H}_0(A,B)\).
A {\em homomorphism group of \(A\) and \(B\)} is a pair \((H,e)\) where \(H\) is a finitely generated abelian group, represented by say a map \(H_0\to H_1\), and \(e:H_1\to \Z^{b\times a}\) is a linear map, such that \(\im(e)=\mathcal{H}_1(A,B)\) and such that \(e\) induces an isomorphism \(H\to \mathcal{H}_1(A,B)/\mathcal{H}_0(A,B)\).
\end{definition}

This definition is somewhat more formal than what we would like to work with.
For example, \(A\) and \(B\) do not have a unique homomorphism group in the sense of the above definition, even though we would like to reason about it as being equal to the group \(\Hom(A,B)\).
Effectively, a homomorphism group of \(A\) and \(B\) is just a finitely generated abelian group \(H\) isomorphic to \(\Hom(A,B)\) together with a means of evaluation \(e\).
We motivate being somewhat lax with the definition by Exercise~\ref{ex:hom_dont_care}.

\begin{theorem}[Hom, Proposition 2.4.1.2 in \cite{Iuliana}]\label{thm:compute_hom}
There exists a polynomial-time algorithm that, given finitely generated abelian groups \(A\) and \(B\), computes \(\Hom(A,B)\).
\end{theorem}
\begin{proof}
Consider the case where \(A\) is a free abelian group and \(B\) is a finitely generated abelian group represented by \(\beta:B_0\to B_1\).
We may compute the matrix \(\beta_*:\Hom(A,B_0)\to\Hom(A,B_1)\) given by \(f\mapsto \beta\circ f\).
Since \(A\) is free we get an exact functor \(\Hom(A,\underline{\phantom{B}})\) such that 
\begin{align*}
\big[\ B_0 \xrightarrow{\beta} B_1 \to B \to 0\ \big] \Rightarrow \big[\ \Hom(A,B_0) \xrightarrow{\beta_*} \Hom(A,B_1) \to \Hom(A,B) \to 0 \ \big],
\end{align*}
and thus \(\Hom(A,B)\cong\coker(\beta_*)\).
Note that \(\Hom(A,B_0)\) and \(\Hom(A,B_1)\) are free, so \(\beta_*\) represents the group \(\Hom(A,B)\).
Evaluating an element of \(\Hom(A,B)\) in \(A\) reduces to evaluating an element of \(\Hom(A,B_1)\) in \(A\), which is just matrix multiplication. Hence we easily obtain the evaluation map.

Now consider the general case where \(A\) and \(B\) are general finitely generated abelian groups represented by \(\alpha:A_0\to A_1\) and \(\beta\) respectively. 
We note that the functor \(\Hom(\underline{\phantom{A}},B)\) is left-exact and contravariant.
Applied to the exact sequence of \(A\) we get 
\[ \big[\ A_0 \xrightarrow{\alpha} A_1 \to A \to 0\ \big] \Rightarrow \big[\ 0\to \Hom(A,B) \to \Hom(A_1,B) \xrightarrow{\alpha^*} \Hom(A_0,B) \ \big].\]
Hence \(\Hom(A,B)\cong\ker(\alpha^*)\).
By the previous case we may compute \(\Hom(A_1,B)\) and \(\Hom(A_0,B)\) since \(A_1\) and \(A_0\) are free.
It is not difficult to show we may then compute \(\alpha^*\) and in turn its kernel using Exercise~\ref{ex:ker_im_alg_2}.
Evaluation in \(A\) of elements in \(\Hom(A,B)\) reduces to evaluation in \(A_1\) of elements in \(\Hom(A_1,B)\), which we may also do by the previous case.
\end{proof}

\begin{exercise}[Exponent, Theorem 2.6.9 in \cite{Iuliana}]\label{ex:group_exponent}
Show that there exists a polynomial-time algorithm that, given a finitely generated abelian group \(A\)
\begin{enumex}
\item and an element \(a\in A\), decides whether \(a\) is torsion and if so computes \(\ord(a)\);
\item decides whether \(A\) is finite and if so computes its exponent \(\exp(A)\) and an \(a\in A\) with \(\ord(a)=\exp(A)\).
\end{enumex}
\end{exercise}

\begin{exercise}[Exact sequences, cf.\ Proposition 2.5.1 in \cite{Iuliana}]\label{ex:split_exact}
Show that there exists a polynomial-time algorithm that, given morphisms \(f:A\to B\) and \(g:B\to C\) of finitely generated abelian groups,
\begin{enumex}
\item decides whether the sequence \(0\to A \xrightarrow{f}  B \xrightarrow{g} C\to 0\) is exact;
\item if so, decides whether the sequence is split;
\item if so, produces a left-inverse of \(f\) and a right-inverse of \(g\).
\end{enumex}
\emph{Hint:} Consider the map \(g_*:\Hom(C,B)\to\Hom(C,C)\).
\end{exercise}

Note that by taking \(C=0\) in Exercise~\ref{ex:split_exact} we may conclude that there exists a polynomial-time algorithm that, given a morphism of finitely generated abelian groups, decides whether it is an isomorphism and if so computes its inverse.

\begin{exercise}\label{ex:hom_dont_care}
Show that there exists a polynomial-time algorithm that, given finitely generated abelian groups \(A\) and \(B\) and two homomorphism groups \((H,e)\) and \((H',e')\) of \(A\) and \(B\), computes an isomorphism \(f:H\to H'\) that respects evaluation, i.e.\ so that for all \(h\in H\) we have \(e(h) \equiv e'(f(h))\ (\textup{mod }\mathcal{H}_0(A,B))\).
\end{exercise}

\begin{exercise}[Tensor, Proposition 2.4.1.1 in \cite{Iuliana}]\label{ex:compute_tensor}
We naturally encode a {\em bilinear map} \(A\times B\to C\) of finitely generated abelian groups as a map \(A\to\Hom(B,C)\). Show that there exists a polynomial-time algorithm that, given finitely generated abelian groups \(A\) and \(B\), computes a group \(T\) isomorphic to \(A\tensor B\) as well as a bilinear map \(A\times B\to T\) that induces the natural bilinear map \(\tensor: A\times B\to A\tensor B\).
\end{exercise}

\begin{exercise}
Formulate and prove statements analogous to Exercise~\ref{ex:compute_tensor} for the bifunctors \(\textup{Tor}_i\) and \(\textup{Ext}_i\).
\end{exercise}

\subsection{Structure theorem for finitely generated abelian groups}

The structure theorem for finitely generated abelian groups is the following.
\begin{theorem}\label{thm:fingen_structure_theorem}
Suppose \(A\) is a finitely generated abelian group. Then there exists a unique sequence \((r,m,n_1,n_2,\dotsc,n_m)\) of integers with \(r,m\geq0\) and \(n_1,\dotsc,n_m>1\) such that \(n_m \mid \dotsm \mid n_2 \mid n_1\), for which
\[A\cong \Z^r \times \prod_{k=1}^m (\Z/n_k\Z).\]
\end{theorem}
In this section we will prove its algorithmic counterpart, which also implicitly proves the structure theorem itself.

\begin{exercise}\label{ex:double_dual}
Let \(n>0\) and \(M\subseteq\Q^n\) a subgroup. We equip \(\Q^n\) with the standard inner product.
For subgroups \(H\subseteq M\) write \(H^\bot=\{x\in M\mid\langle x,H\rangle=0\}\).
Show that for all subgroups \(N\subseteq M\) we have \((N^{\bot})^{\bot}=(\Q N)\cap M\). \\
\emph{Hint:} First consider the case where \(M\) and \(N\) are \(\Q\)-vector spaces.
\end{exercise}

\begin{lemma}[Torsion, cf.\ Theorem 2.6.2 in \cite{Iuliana}]\label{lem:compute_torsion}
There exists a polynomial-time algorithm that, given a finitely generated abelian group \(A\), computes its torsion subgroup.
\end{lemma}
\begin{proof}
For a finitely generated abelian group \(H\) and map \(h:H\to\Z^n\) we write \(h^\bot:H^\bot\to\Z^n\) for (an encoding of) a kernel of the map \(\Z^n\to\Hom(H,\Z)\) given by \(x\mapsto(y\mapsto \langle x,h(y)\rangle)\).
Note that using Theorem~\ref{thm:ker_im_alg} we may compute \(h^\bot\) in polynomial time.
In particular, we may compute \((\alpha^{\bot})^{\bot}:T\to A_1\) for the representative \(\alpha:A_0\to A_1\) of \(A\). 
It follows from Exercise~\ref{ex:double_dual} that \((\alpha^\bot)^\bot\) represents the torsion subgroup of \(A\): Its image is precisely the set of those elements of \(A_1\) for which a positive integer multiple is in \(\alpha(A_0)\).
\end{proof}

\begin{theorem}[Structure theorem, Theorem 2.6.10 in \cite{Iuliana}]\label{thm:fundament_fingen}
There exists a polynomial-time algorithm that, given a finitely generated abelian group \(A\), computes integers \(r\), \(m\), \(n_1,\dotsc,n_m\) with \(r,m\geq 0\) and \(n_1,\dotsc,n_m>1\) such that \(n_m \mid\dotsm\mid n_1\) and computes for \(A\) an isomorphism
\[A\cong \Z^r\times\prod_{k=1}^m (\Z/n_k\Z),\]
i.e.\ projections to and inclusions from the individual factors on the right hand side.
\end{theorem}
\begin{proof}
We may compute using Lemma~\ref{lem:compute_torsion} the torsion subgroup \(T\) of \(A\).
Using the image algorithm of Exercise~\ref{ex:ker_im_alg_2} we may compute an isomorphism \(\Z^r\to A/T\).
We have an exact sequence \(0\to T\to A\to A/T\to 0\) which splits, hence by Exercise~\ref{ex:split_exact} we may compute maps \(A\to T\) and \(A/T\to A\) inducing an isomorphism \(A\cong T\times(A/T)\cong T \times \Z^r\).
Replacing \(A\) by \(T\) we may now assume \(A\) is torsion. If \(A=0\) we are done.
Using Exercise~\ref{ex:group_exponent} we may compute an element \(a\in A\) with order equal to the exponent \(e\) of \(A\).
Again we have an exact sequence \(0\to \Z a \to A \to A/(\Z a) \to 0\) which is split to which we apply Exercise~\ref{ex:split_exact}. We proceed recursively with \(A\) replaced by \(A/(\Z a)\).
Note that the exponent of \(A/(\Z a)\) is a divisor of the exponent of \(A\), so indeed we will get \(n_m \mid \dotsm \mid n_1\).
\end{proof}

\begin{corollary}\label{cor:fundament2_fingen}
There exists a polynomial-time algorithm that, given a finitely generated abelian group \(A\) and a finite set \(S\) of positive integers, computes \(r,m\in\Z_{\geq 0}\) and \(c_1,\dotsc,c_m\in\Z_{>1}\) and \(t_1,\dotsc,t_m\in\Z_{>0}\) such that any two \(c_i\) are either coprime or powers of the same integer, and every \(c_i\) either divides some power of an element of \(S\) or is coprime to all elements of \(S\), and computes for \(A\) an isomorphism
\[A\cong\Z^r\times \prod_{k=1}^m (\Z/c_k\Z)^{t_k}.\]
\end{corollary}
\begin{proof}
Apply Theorem~\ref{thm:fundament_fingen} and compute a coprime basis from \(S\cup\{n_1,\dotsc,n_m\}\) using Theorem~\ref{thm:graph_algorithm}. This algorithm also expresses each \(n_i\) in terms of this basis, and we then proceed as in Theorem~\ref{thm:fundament_fingen}.
\end{proof}

\begin{exercise}\label{ex:projective}
For a ring \(R\) and \(R\)-module \(M\) we say \(M\) is {\em projective} if there exists some \(R\)-module \(N\) such that \(M\oplus N\) is a free \(R\)-module.
Show for \(n\in\Z_{>0}\) that \(n\) is square-free if and only if all \(\Z/n\Z\)-modules are projective.
\end{exercise}

\begin{exercise}\label{ex:rex}
We define for finite sets of positive integers the function
\[\textup{rex}\Big\{\prod_{p} p^{s_{1,p}},\dotsc,\prod_{p} p^{s_{m,p}}\Big\} = \smash{\prod_p} p^{\gcd\{s_{1,p},\dotsc,s_{m,p}\}}.\]
\begin{enumex}
\item Show that there exists a polynomial-time algorithm that computes \(\textup{rex}\). \\ \emph{Note:} The inputs are just integers, and \(\gcd \{\} = \gcd\{0\} = 0\).
\end{enumex}
\noindent For a finite abelian group \(A \cong \prod_{k=1}^m (\Z/a_k\Z)\) write \(\textup{rex}(A)=\textup{rex}(a_1,\dotsc,a_m)\).
\begin{enumex}[resume]
\item Show that \(\textup{rex}(A)\) is well-defined (i.e.\ does not depend on the choice of isomorphism or \(a_i\)) and can be computed in polynomial time.
\item Show for \(r=\textup{rex}(A)\) that \(r^iA/r^{i+1}A\) is a projective \(\Z/r\Z\)-module for all \(i\geq 0\), and when \(A\neq 0\) that \(r\) is the unique maximal integer with this property.
\end{enumex}
\end{exercise}

\subsection{Approximating lattices}

This section is adapted from \cite{Ge}.
Let \(\Lambda\subseteq\R^n\) be a non-zero lattice with generators \(v_1,\dotsc,v_s\).
Note that such vectors can have irrational coordinates, and thus we generally cannot encode this \(\Lambda\).
Nonetheless, we would like to compute the relations 
\[M=\Big\{(x_1,\dotsc,x_s)\in\Z^s : \sum_{i=1}^s x_i v_i = 0\Big\}\]
between the generators when only an approximation of them is available.
The challenge is to decide the required precision, which largely depends on a lower bound on \(\lambda_1(\Lambda)\) (see Definition~\ref{def:succ_min}).
Namely, we need to be able to decide whether an estimation that is close to zero is in fact zero.

For \(t\in\Z_{>0}\), a \emph{\(t\)-approximation} of \((z_1,\dotsc,z_n)\in\R^n\) is an element \((x_1,\dotsc,x_n)\in t^{-1} \Z^n\) such that \(|z_j-x_j|\leq t^{-1}\) for all \(j\).
Suppose for some \(t\) that we have for all \(i\) a \(t\)-approximation \(w_i\) of \(v_i\). We equip \(\Z^s\) with a lattice structure given by 
\[q(x)=\|x\|^2+\Big\|\sum_{i=1}^s x_i w_i\Big\|^2\cdot t^2 \in\Z,\]
where \(\|-\|\) denotes the standard norm on \(\R^s\) and \(\R^n\).

\begin{exercise}\label{ex:Geq}
Show that for all \(x\in M\) we have \(q(x)\leq (ns+1)\|x\|^2 \leq 2ns \|x\|^2\).
\end{exercise}

\begin{proposition}\label{prop:Ge6}
Let \(B=\max_i \|v_i\|\) and \(\lambda=\sqrt{\lambda_1(\Lambda)}\). 
If \(M\) has rank \(k\), then there exist linearly independent \(u_1,\dotsc,u_k\in M\) such that 
\[ \|u_i\|^2\leq (n+1)^{2n+1}\bigg(\frac{B}{\lambda} \bigg)^{2n} \leq 2^{2n+1} n^{2n+1} \bigg(\frac{B}{\lambda} \bigg)^{2n}.\]
\end{proposition}
\begin{proof}
We may assume without loss of generality that \(\Lambda\) is full-rank.
Reorder the \(v_i\) such that \(v_1,\dotsc,v_n\) are linearly independent. Let 
\[A=\bigg\lfloor \bigg(\frac{(n+1)B}{\lambda} \bigg)^{n}\bigg\rfloor \]
and consider the set \(S=\{-\tfrac{1}{2}A,-\tfrac{1}{2}A+1,\dotsc,\tfrac{1}{2}A\}\) of cardinality \(A+1\).
Fix some \(i>n\) and suppose that the map \(\varphi:S^{n+1}\to \R^n\) given by
\[(x_1,\dotsc,x_r,y)\mapsto \sum_{j=1}^n x_j v_j + yv_i\]
is injective. Any two images differ by a lattice vector and thus have distance at least \(\lambda\).
If we draw a sphere of radius \(\lambda/2\) around each image, which are pair-wise disjoint, we obtain a body \(X\) with volume \((A+1)^{n+1} (\lambda/2)^n V\), where \(V\) is the volume of the unit sphere in \(\R^n\). 
Note that each image is bounded in length by \(\tfrac{1}{2}(n+1)AB\), so
\(X\) is contained in the sphere of radius \(\tfrac{1}{2}(n+1)AB+\tfrac{1}{2}\lambda\) around the origin. 
If we compare the volumes, we obtain
\[ (A+1)^{n+1} (\tfrac{1}{2}\lambda)^n V \leq (\tfrac{1}{2}(n+1)AB+\tfrac{1}{2}\lambda)^n V.\]
Rearranging the inequality gives
\[A+1 \leq \bigg(\frac{(n+1)AB+\lambda}{(A+1)\lambda}\bigg)^n\leq \bigg(\frac{(n+1)AB}{A\lambda}\bigg)^n = \bigg(\frac{(n+1)B}{\lambda}\bigg)^n,\]
which contradicts the definition of \(A\). Hence there exist distinct \(s,s'\in S^{n+1}\) with the same image under \(\varphi\). If we consider their difference, we obtain \(x_{i1},\dotsc,x_{in},y_i\in\Z\) with \(y_i\neq 0\) and \(|y_i|\leq A\) and \(|x_{ij}|\leq A\) for all \(1\leq j\leq n\) so that
\[ \sum_{j=1}^n x_{ij} v_j = y_i v_i.\]
As we have this for all \(i>n\) we may consider the vectors
\begin{align*}
\begin{array}{llllllllll}
u_1=(x_{n+1,1}, &x_{n+1,2}, &\dotsc, &x_{n+1,n},& -y_{n+1},& 0,& \dotsc,& 0), \\
u_2=(x_{n+2,1},& x_{n+2,2},& \dotsc,& x_{n+2,n},& 0,& -y_{n+2},& \dotsc,& 0), \\
\ \vdots  \\
u_k=(x_{s,1},& x_{s,2},& \dotsc,& x_{s,n},& 0,& 0,& \dotsc,& -y_{s}). \\
\end{array}
\end{align*}
Clearly they are linearly independent elements of \(M\), and each satisfies \(\|u_i\|^2\leq (n+1) A^2\).
\end{proof}

\begin{lemma}\label{lem:Ge8}
Let \(\lambda=\sqrt{\lambda_1(\Lambda)}\). 
Let \(\omega,t\in\Z_{\geq 1}\) be such that
\[t \geq (ns+1)\cdot\frac{\sqrt{\omega}}{\lambda}\]
If \(x\in\Z^s\) satisfies \(q(x)\leq \omega\), then \(x\in M\).
\end{lemma}
\begin{proof}
Suppose \(q(x)\leq\omega\) and thus also \(\|x\|^2\leq \omega\). If \(x\not\in M\), then
\begin{align*}
\sqrt{q(x)} 
&> \Big\|\sum_{i=1}^s x_i w_i\Big\| \cdot t
\geq \Big( \Big\|\sum_{i=1}^s x_i v_i\Big\| - \Big\|\sum_{i=1}^s x_i(w_i-v_i)\Big\| \Big) \cdot t \\
&\geq \Big(\lambda - \|x\|\cdot \sum_{i=1}^s\|w_i-v_i\| \Big) \cdot t
\geq \Big(\lambda - \sqrt{\omega} \cdot \big(snt^{-1}\big) \Big) \cdot t \\
&= t\lambda - sn\sqrt{\omega} 
\geq \sqrt{\omega},
\end{align*}
which is a contradiction.
\end{proof}

\begin{theorem}[Ge \cite{Ge}]\label{thm:approximate_lattice}
There exists a polynomial-time algorithm that, given \(t,\omega\in\Z_{\geq1}\) and \(t\)-approximations \(w_1,\dotsc,w_s\) of vectors \(v_1,\dotsc,v_s\in\R^n\) that generate a non-zero lattice \(\Lambda=\langle v_1,\dotsc,v_s\rangle\) with 
\[\omega\geq 2^{2n+s+1}n^{2n+2} s\bigg(\frac{B}{\lambda} \bigg)^{2n} \quad\text{and}\quad t \geq 2ns \frac{\sqrt{\omega}}{\lambda},\]
where \(B=\max_i \|v_i\|\) and \(\lambda=\sqrt{\lambda_1(\Lambda)}\), computes
\[M=\Big\{(x_1,\dotsc,x_s)\in\Z^s : \sum_{i=1}^s x_i v_i = 0\Big\}.\]
\end{theorem}

\begin{proof}
It suffices to show that if \(b_1,\dotsc,b_s\) is a 2-reduced basis of \(\Z^s\), then \(\{b_i : q(b_i)\leq \omega\}\) is a basis for \(M\), since we then can compute this basis with the LLL-algorithm.
Let \(k\) be the rank of \(M\) and let \(u_1,\dotsc,u_k\) be as in Proposition~\ref{prop:Ge6}. Then
\[q(b_i)\leq 2^{s-1} \max\{q(u_1),\dotsc,q(u_k)\}\leq 2^{s}ns  \max\{\|u_1\|^2,\dotsc,\|u_k\|^2\} \leq \omega\]
for all \(1\leq i \leq k\), where the first inequality is Exercise~\ref{ex:Ge10}, the second is Exercise~\ref{ex:Geq} and the third is Proposition~\ref{prop:Ge6}.
Thus \(b_i\in M\) by Lemma~\ref{lem:Ge8}.
Hence \(\langle b_1,\dotsc,b_k\rangle\subseteq M\) and both have the same rank, so by Exercise~\ref{ex:base_extension} they are equal.
Then by Lemma~\ref{lem:Ge8} we cannot have \(q(b_i)\leq \omega\) for \(i>k\), so indeed \(\{b_i : q(b_i)\leq \omega\}\) is a basis.
\end{proof}

\section{Fractional ideals}\label{sec:fractional}

In this section we move from finitely generated abelian groups to orders (see Definition~\ref{def:intro_nf}).
Orders in general do not have a unique prime factorization theorem like the integers do.
The best we can hope for is unique prime {\em ideal} factorization.
This observation is at the heart of the study of ideals in number rings (see Section~\ref{sec:numberring}), where we treat the ideals as analogues to elements of the number ring.
In this section we will define fractional ideals, which in this analogy correspond to the elements of the ambient number field.
For an ideal to factor into primes, we require the prime divisors to be {\em invertible}, and we will proceed to describe an algorithm to `make ideals invertible'.
As in Section~\ref{sec:coprime}, prime factorization of ideals is algorithmically infeasible, so we generalize the coprime basis algorithm to ideals.

\begin{definition}\label{def:ideal_quotient}
Let \(S\) be a commutative ring. For (additive) subgroups \(I,J\subseteq S\) we write 
\begin{align*}
I+J &= \{ x+y : x\in I,\, y\in J \}\\
I\cdot J &= IJ = \Big\{ \sum_{i=1}^n x_i y_i : n\geq 0,\, x_1,\dotsc,x_n \in I,\, y_1,\dotsc,y_n\in J \Big\} \\
I:J &= (I:J)_S = \big\{ x\in S : x J \subseteq I \big\}.
\end{align*}
We also inductively define \(I^0=(I:I)_S\) and \(I^{n+1}=I^n\cdot I\) for \(n\geq 0\). 
\end{definition}

Note that if \(I,J\subseteq S\) are subgroups, so are \(I+J\), \(I\cdot J\) and \(I:J\). 

\begin{example}
For \(S=\Q\) the finitely generated subgroups are those of the form \(x\Z\) for \(x\in\Q\) (Exercise~\ref{ex:number_field_noetherian}.a). 
For \(x,y\in\Q\) we have \((x\Z):(y\Z)=(x/y)\Z\) if \(y\neq 0\) and \((x\Z):(y\Z)=\Q\) otherwise.
\end{example}

\begin{lemma}\label{lem:quotient_facts}
Let \(S\subseteq A\) be commutative (sub)rings and \(H,I,J\subseteq S\) subgroups. 
Then:
\begin{enumerate}[label=\textup{(\roman*)}]
\item \(I:J\subseteq (HI):(HJ)\);
\item \(I^0=I:I\) is a subring of \(S\), the \emph{multiplier ring} of \(I\); 
\item \(I^1=I\);
\item \(I:I=I\) if and only if \(I\) is a subring of \(S\);
\item if \(I\) is finitely generated and \(J\cap S^* \neq \emptyset\), then \(I:J\) is finitely generated; 
\item if \(J\cap S^* \neq\emptyset\), then \((I:J)_A=(I:J)_S\).
\end{enumerate}
\end{lemma}
\begin{proof}
Both (i) and (ii) follow directly from the definition. 

(iii) Since \(1\in I:I\) by (ii), we have \(I\subseteq (I:I) \cdot I\), while \((I:I)\cdot I\subseteq I\) by definition of division. Thus \(I=(I:I)\cdot I=I^1\).

(iv) The implication (\(\Rightarrow\)) follows from (ii). 
For (\(\Leftarrow\)), suppose \(I\subseteq S\) is a subring. 
Then \(I\cdot I \subseteq I\) and consequently \(I\subseteq I:I\). 
For \(x\in I:I\) we have \(xI\subseteq I\), so in particular \(x=x\cdot 1 \in  I\) and thus \(I:I\subseteq I\).

(v) Let \(j\in J\) be invertible. Then \(j^{-1}I\) is finitely generated and hence a Noetherian \(\Z\)-module. Thus \(I:J\subseteq j^{-1} I\) is finitely generated.

(vi) Clearly \((I:J)_S\subseteq (I:J)_A\) because \(S\subseteq A\). Choose \(j\in J\cap S^*\). If \(x\in (I:J)_A\), then \(x\in j^{-1} I \subseteq S\), so \(x\in (I:J)_S\).
\end{proof}

\begin{definition}\label{def:fractional_ideal}
Let \(R\subseteq S\) be commutative (sub)rings. 
An {\em invertible ideal of \(R\) within \(S\)} is an \(R\)-submodule \(I\) of \(S\) for which there exists an \(R\)-submodule \(J\) of \(S\) such that \(IJ=R\). 
We write \(\mathcal{I}_S(R)\) for the group of invertible ideals of \(R\) within \(S\). 
For an invertible ideal \(I\) of \(R\) within \(S\) we write \(I^{-1}=(R:I)_S\). 
We write \(\textup{Q}(R)\) for the {\em total ring of fractions} of \(R\), which is obtained from \(R\) by localizing the set of regular elements (Definition~\ref{def:basic_ring}).
A {\em fractional ideal} of \(R\) is a finitely generated \(R\)-submodule \(I\) of \(\textup{Q}(R)\) such that \(I\textup{Q}(R)=\textup{Q}(R)\).
An {\em invertible ideal} of \(R\) is an invertible ideal of \(R\) within \(\textup{Q}(R)\), and we write \(\mathcal{I}(R)=\mathcal{I}_{\textup{Q}(R)}(R)\).
If we want to stress that something is an ideal, instead of a fractional ideal, we will call it an {\em integral} ideal.
\end{definition}

In most applications \(\textup{Q}(R)\) will take the role of \(S\) in Definition~\ref{def:ideal_quotient}, in which case we often write \(I:J\) for \((I:J)_S\), and \(\textup{Q}(R)\) will be a field.

\begin{exercise}\label{ex:inv_fingen}
Let \(R\subseteq S\) be commutative (sub)rings and let \(I\in\mathcal{I}_S(R)\).
Show that \(I\) is a finitely generated \(R\)-module and that \(I^{-1}\) is the unique \(R\)-submodule of \(S\) such that \(II^{-1}=R\).
\end{exercise}

\begin{exercise}\label{ex:quotient_with_invertible}
Let \(R\subseteq S\) be commutative (sub)rings and let \(I,J\subseteq S\) be \(R\)-submodules and \(H\) an invertible \(R\) ideal within \(S\). Show that \((HI):J=H(I:J)\), \(I:(HJ)=H^{-1} (I:J)\) and \(I:J=HI:HJ\).
\end{exercise}

\begin{exercise}\label{ex:field_in_ideal_quotient}
Let \(K\subseteq L\) be a finite degree field extension and let \(R\subseteq K\) be a subring with \(\textup{Q}(R)=K\). 
Show for all \(R\)-submodules \(I,J\subseteq L\) that \((KI:KJ)_L\supseteq K\cdot (I:J)_L\), with equality when \(J\) is finitely generated. Give an example where equality does not hold.
\end{exercise}

\begin{exercise}
Let \(S\) be a commutative ring and let \(R\subseteq S\) be a subring and \(H,I,J\subseteq S\) be \(R\)-submodules such that \(H+I=J\). Does it follow that \(H:J + I:J=R\)? Which inclusions hold? What if \(S\) is a number field? 
\end{exercise}

\subsection{Number rings}\label{sec:numberring}

Recall that a \emph{number field} is a field \(K\) containing the field of rational numbers \(\Q\) such that the dimension of \(K\) over \(\Q\) as a vector space is finite.
We say an additive subgroup \(I\) of a number field \(K\) is of \emph{full rank} if \(\Q I = K\).
A \emph{number ring} is a ring isomorphic to a subring of a number field.
An \emph{order} is a domain whose additive group is isomorphic to \(\Z^n\) for some \(n\in\Z_{>0}\).
Every order \(R\) is a number ring of full rank in the number field \(\textup{Q}(R)\cong \Q\tensor_\Z R\). 
Conversely, every number field comes with a natural order.

\begin{theorem}[Theorem~I.1 and Theorem~I.2 in \cite{AlgebraicNumberTheory}]\label{thm:has_maximal_order}
For every number field \(K\) there exists a unique order \(\mathcal{O}_K\) in \(K\) which is maximal with respect to inclusion. 
It is also the unique full-rank order for which every fractional ideal is invertible. \qed
\end{theorem}

\begin{exercise}\label{ex:number_field_noetherian} 
Let \(n\in\Z_{\geq 0}\) and let \(R\) be a subring of a number field of degree \(n\). 
Show that:
\begin{enumex}
\item each finitely generated subgroup \(G\subseteq \Q^n\) can be generated by \(n\) elements;
\item for all \(a\in\Z_{>0}\) and subgroups \(H\subseteq \Q^n\) we have \(\#(H/aH)\leq a^n\);
\end{enumex}
\emph{Hint:} Assume first that \(H\) is finitely generated.
\begin{enumex}[resume]
\item for all non-zero ideals \(I\subseteq R\) the ring \(R/I\) is finite (Theorem~2.11 in \cite{stevenhage_lecture_notes});
\item the ring \(R\) is Noetherian (Corollary~2.12 in \cite{stevenhage_lecture_notes}); 
\item every non-zero prime ideal of \(R\) is maximal (Corollary~2.12 in \cite{stevenhage_lecture_notes}).
\end{enumex}
\end{exercise}

\begin{exercise}
Let \(R\) be an order. Show that every non-zero integral ideal of \(R\) is a fractional ideal. Does this hold for arbitrary commutative rings \(R\)?
\end{exercise}

\begin{exercise}
Let $R$ be an order. Show that the following are equivalent:
\begin{enumerate}[nosep]
\item \(R\) is equal to the maximal order of \(\textup{Q}(R)\);
\item the sum of every two invertible ideals is again invertible;
\item every maximal ideal of \(R\) is invertible.
\end{enumerate}
\end{exercise}

\begin{exercise}
Let \(R\) be a number ring and let \(\textup{Q}(R)\subseteq L\) be number fields. Show that for every \(I\in\mathcal{I}_L(R)\) there exist  \(J\in\mathcal{I}(R)\) and \(\alpha\in L^*\) such that \(I=\alpha J\). 
\end{exercise}

\begin{exercise}
Give an order \(R\) and an ideal \(I\) of \(R\) such that \(I(R:I)\neq I:I\). 
\end{exercise}

\begin{exercise}
Let \(K\) be a number field and \(I\subseteq K\) be a subgroup. Then \(I^2 = I\) if and only if \(I\) is a ring. \\
\emph{Hint:} If \(I^2=I\), then \(I\) is a finitely generated ideal of the ring \(\Z+I\).
\end{exercise}

In our algorithms we encode an order by first specifying its rank \(n\), and then writing down its \(n\times n\times n\) multiplication table on the standard basis vectors.
By distributivity this completely and uniquely defines a multiplication on the order.
However, not every element of \(\Z^{n\times n\times n}\) encodes a ring structure, because it can fail for example associativity of multiplication (Exercise~\ref{ex:is_order}).
We encode a number field \(K\) simply by an order \(R\) such that \(\textup{Q}(R)=K\).
We encode a finitely generated subgroup of \(K\), in particular a fractional ideal of \(R\), as a finitely generated subgroup of \(R\) together with a denominator.
As we will see in Theorem~\ref{thm:equivalent_problems} it is difficult to compute the maximal order of a number field.

\begin{exercise} Let \(\Q^\text{alg}\) be some algebraic closure of \(\Q\). 
\begin{enumex}
\item Show that \(\#\Q^\text{alg}=\#\N\). 
\item Show that there are precisely \(\#\N\) number fields contained in \(\Q^\text{alg}\). 
\item Show that there are precisely \(\#\R\) subfields of \(\Q^\text{alg}\).
\item The same question as (b) and (c), but counting the fields up to isomorphism.
\item Do there exist \(\#\R\) subfields of \(\Q^\text{alg}\) that are pairwise isomorphic?
\item Let \(K\neq \Q\) be a number field. Show that there are precisely \(\#\N\) orders and \(\#\R\) subrings in \(K\) with field of fractions \(K\). What holds for \(K=\Q\)?
\item Argue why it is natural to restrict the input of our algorithms to orders and number fields as opposed to general number rings.
\end{enumex}
\end{exercise}

\begin{exercise}[Proposition 5.8 in \cite{Buchman-Lenstra}]
Show that there exist polynomial-time algorithms that, given a number field \(K\) and finitely generated subgroups \(I,J\subseteq K\), compute \(I+J\), \(I\cap J\), \(I\cdot J\) and \(I:J\).
\end{exercise}

\begin{exercise}\label{ex:is_order}
You may use without proof that it is possible to factor polynomials in \(\Q[X]\) in polynomial time using the LLL-algorithm.
Show that there exists a polynomial-time algorithm that, given a cube integer matrix, determines whether it encodes an order.
\end{exercise}

Let \(A\subseteq B\) be commutative rings such that \(B\) is free as an \(A\)-module with \(\rk_A(B)=n<\infty\) (see Definition~\ref{def:rk}).
Then the ring \(\End_A(B)\) of \(A\)-module endomorphisms of \(B\) is isomorphic to \(\Mat_n(A)\), the ring of \(n\times n\)-matrices with coefficients in \(A\).
This isomorphism induces a determinant \(\det:\End_A(B)\to A\) and trace \(\Tr:\End_A(B)\to A\) on \(\End_A(B)\) which are multiplicative and \(A\)-linear respectively.
We have a natural map \(B\to\End_A(B)\) given by \(b\mapsto( x\mapsto bx )\). 
It induces the \(A\)-module homomorphism \(\Tr_{B/A}:B\to A\) when composed with \(\Tr\), which we call the \emph{trace of \(B\) over \(A\)}.
For more on traces and determinants, see Section~7 of \cite{Stevenhagen}.

\begin{exercise} \label{ex:unique_trace_det}
Show that $\Tr: \End_A(B) \to A$ and $\det : \End_A(B)\to A$ are independent of the chosen $A$-basis for $B$.
\end{exercise}

We define the \emph{discriminant} of \(B\) over \(A\) as \(\Delta_{B/A}=\det(M)\in A\), where \(M\) is the matrix \((\Tr_{B/A}(e_ie_j))_{1\leq i,j\leq n}\) and \((e_1,\dotsc,e_n)\) is a basis for \(B\) over \(A\). 
This definition depends on a choice of basis for \(B\) (see Exercise~\ref{ex:unique_discriminant}).
However, for \(A=\Z\) the discriminant is uniquely determined, and we write \(\Delta(B)=\Delta_{B/\Z}\) in this case.

\begin{exercise}\label{ex:unique_discriminant}
Prove that \(\Delta_{B/A}\) is defined up to multiples by \((A^*)^2\). Conclude for \(A=\Z\) that \(\Delta_{B/A}\) is independent of choice of basis.
\end{exercise}

\begin{exercise}[7-3 in \cite{Stevenhagen}]\label{ex:discriminant_subring}
If \(C\subseteq B\) is of finite index and both are free of finite rank over \(\Z\), then \(\Delta(C)=\#(B/C)^2\cdot \Delta(B)\).
\end{exercise}

\begin{exercise}\label{ex:small_index}
Show that there exists a polynomial-time algorithm that, given an order \(R\), computes \(\Delta(R)\). Conclude that the lengths of \(\Delta(R)\) and \(\#(\mathcal{O}_{\textup{Q}(R)}/R)\) are polynomially bounded in terms of the representation of the order \(R\).
\end{exercise}

\begin{exercise}\label{ex:radical_discriminant}
Suppose \(d\in\Z\) is non-zero.
\begin{enumex}
\item Show that we can write \(d=\square(d)^2 \cdot \boxtimes(d)\) uniquely with \(\square(d)\in\Z_{>0}\) and \(\boxtimes(d)\) square-free, the {\em square part} and {\em square-free part} of \(d\) respectively.
\item Show that if \(d\) is not a square, then \(\Delta(\Z[\sqrt{d}])=4d\).
\end{enumex}
Suppose that \(d\) is square-free and let \(K=\Q(\sqrt{d})\).
\begin{enumex}[resume]
\item Show that if \(d\equiv 1 \bmod 4\), then \(\mathcal{O}_K=\Z[(\sqrt{d}+1)/2]\).
\item Show that if \(d\equiv 2,3\bmod 4\), then  \(\mathcal{O}_K=\Z[\sqrt{d}]\).
\item Conclude that \(\Delta(\mathcal{O}_K) / \boxtimes(d) \in\{1,4\}\).
\end{enumex}
\end{exercise}

\begin{exercise} 
Let \(n, m \in \Z_{>0}\). 
Prove that there exists a number field \(K\) such that \([K : \Q] = n\) and \(\gcd(\Delta_K, m) = 1\).
\end{exercise}

\begin{exercise}\label{ex:trace_iff_det}
Let $A \subseteq B$ be commutative rings such that $B$ is free of finite rank over $A$. 
Show that the map $B \to \Hom_A(B,A)$ given by $x \mapsto (y \mapsto \Tr_{B/A}(xy))$ is an isomorphism of $B$-modules if and only $\Delta(B/A)$ is a unit in $A$.
\end{exercise}

\subsection{Localization}

We say a commutative ring is {\em local} if it has a unique maximal ideal, and {\em semi-local} if it has only finitely many maximal ideals.
For a commutative ring \(R\) and a prime ideal \(\fp\subseteq R\), the {\em localization of \(R\) at \(\fp\)}, written \(R_\fp\), is the ring obtained from \(R\) by inverting all elements in \(R\setminus \fp\) (see p.~38 of \cite{Atiyah}).
If \(M\) is an \(R\)-module, then we call \(M_\fp=R_\fp \tensor_R M\) the localization of \(M\) at \(\fp\).
We will show that an ideal of a number ring is invertible if and only if it is locally principal (Corollary~\ref{cor:locally_principal}). 

\begin{example}
Let \(\mathfrak{p}=2\Z\subseteq\Z\). 
We have \(\Z_{\mathfrak{p}}=\{a/b: a,b\in\Z,\, b \textup{ odd}\}\) and \(\mathfrak{p}\Z_\mathfrak{p}\) is its unique maximal ideal. Let \(R\) be an order. Note that \(R\) is a \(\Z\)-module, so we obtain a ring \(R_\fp=\Z_\fp\tensor_\Z R\). 
Then all maximal ideals of \(R_\fp\) contain \(\fp R_\fp\) and thus correspond to the maximal ideals of \(R_\fp/\fp R_\fp \cong R/2 R\), which is finite. Hence \(R_\fp\) is semi-local.
\end{example}

\begin{exercise}\label{ex:units_local_ring}
Let $R$ be a commutative ring. 
Show that $R$ is a local ring if and only if $R \setminus R^*$ is an additive subgroup of $R$, in which case \(R\setminus R^*\) is the maximal ideal.
\end{exercise}

\begin{exercise}\label{ex:hereditary_semilocal}
Let \(R\subseteq S\) be number rings. Show that if \(R\) is semi-local, then \(S\) is semi-local. 
\end{exercise}

\begin{proposition}\label{prop:semi_local_ring}
Let \(R\subseteq S\) be commutative (sub)rings such that \(R\) is semi-local and let \(I\subseteq S\) be an \(R\)-submodule. 
Then \(I\) is invertible within \(S\) if and only if there exists some \(x\in S^*\) such that \(I=xR\).
\end{proposition}
\begin{proof}
(\(\Leftarrow\)) We may take \(J=x^{-1}R\).

(\(\Rightarrow\))
Let \(\fm\subseteq R\) be maximal.
As \(IJ=R\not\subseteq \fm\) there exist \(x_\fm\in I\) and \(y_\fm\in J\) such that \(x_\fm y_\fm \in R\setminus \fm\). 
Since \(R\) is semi-local, there exist \(\lambda_\fm\in R\) for each maximal \(\fm\subseteq R\) such that for all maximal \(\fn\subseteq R\) we have \(\lambda_\fm \in \fn\) if and only if \(\fm\neq\fn\): 
Namely, let \(r_{\fn}\in \fn\setminus \fm\), which exist since the ideals are maximal, and take \(\lambda_\fm = \prod_{\fn\neq \fm} r_{\fn}\).
Consider \(x=\sum_\fm \lambda_\fm x_\fm\) and \(y=\sum_\fm \lambda_\fm y_\fm\). Then 
\[xy=\sum_{\fm,\fn} \lambda_\fm \lambda_{\fn} x_\fm y_{\fn}.\]
For all \(\fm\) there is precisely one term not contained in \(\fm\), which is \(\lambda_\fm \lambda_\fm x_\fm y_{\fm}\), hence \(xy\not\in \fm\).
Since \(x_\fm y_{\fn}\in IJ=R\) we have \(xy\in R\).
It follows that \(xy\) is a unit of \(R\), and \(x\) is a unit of \(S\). 
Finally, \(xR \subseteq I = xy I \subseteq x JI = xR\) and \(I=xR\).
\end{proof}

Recall the definition of \(\mathcal{I}(R)\) from Definition~\ref{def:fractional_ideal}.

\begin{corollary}\label{cor:invertible_ideals_local}
Let \(R\) be a is semi-local commutative ring. Then \(\mathcal{I}(R)\cong \textup{Q}(R)^*/R^*\). \qed
\end{corollary}

\begin{lemma}\label{lem:intersect_ideals}
Let \(R\) be a subring of a field \(K\). 
Then for every localization \(S\) of \(R\) the natural map \(S\to K\) is injective, and \(I=\bigcap_{\textup{max. }\fm\subset R} I_\fm\) for every \(R\)-submodule \(I\subseteq K\).
\end{lemma}
\begin{proof}
The map \(S\to K\) is injective since \(R\), and hence \(S\), is a domain. 
Clearly \(I\subseteq I_\fm\). 
Suppose \(x\in \bigcap_\fm I_\fm\). Consider \(J=(I:xR)_K \cap R\) and let \(\fm\subseteq R\) be maximal. 
Since \(x\in I_\fm\) there exists some \(r\in R\setminus\fm\) such that \(rx\in I\). 
Then \(r\in J\) by definition of \(J\), so \(J\not\subseteq\fm\).
As this holds for all \(\fm\) we conclude that \(J=R\).
It follows that \(x\in I\) and hence \(I=\bigcap_\fm I_\fm\).
\end{proof}

\begin{exercise}\label{ex:ideal_right_inverse}
Let \(R\) be a number ring and \(\fp\subset R\) be a maximal ideal. 
Show that if \(I_\fp\subseteq R_\fp\) is an invertible ideal over \(R_\fp\), then \(J=I_\fp\cap R\) is an invertible ideal over \(R\) such that \(J_\fp=I_\fp\), and \(J_\fq=R_\fq\) for all maximal \(\fq\subset R\) not equal to \(\fp\).
\end{exercise}

\begin{theorem}[cf.\ Theorem~5.3 in \cite{Stevenhagen}]\label{thm:localization_inverse}
Let \(R\subseteq S\) be number rings such that \(\textup{Q}(R)=\textup{Q}(S)\). 
Then we have a group isomorphism \(\mathcal{I}(S)\to \bigoplus_{\textup{max.\ }\fp\subset R} \mathcal{I}(S_\fp)\) given by \(I\mapsto (I_\fp)_{\fp}\), with inverse \((I_\fp)_\fp \mapsto \bigcap_\fp I_\fp\).
\end{theorem}
\begin{proof}
For any maximal \(\fq\subset R\) let 
\[\varphi_\fq : \mathcal{I}(S_\fq) \to \bigoplus_{\fm\subset S_\fq} \mathcal{I}((S_\fq)_\fm) = \bigoplus_{\fq \subseteq\fm \subset S} \mathcal{I}(S_\fm)\] 
be the map as in the theorem with \(R\) replaced by \(R_\fq\) and \(S\) replaced by \(S_\fq\).
Then the following diagram is commutative:
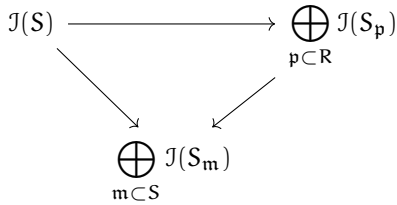
\begin{center}
\begin{tikzcd}[column sep=small]
\mathcal{I}(S) \ar{rr} \ar{dr} & & \displaystyle\bigoplus_{\fp\subset R} \mathcal{I}(S_\fp) \ar{dl}\\
& \displaystyle\bigoplus_{\fm\subset S} \mathcal{I}(S_\fm) &
\end{tikzcd}
\end{center}
Hence to show that the horizontal map is an isomorphism, it suffices to prove this for the downward maps. 
These maps are both injective with the appropriate left inverse by Lemma~\ref{lem:intersect_ideals}, while they are surjective by Exercise~\ref{ex:ideal_right_inverse}.
\end{proof}

\begin{corollary}\label{cor:locally_principal}
Let \(R\) be a subring of a number field. Then a fractional \(R\)-ideal is invertible if and only if it is locally principal.
\end{corollary}
\begin{proof}
By Proposition~\ref{prop:semi_local_ring} locally principal is equivalent to locally invertible, which in turn is equivalent to invertible by Theorem~\ref{thm:localization_inverse}.
\end{proof}

\begin{lemma}\label{lem:inv_ideal_surj}
Let \(R\subseteq S\) be number rings such that \(\textup{Q}(R)=\textup{Q}(S)\).
Then there is a surjective group homomorphism \(\sigma:\mathcal{I}(R)\to \mathcal{I}(S)\) given by \(I\mapsto SI\).
\end{lemma}
\begin{proof}
The only non-trivial part is surjectivity. 
Using Theorem~\ref{thm:localization_inverse} one verifies that the induced map \(\sigma:\bigoplus_{\fp\subseteq R}\mathcal{I}(R_\fp) \to \bigoplus_{\fp\subseteq R} \mathcal{I}(S_\fp) \) is the direct sum of maps \(\sigma_\fp :\mathcal{I}(R_\fp) \to\mathcal{I}(S_\fp)\) given by \(I\mapsto SI\).
Thus it suffices to prove the theorem when \(R\) is local, and hence \(S\) is semi-local by Exercise~\ref{ex:hereditary_semilocal}. 
Then the invertible ideals of \(S\) are principal by Proposition~\ref{prop:semi_local_ring}, and we may use its generator to generate an invertible \(R\)-ideal.
\end{proof}

\begin{exercise}
Let \(R\) be a number ring. Show that every fractional ideal of \(R\) is invertible if and only if every fractional ideal of \(R\) generated by at most two elements is invertible.
\end{exercise}

\subsection{Blowing up}

In this section we give an algorithm to make ideals invertible, i.e.\ given a fractional ideal \(I\) of some order \(R\) of a number field \(K\), computing the minimal order \(R\subseteq S\subseteq K\) such that \(SI\) is invertible. Borrowing terminology from algebraic geometry (cf.\ Proposition 7.14 in \cite{Hartshorne}), we will be computing a {\em blowup} of \(R\) at the ideal \(I\).
We give a variant on the proof of \cite{Dade-Zassenhaus}.
We warn those reading \cite{Dade-Zassenhaus} that it contains non-standard definitions of fractional and invertible ideals.

\begin{proposition}\label{prop:blowup}
Let \(S\) be a commutative ring and \(I\subseteq S\) a subgroup. Then:
\begin{enumerate}
\item the \emph{blowup} at \(I\),
\[ \textup{Bl}(I) = \textup{Bl}_S(I) := \bigcup_{k\geq 0} (I^k:I^k)_S, \]
is a subring of \(S\);
\item if \(R\subseteq S\) is a subring and \(RI\) is an invertible ideal of \(R\) within \(S\), then \(\textup{Bl}(I)\subseteq R\);
\item if \(RI\) is an invertible ideal of \(R=I^n:I^n\) within \(S\) for some \(n\geq 0\), then \(R=\textup{Bl}(I)\);
\item if \(S\) is a number field and \(I\) is finitely generated and non-zero, then there exists some \(n\geq 0\) such that \(\textup{Bl}(I)=I^n:I^n\), and \(\textup{Bl}(I)\) is an order.
\end{enumerate}
\end{proposition}
\begin{proof}
(1) It follows from Lemma~\ref{lem:quotient_facts} for all \(k\geq 0\) that \(I^k:I^k\) is a ring and that \(I^k:I^k\subseteq I^{k+1}:I^{k+1}\). Hence \(\textup{Bl}(I)\), being their union, is a ring. 

(2) Note that \(RI^k=(RI)^k\) is invertible within \(S\) for all \(k\geq 0\). By Lemma~\ref{lem:quotient_facts} and Exercise~\ref{ex:quotient_with_invertible} we have \(I^k:I^k\subseteq (RI^k):(RI^k)=R:R=R\) for all \(k\), so \(\textup{Bl}(I)\subseteq R\).

(3) This follows from (2) and the fact that \(R\subseteq\textup{Bl}(I)\).

(4) Note that \(I^k:I^k\) is both a subring and a finitely generated subgroup of \(S\) by Lemma~\ref{lem:quotient_facts}, hence it is an order.
Because \(I^k:I^k\subseteq\mathcal{O}_S\), see Theorem~\ref{thm:has_maximal_order}, and \(\mathcal{O}_S\) is a Noetherian \(\Z\)-module, there must exist some \(n\geq 0\) such that \(I^n:I^n=I^k:I^k\) for all \(k\geq n\). Consequently \(I^n:I^n=\textup{Bl}(I)\).
\end{proof}

\begin{example}\label{ex:blowup_kummer}
Let \(p\) be a prime and \(n\in\Z_{>1}\). Let \(\pi=\sqrt[n]{p}\) and consider the orders \(R=\Z[\pi^k : k\in\Z_{\geq n}]\subseteq \Z[\pi]=S\) and fractional ideal \(I=R+R\pi\) of \(R\). Then \(I^{k}=R\pi^0+R\pi^1+\dotsm+R\pi^k\) for all \(k\geq 1\), and \(I^{n-1}=S\). Hence \(I^{n-1}:I^{n-1}=S\) by Lemma~\ref{lem:quotient_facts}.iv. We have \(I\subseteq S\) and \(1\in I\), so \(SI=S\). In particular, \(SI\) is invertible over \(S\), so \(\textup{Bl}(I)\subseteq S\) by Proposition~\ref{prop:blowup}.2. We conclude that \(\textup{Bl}(I)=I^{n-1}:I^{n-1}=S\).
Note that \(S\) need not be a maximal order; For \(n=2\) and \(p\equiv 1\ (\textup{mod }4)\) we have \(S\subsetneq \Z[\tfrac{1+\pi}{2}]\).
\end{example}

We will show later in Theorem~\ref{thm:blowup}, in the case that \(I\) is a finitely generated non-zero subgroup of a number field \(K\), that \(\textup{Bl}_K(I) I\) is invertible over \(\textup{Bl}_K(I)\). Hence in this case \(\textup{Bl}_K(I)\) is the unique minimal subring of \(K\) over which \(I\) becomes invertible. Moreover, the proof lends itself to an algorithm to compute the blowup, which is a non-trivial statement by Exercise~\ref{ex:bad_compute_blowup}.

\begin{example}
Let \(S=\Q[X]\) and \(I=\Z+\Z X\). 
Then \(I^n\subseteq\Z[X]\) is the set of polynomials of degree at most \(n\). 
Every non-zero element of \(I^n:I^n\) needs to have degree \(0\), from which we deduce that \(I^n:I^n=\Z\).
In particular, \(\textup{Bl}(I)=\Z\) and \(\Z I\) is not an invertible \(\Z\)-ideal.
\end{example}

\begin{exercise}
Let \(S\) be a commutative ring and \(I,J\subseteq S\) subgroups. Show that \(\textup{Bl}(I)\cdot\textup{Bl}(J)\subseteq \textup{Bl}(IJ)\). Show that \(\textup{Bl}(I)=I\) if and only if \(I\) is a subring of \(S\).
\end{exercise}

\begin{exercise}\label{ex:blowup_of_invertible}
Let \(R\subseteq S\) be a commutative subrings and and \(I,J\subseteq S\) fractional \(R\)-ideals. 
Show that if \(I\) is invertible, then \(\textup{Bl}(IJ)=\textup{Bl}(J)\).
\end{exercise}

\begin{exercise}
Let \(S\) be a commutative ring, \(I\subseteq S\) a finitely generated non-zero subgroup and \(n\in\Z_{\geq 0}\). 
Show that if \(RI\) is invertible over \(R=I^n:I^n\), then \(R=\textup{Bl}(I)\).
\end{exercise}

\begin{exercise}\label{ex:bad_compute_blowup}
Let \(n\in\Z_{>1}\). Construct an order of rank \(n\) with a fractional ideal \(I\) such that \(R=(I^0:I^0)=\dotsm= (I^{n-2}:I^{n-2})\neq (I^{n-1}:I^{n-1})=\textup{Bl}(I)\).
\end{exercise}

We collect some facts about generators of modules. 

\begin{definition}
Let \(R\) be a commutative ring and \(M\) an \(R\)-module. We define 
\[\gamma_R(M) = \min \Big\{  \# X : X\subseteq M,\, \sum_{x\in X} Rx=M\Big\},\]
the minimal number of generators needed to generate \(M\) as an \(R\)-module.
\end{definition}

Note that \(M\) is finitely generated if and only if \(\gamma_R(M)<\infty\). We will show how this quantity behaves with respect to quotients.

\begin{lemma}[Nakayama, Corollary~2.7 in \cite{Atiyah}]\label{lem:local_nakayama}
Let \(R\) be a commutative ring, \(I\) an ideal of \(R\) contained in every maximal ideal, and \(M\) a finitely generated \(R\)-module.
\begin{enumerate}
\item If \(N\) is any submodule of \(M\) and \(M=N+IM\), then \(M=N\). 
\item We have \(\gamma_R(M)=\gamma_{R/I}(M/IM)\). \qed
\end{enumerate}
\end{lemma}

\begin{exercise}[cf.\ Corollary~5.8 in \cite{Atiyah}]\label{ex:maximal_ideals_in_finite_extensions}
Let \(f:R\to S\) be a morphism of commutative rings such that \(S\) is Noetherian as an \(R\)-module.
Show that the map 
\[\{\textup{maximal }\mathfrak{m}\subset S\}\to\{\textup{maximal }\mathfrak{m}\subset R\}\quad \textup{given by}\quad  \fm\mapsto f^{-1} \fm\] 
is well-defined and has fibers of size at most \(\gamma_R(S)\). \\
\emph{Note:} We may even replace the Noetherian assumption by \emph{integrality}.
\end{exercise}

For a local commutative ring \(R\) with unique maximal ideal \(\fm\) write \(\kappa(R)=R/\fm\) for the {\em residue field} of \(R\). 

\begin{lemma}\label{lem:min_gen}
Let \(R\) be a local commutative ring with maximal ideal \(\fp\) and \(A\) a local commutative \(R\)-algebra with maximal ideal \(A\fp\).
Let \(L\) and \(M\) be finitely generated \(R\)-modules and \(N\) a finitely generated \(A\)-module. 
Then:
\begin{enumerate}
\item \(\gamma_R(L\tensor_R M)=\gamma_R(L)\cdot\gamma_R(M)\);
\item \(\gamma_R(N)=\gamma_R(A)\cdot\gamma_A(N)\);
\item \(\gamma_A(A\tensor_R M)=\gamma_R(M)\).
\end{enumerate}
\end{lemma}
\begin{proof}
Note that \(\kappa(R)\tensor_R N = N/\fp N = N/\fp A N =\kappa(A)\tensor_R N\) and similarly \(\kappa(R)\tensor_R A=\kappa(A)\).
Applying Lemma~\ref{lem:local_nakayama}.2 and taking the tensor product with \(\kappa(R)\), we may assume that \(R\) and \(A\) are fields.
Thus all modules are free and the rank corresponds with \(\gamma\), in which case the results are trivial.
\end{proof}

\begin{lemma}\label{lem:cardinality_unit}
Let \(R\) be a local commutative ring and \(S\) a semi-local commutative \(R\)-algebra. 
Let \(M\subseteq S\) be an \(R\)-submodule such that \(M\not\subseteq \fm\) for all maximal ideals \(\fm\subseteq S\). If \(S\) has at most \(\#\kappa(R)\) maximal ideals, then \(M\cap S^*\neq \emptyset\).
\end{lemma}
\begin{proof}
Note that \(S\setminus S^* = \bigcup_{\fm} \fm\), where the union ranges over the maximal ideals of \(S\), as every non-unit is contained in some proper and hence in some maximal ideal.
Thus it suffices to show that \(\bigcup_\fm (M\cap \fm) \neq M\). 
We will inductively prove the stronger statement that an \(R\)-module \(M\) cannot be written as a union of \(k\) proper submodules for any integer \(0\leq k \leq \# \kappa(R)\). 
For \(k=0\) this is trivial. 
Now we assume that \(\bigcup_{i<k} M_i \subsetneq M\), and for the sake of contradiction that \(\bigcup_{1\leq i\leq k} M_i = M\). 
Choose \(y\in M \setminus \bigcup_{i<k} M_i\). 
In particular, \(y\in M_k\).
Now let \(x\in M\setminus M_k\). 
Then for all \(t\in R\) we have \(x+ty \in M_{i(t)}\) for some \(1\leq i(t) < k\).
Since \(k\leq \#\kappa(R)\) there exist \(s,t\in R\) such that \(s\not\equiv t\) in \(\kappa(R)\) and \(i(s)=i(t)\).
It follows (Exercise~\ref{ex:units_local_ring}) that \(s-t\in R^*\) and thus \(y=(s-t)^{-1}((x+ys)-(x+yt)) \in M_{i(t)}\), which is a contradiction.
Hence \(y\not\in M_k\) and thus \(\bigcup_{i\leq k} M_i \subsetneq M\), as was to be shown.
\end{proof}

\begin{lemma}\label{lem:large_extension}
Let \(R\) be local commutative ring with maximal ideal \(\fp\). 
Then for every \(n\in\Z_{\geq 0}\) there exists a commutative \(R\)-algebra \(A\) which is free of finite rank over \(R\), local with maximal ideal \(A\fp\), and such that \(\#\kappa(A)>n\cdot\rk_R(A)\).
\end{lemma}
\begin{proof}
If \(\kappa(R)\) is infinite we may choose \(A=R\), so assume \(\kappa(R)=\F_q\) for some prime power \(q\).
Let \(d\in\Z_{>0}\) and consider the field extension \(\kappa(R)\subseteq \F_{q^d}=\F_q(\alpha)\) for some \(\alpha\in\F_{q^d}\).
Let \(f\in R[X]\) be some monic lift of the minimal polynomial of \(\alpha\) over \(\kappa(R)\).
Clearly \(A=R[X]/(f)\) is free of rank \(d\) over \(R\). 
Since \(\fp\) is the unique maximal ideal of \(R\), the maximal ideals of \(A\) correspond to those of \(A/\fp A=\F_{q^d}\) by Exercise~\ref{ex:maximal_ideals_in_finite_extensions}, which is a field.
Hence \(A\) is local with maximal ideal \(A\fp\).
Finally, \(\#\kappa(A)=\#\kappa(R)^d > n\cdot d \) for \(d\) sufficiently large, as was to be shown.
\end{proof}

\begin{lemma}\label{lem:module_has_one}
Let \(R\subseteq S\) be commutative (sub)rings with \(R\) local. 
Let \(M\subseteq S\) be an \(R\)-submodule such that \(1\in M\), and \((M^n:M^n)_S=R\) for some integer \(n+1\geq \gamma_R\big(\bigcup_{k\geq 0} M^k \big)\). 
Then \(M=R\).
\end{lemma}
\begin{proof}
Let \(\fp\subseteq R\) be the maximal ideal.
From \(1\in M\) it follows that \(R=M^0 \subseteq M^1 \subseteq M^2 \subseteq \dotsm\).
We write \(M^\infty=\bigcup_{k\geq 0} M^k\), which is an \(R\)-subalgebra of \(S\).
With \(N_k=(M^k+\fp M^\infty)/(\fp M^\infty)\) we get a chain \(N_0\subseteq N_1\subseteq N_2\subseteq \dotsm\) of \(\kappa(R)\)-vector spaces with limit \(N_\infty=M^\infty/\fp M^\infty\).
We have that \(\dim_{\kappa(R)}(N_\infty) = \gamma_R(M^\infty)\leq n+1\) by Lemma~\ref{lem:local_nakayama}.2.
Hence there exists a minimal \(k\) such that \(N_{k}=N_{k+1}\), and it must satisfy \(k \leq \dim_{\kappa(R)}(N_\infty)-\dim_{\kappa(R)}(N_0)\leq n\).
Then \(N_{k}=N_{k+1}=N_{k+2}=\dotsm\) and thus \(N_k=N_\infty\).
Equivalently, we have \(M^k+\fp M^\infty= M^\infty\), so \(M^k=M^\infty\) by Lemma~\ref{lem:local_nakayama}.1.
It follows from Lemma~\ref{lem:quotient_facts} that \(M\subseteq M^\infty = (M^\infty : M^\infty)_S = (M^n : M^n)_S = R \subseteq M\), so indeed \(M=R\).
\end{proof}

\begin{lemma}\label{lem:tensor_division}
Let \(R\subseteq S\) be commutative subrings, let \(A\) be a commutative \(R\)-algebra that is free as an \(R\)-module, and let \(M\subseteq S\) be an \(R\)-module. Then \((M\tensor_R A : M\tensor_R A)_{S\tensor A} = (M:M)_S \tensor_R A\) as subsets of \(S\tensor A\).
\end{lemma}
\begin{proof}
The inclusion \(\supseteq\) is straight-forward. Choose a basis \(\mathcal{A}\subseteq A\) of \(A\) as an \(R\)-module. 
Let \(x\in (M\tensor_R A : M\tensor_R A)_{S\tensor A}\). 
Then we may uniquely write \(x=\sum_{a\in\mathcal{A}} x_a \tensor a\). Then \(M\tensor_R A \ni  x \cdot (m\tensor 1) = \sum_{a\in\mathcal{A}} x_a m \tensor a\), so \(x_a m \in M\) for all \(a\in\mathcal{A}\). Hence \(x_a\in (M : M)_S\) and \(x\in (M:M)_S \tensor_R A\).
\end{proof}

\begin{theorem}[Dade--Taussky--Zassenhaus \cite{Dade-Zassenhaus}]\label{thm:blowup}
Let \(K\) be a number field, \(R\subseteq K\) a subring and \(I\subseteq K\) a finitely generated non-zero \(R\)-module.
Then \(\textup{Bl}_K(I)=I^n:I^n\) for all \(n\geq [K:\Q]-1\).
For any subring \(T\subseteq K\), it holds that \(\textup{Bl}_K(I)\subseteq T\) if and only if \(TI\) is an invertible ideal of \(T\) within \(K\).
\end{theorem}

Note that we may take \(R=\Z\) in this theorem. 

\begin{proof}
Choose some \(n\geq [K:\Q]-1\). It suffices by Proposition~\ref{prop:blowup} to show that \((I^n:I^n) I \in \mathcal{I}_K(I^n:I^n)\).

We first prove this under the added assumptions that \(R=I^n:I^n\), that \(\textup{Q}(R)=K\) and that \(R\) is a local ring.
Let \(S=R\cdot\mathcal{O}_K\) and note that it is a ring and a Noetherian \(R\)-module.
In particular, \(S\) is semi-local by Exercise~\ref{ex:maximal_ideals_in_finite_extensions}.
Every fractional ideal of \(S\), in particular \(SI\), is invertible, because \(\mathcal{O}_K\subseteq S\).
Hence there exists by Lemma~\ref{lem:inv_ideal_surj} some \(J\in\mathcal{I}(R)\) such that \(S I= S J\).
Then \(I\in\mathcal{I}(R)\) if and only if \(IJ^{-1}\in\mathcal{I}(R)\), so it suffices to show that \(M=IJ^{-1}\) is invertible.
Note that \(SM=S\).
Consider now the \(R\)-algebra \(A\) from Lemma~\ref{lem:large_extension} satisfying  \(\#\{\textup{maximal }\mathfrak{m}\subset S\}\cdot \rk_R(A) <\#\kappa(A)\), and write \(-_A\) for \(-\tensor_R A\). 
Since \(A\) is free over \(R\), we have \((M^{n}_A:M^{n}_A)_{S_A}=R_A=A\) by Lemma~\ref{lem:tensor_division}. 
As \(S\) is Noetherian, \(S_A\) is free and hence Noetherian over \(S\).
If \(M_A\subseteq \mathfrak{m}\) for some maximal \(\mathfrak{m}\subset S_A\), then also \(M=M_A\cap S\) is contained in the maximal ideal of \(\mathfrak{m}\cap S\) of \(S\) by Exercise~\ref{ex:maximal_ideals_in_finite_extensions}, which is impossible since \(SM = S\). 
From Exercise~\ref{ex:maximal_ideals_in_finite_extensions} we also obtain \(\#\{\textup{maximal }\mathfrak{m}\subset S_A\} \leq \rk_R(A) \cdot \# \{\textup{maximal }\mathfrak{m}\subset S\} < \#\kappa(A)\).
Thus we may apply Lemma~\ref{lem:cardinality_unit} with \(R_A\), \(S_A\) and \(M_A\) in the roles of \(R\), \(S\) and \(M\) respectively.
Hence we may choose some \(u\in M_A\cap S_A^*\).
Now let \(N=u^{-1} M_A\) and \(N^\infty=\bigcup_{k\geq 0} N^k\). 
Since \(N^\infty\subseteq S_A\) is Noetherian as an \(R\)-module, we may choose some finite set of \(R\)-module generators \(X\subseteq N^\infty\).
By Exercise~\ref{ex:number_field_noetherian}.a and Lemma~\ref{lem:min_gen}.1 we have 
\[\gamma_R(N^\infty) = \gamma_R\Big(\sum_{x\in X} Rx\Big) \leq  \gamma_\Z\Big(\sum_{x\in X}\Z x\Big) \leq \dim_\Q K_A \leq \rk_R(A)\cdot (n+1),\]
so \(\gamma_A(N^\infty)\leq n+1\) by Lemma~\ref{lem:min_gen}.2.
Then by Lemma~\ref{lem:module_has_one} we have \(N=A\).
By Lemma~\ref{lem:min_gen}.3 we have \(1=\gamma_A(N)=\gamma_R(M)\), so \(M=Rv\) for some \(v\in S\).
In particular, \(M\) is invertible, as was to be shown.

Now consider the general case. 
Let \(S=I^n:I^n\) and \(F=\textup{Q}(S)\). 
Let \(x\in I\) be a non-zero element and consider \(J=SIx^{-1}\).
Note that \(x^n J^n=SI^n=I^n\) and thus \(J^n:J^n=x^n J^n : x^n J^n=I^n:I^n=S\).
Since \(1\in J\), we have \(J^0\subseteq J^1 \subseteq J^2\subseteq \dotsm\).
From a dimension argument it then follows that \(\Q J^n=\Q J^{n+1} =\dotsm\) and thus \(\Q J^n\) is a ring.
It follows from Lemma~\ref{lem:quotient_facts}.iv and Exercise~\ref{ex:field_in_ideal_quotient} that \(\Q J^n = (\Q J^n):(\Q J^n) = \Q \cdot (J^n:J^n) = F\), so \(J \subseteq F\).
Let \(\fp\subset S\) be a maximal ideal, so that \(S_\fp\) is local with \(S_\fp=J_\fp^n : J_\fp^n\).
Hence \(J_\fp\) is an invertible ideal of \(S_\fp\) by the previous case.
Invertibility is a local property by Theorem~\ref{thm:localization_inverse}, so \(J\) is an invertible \(S\)-ideal.
Then \(x^n J^n=(SI)^n\) is invertible over \(S\) within \(K\).
If \(n=0\), then \(K=\Q\) and the theorem is trivial.
Otherwise, \(SI\) is invertible over \(S\) within \(K\), as was to be shown.
\end{proof}

\begin{corollary}\label{cor:compute_blowup}
There exists a polynomial-time algorithm that, given a number field \(K\) and a finitely generated non-zero subgroup \(I\subseteq K\), computes \(\textup{Bl}(I)\).
\end{corollary}
\begin{proof}
By Exercise~\ref{ex:number_field_noetherian}.a, each subgroup \(I^1,I^2,\dotsc\subseteq K\) can be generated by \([K:\Q]\) elements. Thus we can inductively compute \(I^{k+1}\) as the image of the tensor product \(I^k\tensor I \to K\), and in turn \(\textup{Bl}(I)=I^n:I^n\) for \(n=[K:\Q]-1\), in polynomial time.
\end{proof}

\begin{exercise}
Let \(K\) be a number field and \(I\subseteq K\) a finitely generated non-zero subgroup. Show that \(\Q(\textup{Bl}(I))=\Q(Ix^{-1})\) for any non-zero \(x\in I\).
\end{exercise}

\begin{exercise}\label{ex:torsion_ideals}
Let \(R\subseteq S\subseteq\textup{Q}(R)\) be number rings. 
Show that \(\ker(\mathcal{I}(R)\to\mathcal{I}(S))\cong (S/\mathfrak{f})^*/(R/\mathfrak{f})^*\), where \(\mathfrak{f}=R:S\) is the \emph{conductor}. Conclude that this kernel equals \(\mathcal{I}(R)^\textup{tors}\), the torsion subgroup of \(\mathcal{I}(R)\), when \(S\) is the maximal order.
\end{exercise}

\begin{exercise}
Let \(R\) be an order of a number field \(K\). Show that the following are equivalent:
\begin{enumerate}[nosep,leftmargin=*]
\item There exist a finite {\em Boolean ring} \(B\), i.e.\ a ring \(B\) such that \(x^2=x\) for all \(x\in B\), a subring \(A\subseteq B\) and a surjective ring homomorphism \(f:\mathcal{O}_K\to B\) such that \(R=f^{-1} A\). 
\item There exists an equivalence relation \(\sim\) on the set \(\Hom_\text{Ring}(\mathcal{O}_K,\F_2)\) such that \(R=\{x\in\mathcal{O}_K : f \sim g \Rightarrow f(x)=g(x)\}\).
\item The group \(\mathcal{I}(R)\) is torsion free.
\end{enumerate}
Give an example of a non-maximal order such that \(\mathcal{I}(R)\) is torsion free.
\end{exercise}

\subsection{Ideals generated by two elements}

Every principal fractional ideal of a number ring \(R\) is invertible, and thus its blowup is \(R\) itself.
In this section we give special attention to (computing) the blowup of fractional ideals generated by two elements.
Such ideals are ubiquitous: Every invertible ideal of a number ring can be generated by only two elements (Exercise~\ref{ex:inv_gen_two}). 

\begin{exercise}\label{ex:blowup_ring_extension}
Let \(K\) be a number field and \(R\subseteq K\) a subring. Show that for all finitely generated subgroups \(I,J\subseteq K\) we have \(\textup{Bl}(IJ)=\textup{Bl}(I)\cdot\textup{Bl}(J)\). Conclude that \(\textup{Bl}(R\alpha+R\beta)=R\cdot\textup{Bl}(\Z\alpha+\Z\beta)\) for all \(\alpha,\beta\in K^*\).
\end{exercise}

\begin{exercise}\label{ex:gauss_lemma}
Let \(f\in\Z[X]\) be a \emph{primitive} polynomial, i.e.\ a polynomial for which the only integer divisors are \(\pm 1\). Show that \((f\Q[X]) \cap \Z[X] = f\Z[X]\). \\ \emph{Hint:} Gauss's lemma states that for \(f,g\in\Z[X]\) the product \(fg\) is primitive if and only if \(f\) and \(g\) are primitive.
\end{exercise}

\begin{theorem}\label{thm:MM}
Let \(K=\Q(\gamma)\) be a number field of degree \(n\geq 2\) and let \(f=a_nX^n+\dotsm+a_0X^0\in\Z[X]\) be irreducible with \(f(\gamma)=0\). Consider for \(0\leq j\leq n\) the elements
\begin{align*}
p_j &= a_n \gamma^{j} + a_{n-1}\gamma^{j-1} + \dotsm + a_{n-j} \quad\text{and}\quad q_j = a_j + a_{j-1}\gamma^{-1} +\dotsm + a_0 \gamma^{-j}
\end{align*}
and the groups \(M=\Z+\Z\gamma+\dotsm+\Z\gamma^{n-1}\),
\begin{align*}
D = \Z p_0 + \Z p_1 + \dotsm + \Z p_{n-1} \quad\text{and}\quad N = \Z q_0+\Z q_1 +\dotsm+\Z q_{n-1}.
\end{align*}
Then
\[\textup{Bl}(\Z+\Z\gamma)=M:M=\Z[\gamma]\cap\Z[\gamma^{-1}]=\Z+N=\Z+D\]
and \(N\) and \(D\) are coprime invertible proper ideals of this ring, with quotients isomorphic as rings to \(\Z/a_0\Z\) and \(\Z/a_n\Z\) respectively, such that \(N:D = \gamma\cdot \textup{Bl}(\Z+\Z\gamma)\).
\end{theorem}
\begin{proof}
It follows from Theorem~\ref{thm:blowup} that \(\textup{Bl}(\Z+\Z\gamma)=M:M\), because \(M=(\Z+\Z\gamma)^{n-1}\). 
We extend the definition of the \(a_i\) so that \(a_i=0\) for \(i<0\) and \(i>n\), and similarly define \(p_i\) and \(q_i\) for all \(i\geq 0\). It follows that \(p_i=q_i=0\) for \(i\geq n\).
We proceed to prove the following inclusions.
\[ \Z+N \stackrel{(1)}{\subseteq} \Z+D \stackrel{(2)}{\subseteq} M:M\stackrel{(3)}{\subseteq} \Z[\gamma]\cap\Z[\gamma^{-1}] \stackrel{(4)}{\subseteq} \Z+N.\]

(1) Note that for \(0\leq j\leq n\) we have
\begin{align*}
p_{n-j}+q_{j} 
&= \gamma^{-j}(a_n\gamma^n+\dotsm+a_{j}\gamma^{j})+\gamma^{-j}(a_j \gamma^j + \dotsm + a_0) \\
&= \gamma^{j-n} f(\gamma) + a_{j} = a_{j}.
\end{align*}
In particular \(q_{j}\in \Z+D\) for all \(0\leq j \leq n\) and thus \(\Z+N\subseteq\Z+D\).

(2) Let \(0\leq i,j\leq n\). If \(i<n-j\), then \(p_j\gamma^i\in M\).
Otherwise, we have \(q_{n-j}\gamma^i\in M\), and \(p_j+q_{n-j}=a_{n-j}\in\Z\) as shown in (1), so \(p_j\gamma^i\in M\).
We conclude that \(p_j M \subseteq M\) and \(\Z+D\subseteq M:M\).

(3) We have \(M:M\subseteq M \subseteq \Z[\gamma]\) since \(1\in M\). 
For \(L=\Z+\Z\gamma^{-1}+\dotsm+\Z(\gamma^{-1})^{n-1}=\gamma^{1-n} M\) we similarly have \(M:M=L:L\subseteq\Z[\gamma^{-1}]\).

(4) Suppose there exists some \(x\in \Z[\gamma]\cap\Z[\gamma^{-1}]\) such that \(x\not\in \Z+N\), and choose \(x=\sum_{i=0}^{B} b_i \gamma^i =\sum_{i=0}^{C} c_i \gamma^{-i}\) to be such an element for which \(C\) is minimal.
As \(x\not\in \Z\), we have \(C>0\).
Then \(g=X^C( \sum_{i=0}^{B} b_i X^i - \sum_{i=0}^C  c_i X^{-i}) \in \Z[X] \). 
As \(g(\gamma)=\gamma^C(x-x)=0\) we conclude \(g\in f\Q[X]\) since \(f\) is irreducible. 
Hence \(g=fh\) for some \(h\in\Z[X]\) by Exercise~\ref{ex:gauss_lemma}. 
It follows that \(a_0 \mid c_C\), so we may subtract \((c_C/a_0) q_C\in \Z+N\) from \(x\), violating minimality of \(C\). This proves all inclusions.

It remains to prove the statements about \(N\) and \(D\). Let \(A=\textup{Bl}(\Z+\Z\gamma)\).
Note that \(\gamma p_j = p_{j+1}-a_{n-j-1}=q_{n-j-1}\) for all \(0\leq j < n\), so \(D\gamma = N\).
For \(j=0\) this gives \(\gamma=(p_1-a_{n-1})/a_n\).
Because the ring \(A=\Z+D\) has \Z-basis \(1,p_1,p_2,\dotsc,p_{n-1}\), we conclude \(A\cap\Z\gamma=a_n \Z \gamma\).
Then \(A:(A+A\gamma)=\{\delta\in A: \delta\gamma\in A\}\) intersects \(\Z\) in \(a_n\Z\) and therefore its index in \(A\) is divisible by \(|a_n|\). 
As \(\gamma D=N\subseteq A\) we have that \(D\subseteq A:(A+A\gamma)\), where we note that \(D\) has index \(|a_n|\) in \(A\) by comparing \Z-bases.
Hence \(D=A:(A+A\gamma)\) and in particular \(D\) is an \(A\)-ideal with \(A/D\cong\Z/a_n\Z\) as rings.

Analogously \(N\) is an ideal of \(A\) with \(A/N\cong\Z/a_0\Z\).
We have \(a_j=p_{n-j}+q_j\in N+D\) for all \(0\leq j\leq n\). 
Since \(f\) is irreducible we have that \(\Z=\sum_j \Z a_j \subseteq N+D\), so \(N\) and \(D\) are coprime.
Moreover, \(D\cdot(A+A\gamma)=D+N=A\), so \(D\) and consequently \(N=D\gamma\) are invertible. Thus \(N:D=A\gamma\).
\end{proof}

For more information about the ring in Theorem~\ref{thm:MM}, see \cite{Simon,Melanie}.

\begin{example}\label{ex:MM}
Let \(K\) be a number field, \(R\) a subring of \(K\), and \(\alpha,\beta\in K^*\).
We will give an explicit expression for \(\textup{Bl}(R\alpha+R\beta)\).
Let \(\gamma=\alpha/\beta\). 
By Exercise~\ref{ex:blowup_ring_extension} we may write
\[\textup{Bl}(\alpha R + \beta R) = \textup{Bl}(R+\gamma R) = R\cdot \textup{Bl}(\Z+\gamma\Z).\]
If \(\gamma\in\Q^*\), then \(\Z+\gamma\Z\subseteq\Q\). Since \(\textup{Bl}(\Z+\gamma\Z)\) is an order it must equal \(\Z\), so \(\textup{Bl}(\alpha R+\beta R)=R\).
Otherwise, we have
\[\textup{Bl}(\alpha R+\beta R) = R+p_1R+\dotsm+p_{n-1}R,\]
where the \(p_i\) are as in Theorem~\ref{thm:MM} and \(n\) is the degree of \(\Q(\gamma)\) over \(\Q\).
\end{example}

\begin{exercise}
Let \(K\) be a number field and \(\alpha_1,\dotsc,\alpha_n\in K^*\) for \(n>0\). 
Consider \(I=\alpha_1\Z+\dotsm+\alpha_n\Z\subseteq K\) and \(R_i=\Z[\alpha_1/\alpha_i,\dotsc,\alpha_n/\alpha_i]\). Show that \(\textup{Bl}(I)=\bigcap_i R_i\).
\end{exercise}

\begin{exercise}\label{ex:inv_gen_two}
Let \(R\) be a number ring. 
\begin{enumex}
\item Show that every invertible ideal is generated by two elements. 
\item It follows from (a) that, if \(R\) is a maximal order, then every ideal is generated by two elements. Show that the converse fails for \(R=\Z[\sqrt{5}]\).
\end{enumex}
\end{exercise}

\subsection{Coprime bases for number rings}

We will proceed as in Section~\ref{sec:coprime} and give a coprime basis algorithm for ideals using the blowup.
Recall the notation \(\mathcal{I}(R)\) for the group of invertible ideals of a commutative ring \(R\).

\begin{definition}\label{def:fractional_blowup}
Let \(R\) be a number ring and let \(X\) be a set of non-zero integral ideals of \(R\). 
For a number ring \(R\subseteq S\subseteq \Q R\) we write \(SX=\{SI : I\in X\}\).
We write \(\overline{X}_R\), or \(\overline{X}\) when the ring is understood, for the closure of \(X\cup\{R\}\) under addition, multiplication and integral division, i.e.\ \(I:J\) for ideals \(I,J\subset R\) for which \(J\) is invertible and \(I:J\subseteq R\). 
A \emph{coprime basis} for \(X\) is a set \(B\) of invertible ideals strictly contained in \(R\) which are pairwise coprime and for which \(X\) is contained in \(\langle B\rangle\), the subgroup of \(\mathcal{I}(R)\) generated by \(B\).
We equip the set of coprime bases of \(X\) with a partial order (cf. Exercise~\ref{ex:partial_order}) where \(B\leq C\) if \(\langle B\rangle\subseteq \langle C\rangle\).
\end{definition}

In particular, for a coprime basis of \(X\) to exist, the elements of \(X\) need to be invertible. 
Hence we cannot have a direct translation of Lemma~\ref{lem:closure} to the case of ideals.

\begin{exercise}\label{ex:also_intersection}
Let \(R\) be a number ring and \(X\) a set of non-zero integral ideals of \(R\). Show that if \(\overline{X}\subseteq \mathcal{I}(R)\), then \(\overline{X}\) is closed under intersection.
\end{exercise}

\begin{exercise}\label{ex:ring_extension_closure}
Let \(R\subseteq S \subseteq \Q R\) be number rings and let \(X\) be a set of non-zero ideals of \(R\). 
\begin{enumex}
\item Show that if \(I\) and \(J\) are fractional \(R\)-ideals with \(J\) invertible, then \(SI : SJ = S(I:J)\), and give a counter example where \(J\) is non-invertible.
\item Show that the map 
\(I\mapsto SI\) that maps ideals of \(R\) to ideals of \(S\) respects addition, multiplication and integral division.
\item With the notation as in Definition~\ref{def:fractional_blowup}, conclude that \(S\overline{X}\subseteq\overline{SX}\).
\end{enumex}
\end{exercise}

\begin{exercise}\label{ex:unique_minimal_invertible_closure}
Let \(R\) be an order and \(X\) a (possibly infinite) set of non-zero ideals of \(R\). Show that there exists a unique number ring \(R\subseteq S\subseteq K\) which is minimal with respect to inclusion such that \(\overline{SX}\subseteq\mathcal{I}(S)\). Show that this \(S\) is an order. \\
\emph{Hint:} Iteratively blow up the ideals of \(X\) and use that the index of \(R\) in \(\mathcal{O}_K\) is finite.
\end{exercise}

\begin{proposition}\label{prop:invertible_closure_iff_coprime_basis}
Let \(R\) be an order and \(X\) a (possibly infinite) set of non-zero ideals of \(R\). 
Then \(\overline{X}\subseteq \mathcal{I}(R)\) if and only if \(X\) has a coprime basis. 
If \(X\) has a coprime basis, then \(\{I\in\overline{X}\,|\, \#\{ J \in\overline{X} \,|\, I\subseteq J\} = 2\}\) is the minimal coprime basis (wrt.\ the partial order of Definition~\ref{def:fractional_blowup}).
\end{proposition}
\begin{proof}
If \(\overline{X}\subset\mathcal{I}(R)\), then the proof of Lemma~\ref{lem:closure} can be followed to show that \(\{I\in\overline{X}\,|\, \#\{ J \in\overline{X} \,|\, I\subseteq J\} = 2\}\) is the minimal coprime basis of \(X\). Conversely, if \(X\) has a coprime basis \(B\), then \(\overline{X}\subseteq \langle B\rangle \subseteq\mathcal{I}(R)\).
\end{proof}

\begin{theorem}\label{thm:coprime_base_ideals}
There exists a polynomial-time algorithm that, given an order \(R\) and a finite set of non-zero ideals \(X\) of \(R\), computes the minimal order \(R\subseteq S\) for which there exists a coprime basis for \(SX\) and computes the minimal coprime basis.
\end{theorem}

\begin{Algorithm}\label{alg:coprime_base_ideals}
Write \(\mathfrak{a}_1,\dotsc,\mathfrak{a}_m\) for the elements of \(X\) not equal to \(R\).
Construct a complete simple graph \(G\) with \(m\) vertices and label the vertices \(\mathfrak{a}_1,\dotsc,\mathfrak{a}_m\). 
Set \(S\) equal to \(R\).
Then, for each \(i\) successively replace \(S\) by \(\textup{Bl}(S\mathfrak{a}_i)\).
While there are edges in \(G\), repeat the following 5 steps:
\begin{enumerate}[topsep=0pt,itemsep=-1ex,partopsep=1ex,parsep=1ex]
\item Choose an edge \(\{a,b\}\) of \(G\) and let \(\mathfrak{a}\) and \(\mathfrak{b}\) be the labels of \(a\) and \(b\).
\item Compute \(\mathfrak{c}=\mathfrak{a}+\mathfrak{b}\), replace \(S\) by \(\textup{Bl}(S\mathfrak{c})\), \(\mathfrak{c}\) by \(S\mathfrak{c}\) and compute \(\mathfrak{c}^{-1}=S:\mathfrak{c}\).
\item Add a vertex \(c\) labeled \(\mathfrak{c}\) to \(G\) and connect it to \(a\), \(b\) and those vertices which are neighbors of both \(a\) and \(b\).
\item Update the labels of \(a\) and \(b\) to \(\mathfrak{a}\mathfrak{c}^{-1}\) and \(\mathfrak{b}\mathfrak{c}^{-1}\) respectively.
\item For each \(s\in\{a,b,c\}\), if the label of \(s\) is \(S\), then delete \(s\) and its incident edges from \(G\). 
\end{enumerate}
With \(L\) the set of labels of the remaining vertices, the required order is \(S\) with basis \(\{S\mathfrak{a} \,|\, \mathfrak{a}\in L\}\) of \(\overline{SX}\).
\end{Algorithm}

Note that after step (4), the labels of \(a\), \(b\) and \(c\) are ideals of \(S\), but this is not necessarily true for the remaining labels if \(S\) has become larger.

\begin{proof}[Proof of Theorem~\ref{thm:coprime_base_ideals}]
We will show that Algorithm~\ref{alg:coprime_base_ideals} is correct and runs in polynomial time.
First note that the output is indeed a coprime basis for \(SX\).
Observe that in each iteration \(S\) is contained in the minimal order \(T\) for which \((\overline{TX})_T\subseteq\mathcal{I}(T)\), since \(\textup{Bl}(S\mathfrak{c})\) is the minimal order containing \(S\) in which \(\mathfrak{c}\) becomes invertible.
By Proposition~\ref{prop:invertible_closure_iff_coprime_basis} this means that \(S\) is indeed the minimal order for which \(SX\) has a coprime basis.
That the output is the minimal coprime basis is analogous to the proof of Theorem~\ref{thm:graph_algorithm}. 
One similarly proves the algorithm runs in polynomial time, with one caveat: We need to remark that by Exercise~\ref{ex:small_index} the growth of \(S\) during the algorithm is sufficiently bounded and thus all operations on the ideals can be performed in polynomial time.
\end{proof}

\subsection{Applications of coprime bases}

In this section we consider some algorithmic applications of the coprime basis algorithm for ideals.

\begin{theorem}\label{thm:multiplicative_inversion}
Let \(K\) be a number field and let \(G\subseteq K^*\) be a subgroup.
Then there exists a unique subring \(S\subseteq K\), minimal with respect to inclusion, for which for all finite non-empty subsets \(Y\subseteq G\) the fractional ideal \(SY\) is invertible. Moreover, this \(S\) is an order.
\end{theorem}
\begin{proof}
Let \(T\) be the smallest subring for which all fractional ideals \(T+ T x\) with \(x\in G\) are invertible, which is the smallest ring containing \(\textup{Bl}(\Z+\Z x)\) for all \(x\in G\).
One easily shows that \((T+T x)^{-1}\subseteq T\) and \(Tx=(T+Tx^{-1})^{-1} : (T+Tx)^{-1}\) for all \(x\in G\).
Now let \(T\subseteq S\subseteq K\) be the smallest subring for which all ideals in the closure \(\mathcal{C}\) of \(S\cdot\{(T+Tx)^{-1} \,|\, x\in G\}\) are invertible, which exists by Exercise~\ref{ex:unique_minimal_invertible_closure}. 
Write \(I_x=(S+Sx)^{-1}\) for \(x\in G\). Then for \(Y\subseteq G\) we have
\[SY=\sum_{x\in Y} (I_{x^{-1}}:I_x)=\bigg(\sum_{x\in Y} \Big(\prod_{y\in Y\setminus\{x\}} I_y\Big) I_{x^{-1}}\bigg):  \bigg(\prod_{y\in Y} I_y\bigg), \]
where the latter numerator and denominator are in \(\mathcal{C}\) and hence invertible.
Clearly \(S\) is an order.
It remains to show that this \(S\) is minimal.

Let \(U\) be an order such that for all finite non-empty subsets \(Y\subseteq G\) the fractional ideal \(UY\) is invertible.
Clearly \(T\subseteq U\) since \(\{1,x\}\) is a finite subset of \(G\) for all \(x\in G\).
Let \(Z\) be the set of fractional ideals of \(U\) of the form \(UX : UY\) for finite non-empty \(X,Y\subseteq G\).
Then \(Z\subseteq \mathcal{I}(U)\) and \(Z\) is closed under addition, multiplication and division.
Since \(\{(U+Ux)^{-1}\,|\,x\in G\}\subseteq Z\), we conclude that \(Z\) also contains its closure.
Thus \(S\subseteq U\) and \(S\) is minimal.
\end{proof}

\begin{exercise}
Show that the order \(S\) produced by Theorem~\ref{thm:multiplicative_inversion} need not satisfy \(\Q S = K\). 
Can you give an expression for \(\Q S\)?
\end{exercise}

\begin{theorem}
There exists a polynomial-time algorithm that, given an order \(R\) in a number field \(K\) and a finitely generated subgroup \(G=\langle x_1,\dotsc,x_n\rangle\subseteq K^*\), computes the order \(R\subseteq S\), minimal with respect to inclusion, such that for all finite non-empty subsets \(Y\subseteq G\) the fractional ideal \(SY\) is invertible. 
\end{theorem}
\begin{proof}
We proceed as in Theorem~\ref{thm:multiplicative_inversion}.
Compute the minimal order \(T\) for which for all \(i\) the \(T+Tx_i\) are invertible, for example using Example~\ref{ex:MM}.
Note that \(T+Tx_i^{-1}=x_i^{-1}\cdot (T+Tx_i)\) is also invertible, and that the \((T+Tx_i^{\pm1})^{-1}\) are integral ideals of \(T\).
Then compute the minimal order \(S\) for which all ideals in the closure \(\mathcal{C}\) of \(S\cdot\{(T+Tx_i)^{-1}, (T+Tx_i^{-1})^{-1}\,|\, i\}\) are invertible using Theorem~\ref{thm:coprime_base_ideals}.
We claim \(S\) is the order as required.

Clearly this \(S\) is contained in the one produced in the proof of Theorem~\ref{thm:multiplicative_inversion}.
For the reverse inclusion it suffices to show that \(S+Sx\) is invertible in \(S\) for all \(x\in Y\).
Write \(I_{y}=(S+Sy)^{-1}\) for \(y\in\{x_1,\dotsc,x_n,x_1^{-1},\dotsc,x_n^{-1}\}\). 
Let \(x\in Y\) and write \(x=y_1\dotsm y_k\) for \(y_1,\dotsc,y_k\in\{x_1,\dotsc,x_n,x_1^{-1},\dotsc,x_{n}^{-1}\}\).
We have
\[S+Sx=S+\prod_{i=1}^k(I_{y_i^{-1}}:I_{y_i}) = \Big(\prod_{i=1}^k I_{y_i}+\prod_{i=1}^k I_{y_i^{-1}}\Big):\Big(\prod_{i=1}^k I_{y_i} \Big),\]
with numerator and denominator in \(\mathcal{C}\) and hence invertible.
\end{proof}

\begin{theorem}\label{thm:compute_unit_product_ideals}
There exists a polynomial-time algorithm that, given an order \(R\) of a number field \(K\) and fractional ideals \(\mathfrak{a}_1,\dotsc,\mathfrak{a}_m\) of \(R\), computes the kernel of the map \(\Z^m\to \mathcal{I}(\mathcal{O}_K)\) given by \((k_1,\dotsc,k_m)\mapsto \prod_i (\mathcal{O}_K\mathfrak{a}_i)^{k_i}\).
\end{theorem}
\begin{proof}
We may write \(\mathfrak{a}_i\) as a quotient \(N_i:D_i\) for all \(i\) with \(N_i,D_i\subseteq R\), for example with \(D_i=n R\) for some \(n\in\Z_{>0}\).
Compute an order \(R\subseteq S\subseteq K\) and a coprime basis \(\mathfrak{c}_1,\dotsc,\mathfrak{c}_n\) for \(SN_1,\dotsc,SN_m,SD_1,\dotsc,SD_m\), as well as integers \(n_{ij},d_{ij}\) such that \(N_i=\prod_j \mathfrak{c}_j^{n_{ij}}\) and \(D_i=\prod_j \mathfrak{c}_j^{d_{ij}}\).
Then the kernel of \(\Z^m\to\mathcal{I}(S)\) is the kernel of the map \(\Z^m\to\Z^n\) given by \((k_1,\dotsc,k_m)\mapsto (\sum_{i} k_i (n_{ij}-d_{ij}))_j\), which we may compute using Theorem~\ref{thm:ker_im_alg}. Finally, note that the map \(\mathcal{I}(S)\to\mathcal{I}(\mathcal{O}_K)\) is injective on any subgroup generated by a coprime basis.
\end{proof}

In particular, we may decide whether some power-product of invertible ideals over any order \(R\) is torsion in \(\mathcal{I}(R)\) (Exercise~\ref{ex:torsion_ideals}).

\begin{corollary}\label{cor:compute_unit_product_elements}
There exists a polynomial-time algorithm that, given elements \(\alpha_1,\dotsc,\alpha_m\) in a number field \(K\) and integers \(k_1,\dotsc,k_m\in\Z\), computes the kernel of the map \(\Z^m\to K^*/\mathcal{O}_K^*\) given by \((k_1,\dotsc,k_m)\mapsto \prod_i \alpha_i^{k_i} \cdot\mathcal{O}_K^*\).
\end{corollary}
\begin{proof}
Pick a full-rank order \(R\subseteq K\) and apply Theorem~\ref{thm:compute_unit_product_ideals} to \(\alpha_1 R,\dotsc,\alpha_m R\).
\end{proof}

\subsection{Unit products}\label{sec:main_theorem}

In this section we will prove the following generalization of Theorem~\ref{thm:main_intro} from the introduction.

\begin{theorem}[Ge \cite{Ge}]\label{thm:compute_trivial_product_elements}
There exists a polynomial-time algorithm that, given elements \(\alpha_1,\dotsc,\alpha_m\) in a number field \(K\), computes the kernel of the map \(\Z^m\to K^*\) given by \((k_1,\dotsc,k_m)\mapsto \prod_i \alpha_i^{k_i}\).
\end{theorem}

Corollary~\ref{cor:compute_unit_product_elements} is close to being a proof of this. For this, it essentially remains to prove Theorem~\ref{thm:compute_trivial_product_elements} for \(\alpha_i\in\mathcal{O}_K^*\). 

Let \(K\) be a number field and let \(X=X(K)\) be the set of ring homomorphisms \(K\to\C\).
For a subgroup \(G\subseteq K^*\) one obtains a subgroup
\[ \Lambda_G=\Big\{ \big(\log \sigma(\alpha)\big)_{\sigma\in X} : \alpha\in G \Big\} + 2\pi \textup{i} \Z^X \subseteq \C^X, \]
which is independent of the chosen branch of the natural logarithm, and a group isomorphism \(G\to \Lambda_G / 2\pi\textup{i} \Z^X\). 
In our special case where \(G\subseteq\mathcal{O}_K^*\), it is even a lattice (Exercise~\ref{ex:log_lattice}).
Then, we compute the kernel of \(\Z^m\to G\) from the map \(\Z^m\to \Lambda_G\).

\begin{lemma}
Let \(K\) be a number field and let \(\varepsilon\in \mathcal{O}_K\) be non-zero.
Then \(\varepsilon=1\) if and only if for all \(\sigma\in X\) there exists \(k_\sigma\in\Z\) such that 
\(|\log \sigma(\varepsilon) - k_\sigma \cdot 2\pi \textup{i} | < \log 2\).
\end{lemma}
\begin{proof}
The forward implication is trivial.
Suppose now we have such \(k_\sigma\in\Z\) for each \(\sigma\in X\).
Let \(\sigma\in X\). 
With \(z=\log \sigma(\varepsilon)-k_\sigma\cdot2\pi\textup{i}\) we obtain
\begin{align*}
| \sigma(\varepsilon)-1|=|\textup{e}^{z}-1|=\Big|\sum_{n=1}^\infty \frac{z^n}{n!} \Big| \leq \sum_{n=1}^\infty \frac{|z|^n}{n!} = \textup{e}^{|z|}-1 < \textup{e}^{\log 2}-1 = 1.
\end{align*}
Write \(N:K\to\Q\) for the norm function. Then \(N(\varepsilon-1)\) is an integer, because it is both integral and rational, such that
\[ |N(\varepsilon-1)| = \prod_{\sigma\in X} | \sigma(\varepsilon) - 1| < 1.\]
Thus \(N(\varepsilon-1)=0\) and \(\varepsilon=1\).
\end{proof}

\begin{exercise}\label{ex:log_lattice}
Let \(K\) be a number field.
Suppose we consider \(\C\) as a 2-dimensional Euclidean vector space with orthonormal basis \(\{1,\textup{i}\}\) and in turn consider the orthogonal sum \(\C^{X(K)}\) as a Euclidean vector space.
Show that for every subgroup \(H\subseteq \mathcal{O}_K^*\) the subgroup \(\Lambda_H \subseteq \C^X\) is a lattice with \(\lambda_1(\Lambda_H)\geq (\log 2)^2\).
\end{exercise}

\begin{proof}[Proof of Theorem~\ref{thm:compute_trivial_product_elements}]
First compute using Corollary~\ref{cor:compute_unit_product_elements} a basis \(\langle b_1,\dotsc,b_n\rangle\) for 
\[H=\Big\{(k_1,\dotsc,k_m)\in\Z^m : \prod_i \alpha_i^{k_i} \in \mathcal{O}_K^*\Big\}.\]
Consider now the lattice
\[\Lambda_H=\sum_{i=1}^n \Z \cdot \Big( \sum_{j=1}^m b_{ij} \log \sigma(\alpha_j) \Big)_{\sigma\in X(K)} + 2\pi \textup{i} \Z^{X(K)},\]
to which we want to apply Theorem~\ref{thm:approximate_lattice} to deduce the required kernel.
Exercise~\ref{ex:log_lattice} states that \(\lambda_1(\Lambda_H)\geq (\log 2)^2\). 
Hence we may pick \(\omega\) and \(t\) in Theorem~\ref{thm:approximate_lattice} so that their lengths are polynomially bounded. 
Using the work of Sch\"onhage \cite{Schonhage} and Brent \cite{Brent} we can compute in polynomial time a \(t\)-approx\-ima\-tion of the elements generating \(\Lambda_H\).
Hence the algorithm runs in polynomial time.
\end{proof}

\section{Computing the maximal order}\label{sec:rings}

This section we dedicate to the computation of the maximal order of a number field, in particular why it is difficult. We do this by proving that this problem is equally hard in some sense yet to be made precise (Definition~\ref{def:reduction}) as computing the \emph{radical}, written \(\textup{rad}(n)\), of some \(n\in\Z_{>0}\), which is the largest square-free divisor of \(n\) or equivalently the product of all primes dividing \(n\) without multiplicity. Computing radicals is considered difficult, as it seems like we need to factor \(n\), although it is unknown whether it is actually equally hard as factorization.

We will consider rings which we encode as a finitely generated (additive) abelian group \(A\) together with a multiplication morphism \(A\tensor A\to A\).
This includes both orders and finite commutative rings.

Recall the definition of the the nilradical from Definition~\ref{def:basic_ring}.3.

\begin{theorem}[cf.\ Theorem 1.3 in \cite{Buchman-Lenstra}]\label{thm:equivalent_problems}
The following three computational problems are equally hard:
\begin{enumerate}[topsep=0pt,itemsep=-1ex,partopsep=1ex,parsep=1ex]
\item Given a number field $K$, compute $\mathcal{O}_K$.
\item Given a finite ring $A$, compute $\nil(A)$.
\item Given an integer $m > 0$, compute $\rad(m)$.
\end{enumerate}
\end{theorem}

We prove this theorem at the end of Section~\ref{ssec:singular_ideals}.
As a corollary to this theorem we obtain an algorithm to compute maximal orders, albeit a rather slow one, because we have a naive algorithm to compute radicals.

\subsection{Polynomial-time reductions}

Recall the definition of a problem from Definition~\ref{def:algorithm}.
Intuitively, a problem \(f\) is easier than a problem \(g\) if you can solve \(f\) once you know how to solve \(g\), or more specifically to algorithms, when you can solve them in polynomial time. We can make this more formal.

\begin{definition}\label{def:reduction}
Let \(f:I_f\to f(I_f)\) and \(g:I_g\to g(I_g)\) be two problems. We say \emph{\(f\) can be reduced to \(g\)}, or symbolically \(f\preceq g\), if there exist polynomial-time algorithms \(a:I_f\to I_g\) and \(b:I_f\times g(I_g) \to f(I_f)\) so that \(f(i)=b(i,g(a(i)))\) for all \(i\in I_f\). We say \(f\) and \(g\) are \emph{equally hard} if \(f\preceq g\) and \(g\preceq f\).
\end{definition}

Effectively, this definition states that \(f\preceq g\) if we can efficiently transform the input \(i\) for \(f\) into an input for \(g\), then run \(g\), and from its output efficiently deduce \(f(i)\).
Note that if there exists a polynomial-time algorithm for \(f\), then we always have \(f\preceq g\) for any non-empty problem \(g\). 
This shows the concept of reductions is not very useful for problems for which we know a polynomial-time algorithm exists.
The reader should be aware that there exist non-equivalent definitions of a reduction of problems, for example where it is allowed to make multiple calls to \(g\).

\begin{example}\label{ex:nilrad_to_rad}
The problem of computing radicals can be reduced to computing nilradicals of finite commutative rings. Namely, for \(a\in\Z_{>0}\) we have that \(\textup{nil}(\Z/a\Z)=\textup{rad}(a)\Z/a\Z\), so we can deduce \(\textup{rad}(a)\) given \(\textup{nil}(\Z/a\Z)\).
\end{example}

\begin{exercise}[cf.\ Lemma~8.1.3 in \cite{Radicals}] \label{ex:is_field_rad} 
Let \(l\) be prime, \(K\) a field and \(a\in K\). 
Show that \(f=X^l-a\) is irreducible in \(K[X]\) if and only if \(a\) is not an \(l\)-th power in \(K\). \\
\emph{Hint:} Write the constant term of a factor of \(f\) in terms of a single root of \(f\).
\end{exercise}

\begin{lemma}\label{lem:OK_to_rad}
The problem of computing radicals can be reduced to computing the maximal order of a number field.
\end{lemma}
\begin{proof}
Suppose we are given \(a\in\Z_{>1}\) for which to compute \(\textup{rad}(a)\).
Choose a sufficiently small prime \(l>\log_2 a\). 
Such an \(l\) will be at most \(2\log_2 a\) by Bertrand's postulate, and \(2\log_2 a\) is sufficiently small to do a naive primality test.
Let \(K=\Q(\alpha)=\Q[X]/(X^l-a)\), which is a field by Exercise~\ref{ex:is_field_rad}.
Compute the maximal order \(\mathcal{O}\) of \(K\).
For each \(0 < i < l\) express \(\alpha^i\) in terms of a basis for \(\mathcal{O}\), from which we may compute the largest integer \(b_i\) such that \(\alpha^i/b_i\in \mathcal{O}\).
Then \(c_i=a^i/b_i^l=(\alpha^i/b_i)^l\in\mathcal{O}\cap\Q=\Z\).
Now compute \(\gcd\{c_i \,|\, 0 < i < l\}\).
We claim that this equals \(\textup{rad}(a)\).

Write \(a=\prod_p p^{a_p}\) and similarly \(b_i\) and \(c_i\).
Fix some prime \(p\mid a\).
Then \(0<a_p\leq\log_2 a < l\). 
As \(b_i\) is maximal such that \(a^i/b_i^l\in\Z\), we have that \(b_{ip}=\lfloor ia_{p} / l \rfloor\) and that \(c_{ip}\) is the remainder of \(ia_p\) upon division by \(l\).
This remainder is positive for all \(i\), and equals \(1\) for some \(0<i<l\).
It follows that \(\gcd\{c_i\,|\, 0<i<l\}=\prod_{p\mid a} p = \textup{rad}(a)\), as was to be shown.
\end{proof}

Example~\ref{ex:nilrad_to_rad} and Lemma~\ref{lem:OK_to_rad} already prove a part of Theorem~\ref{thm:equivalent_problems} stated in the introduction. Each of the problems in this theorem has an associated decision problem, which we will show are all equally hard as well.

Each of the problems in Theorem~\ref{thm:equivalent_problems} has an associated decision problem, which we will show are all equally hard as well.

\begin{theorem}[cf.\ Theorem 6.12 in \cite{Buchman-Lenstra}]\label{thm:equivalent_decision_problems}
The following three decision problems are equally hard:
\begin{enumerate}[topsep=0pt,itemsep=-1ex,partopsep=1ex,parsep=1ex]
\item Given an order $R$ in a number field \(K\), decide whether $R = \mathcal{O}_{K}$.
\item Given a finite commutative ring $A$, decide whether $\textup{nil}(A)=0$.
\item Given an integer $m > 0$, decide whether $m$ is square-free.
\end{enumerate}
\end{theorem}

We will prove this theorem at the end of Section~\ref{ssec:reduced_discriminants}.
As of \today, for none of the six problems in Theorem~\ref{thm:equivalent_problems} and Theorem~\ref{thm:equivalent_decision_problems} there are known polynomial-time algorithms.
Note that each decision problem trivially reduces to its corresponding computational problem.

\begin{exercise}\label{ex:equivalent_decision_problems}
Prove \(3\preceq 1\) and \(3\preceq 2\) for Theorem~\ref{thm:equivalent_decision_problems}. \\
\emph{Hint: } Use Exercise~\ref{ex:radical_discriminant}.
\end{exercise}

\begin{exercise}
Show that there exists a polynomial-time algorithm that, given \(a,b\in\Z_{>0}\), decides whether \(\textup{rad}(a)=\textup{rad}(b)\).
\end{exercise}

\begin{exercise}
Show that prime factorization is equally hard as decomposing a finite abelian group as a product of cyclic groups of prime-power order.
\end{exercise}

\subsection{Nilradicals}

In this section we will compute \(\textup{nil}(A)\) for a finite commutative ring \(A\) given information about \(\# A\), which will prove part of Theorem~\ref{thm:equivalent_problems}. We will need to treat the small primes in \(\#A\) separately.

\begin{lemma}\label{lem:reduce_powers_radical}
Let \(A\) be a commutative ring and let \(n,k\in\Z_{\geq 1}\). Then \(\nil(A/n^kA)\) is the inverse image of \(\nil(A/nA)\) under the natural map \(A/n^kA\to A/nA\). 
\end{lemma}
\begin{proof}
We have \((nA)^k\equiv 0 \bmod n^k A\), so \(nA/n^kA\subseteq \nil(A/n^kA)\).
\end{proof}

\begin{lemma}\label{lem:frobenius}
There exists a polynomial time algorithm that, given a prime \(p\) and a finite commutative \(\F_p\)-algebra \(R\), computes the Frobenius morphism \(x\mapsto x^p\) of \(R\).
\end{lemma}
\begin{proof}
It suffices to compute \(x^p\) for each \(x\) in some additively generating set of \(R\). Note that \(p\) can be too large to naively compute \(x^p\) with \(p-1\) multiplications.
Iteratively compute \(x^{2^0},x^{2^1},x^{2^2},\dotsc,x^{2^n}\) by squaring for \(n\leq \log_2 p\), and by writing \(p\) in base-2 then compute \(x^p\) in at most \(\log_2 p\) multiplications.
\end{proof}

\begin{proposition}[cf.\ Algorithm 6.1 in \cite{Buchman-Lenstra}]\label{prop:nilradical_prime_power}
There is a polynomial-time algorithm that, given a finite commutative ring $A$ with $\# A$ a prime power, computes the nilradical $\nil(A)$. 
\end{proposition}

We first state the algorithm.

\begin{Algorithm}\label{alg:nilradical_prime_power}
Let \(A\) be a finite commutative ring of prime power order.
\begin{enumerate}[topsep=0pt,itemsep=-1ex,partopsep=1ex,parsep=1ex]
\item Compute \(\# A\) and find \(p\) prime and \(n\in\Z_{\geq 0}\) such that \(\#A=p^n\).
\item Compute \(B=A/pA\), which is an \(\F_p\)-algebra, and the least \(t\in\Z_{\geq 0}\) such that \(\dim_{\F_p} B \leq p^t\).
\item Compute the Frobenius \(F\) of \(B\) with Lemma~\ref{lem:frobenius} and compute \(K=\ker(F^t)\).
\item Return the inverse image of \(K\) under the natural map \(\pi:A\to B\).
\end{enumerate}
\end{Algorithm}

\begin{proof}[Proof of Proposition~\ref{prop:nilradical_prime_power}]
We will show Algorithm~\ref{alg:nilradical_prime_power} satisfies our requirements.
Note that all steps in the algorithm involve only basic computations as described in Section~\ref{sec:basic_comp} and Section~\ref{sec:abgp}.
Most importantly, we use that \(F\) and thus \(F^t\) is a linear map, to compute \(K\).
Hence the algorithm runs in polynomial time.
We will thus show it is correct.
For each \(x\in\nil(B)\) and \(m\geq \dim_{\F_p} B\) we have \(x^m=0\), so \(x\in\nil(B)\) if and only if \(F^t(x)=0\).
Then \(\pi^{-1}K=\nil(A)\) by Lemma~\ref{lem:reduce_powers_radical}, so the algorithm is correct.
\end{proof}

\begin{exercise}\label{ex:test_local_ring}
You may assume the existence of a polynomial-time primality test, for example the AKS algorithm~\cite{AKS}.
\begin{enumex}
\item Show that a finite commutative ring \(R\) is local if and only if \(\# R=p^k\) for some prime \(p\) and integer \(k>0\) and \(\textup{nil}(R)\) is a maximal ideal.
\item Show that a finite commutative ring \(R\) with \(\textup{nil}(R)=0\) is a field if and only if there exists some prime \(p\) such that  \(\# R=p^k\) for some \(k\in\Z_{>0}\) and \(\dim_{\F_p} \ker(x\mapsto x^p-x)=1\).
\item Conclude that there exist a polynomial-time algorithm that, given a finite commutative ring \(R\), decides whether \(R\) is local and if so computes its maximal ideal.
\end{enumex}
\end{exercise}

Recall the definition of the trace from Section~\ref{sec:numberring}.

\begin{definition}
Let \(A\subseteq B\) be commutative rings such that \(B\) is free of finite rank over \(A\) as an \(A\)-module. Then we define the \emph{trace radical} of \(B\) over \(A\) as
\[\Trad(B/A)=\{x\in B \,|\, \Tr_{B/A}(xB)=0\} = \ker( x \mapsto (y \mapsto \Tr_{B/A}(xy)) ), \]
where the kernel is taken of a homomorphism of \(B\)-modules \(B \to \Hom_A(B,A)\).
\end{definition}

It follows trivially that \(\Trad(B/A)\) is an ideal of \(B\).
In fact, it is the largest ideal of \(B\) contained in \(\ker(\Tr_{B/A})\).
For \(A\) a commutative ring that is finitely generated as a module over \(\Z\) we may compute \(\Trad(B/A)\) in polynomial time using the theory of Section~\ref{sec:abgp}.
As the following exercise in linear algebra shows, the trace function can in some sense test nilpotency.
Similarly, we hope to compute \(\nil(B)\) from \(\Trad(B/A)\).

\begin{exercise}\label{ex:field_nilpotent}
Let \(A\) be a field and let \(B\) be a finite-dimensional \(A\)-vector space.
\begin{enumex}
\item Suppose \(A\) has characteristic \(0\) or \(p\) for some prime \(p>\dim_A(B)\).
Show that \(M\in \End_A(B)\) is nilpotent if and only if \(\Tr(M^n)=0\) for all \(n\in\Z_{>0}\). \\ 
\emph{Hint:} Apply Newton's identities to the characteristic polynomial of \(M\).
\item Suppose \(A\) has characteristic \(0\) or \(p\) for some prime \(p>\dim_A(B)\) and that \(B\) is a commutative \(A\)-algebra. 
Show that \(\Trad(B/A)=\nil(B)\).
\item Give a counter-example to (b) when \(A\) has prime characteristic \(p\leq \dim_A(B)\).
\item What happens when we replace \(A\) by a local ring and \(B\) by a free \(A\)-module? 
\end{enumex}
\end{exercise}

\begin{exercise}\label{ex:trace_reductions}
Let \(A\subseteq B\) be a commutative rings such that \(B\) is free of finite rank over \(A\). Let \(S\subseteq A\) be a multiplicative set, \(\mathfrak{a}\subseteq A\) an ideal and \(x\in B\).
\begin{enumex}
\item Show that the image of \(\textup{Tr}_{B/A}(x)\) in \(S^{-1}A\) equals \(\textup{Tr}_{S^{-1}B/S^{-1}A}(x)\).
\item Show that the image of \(\textup{Tr}_{B/A}(x)\) in \(A/\mathfrak{a}\) equals \(\textup{Tr}_{(B/\mathfrak{a}B)/(A/\mathfrak{a})}(x)\).
\end{enumex}
\end{exercise}

\begin{exercise}\label{ex:trace_det_product}
Let \(A\subseteq B_1,B_2\) be commutative rings such that \(B_1\) and \(B_2\) are free \(A\)-modules of finite rank, and let \(b=(b_1,b_2)\in B_1\times B_2\). Show that
\begin{align*}
\Tr_{(B_1\times B_2)/A}(b)&=\Tr_{B_1/A}(b_1)+\Tr_{B_2/A}(b_2) \quad\text{and} \\
\det{}_{(B_1\times B_2)/A}(b)&=\det{}_{B_1/A}(b_1)\cdot\det{}_{B_2/A}(b_2).
\end{align*}
\end{exercise}

\begin{exercise}\label{ex:canon_trace_det}
Suppose \(B\) is a free \(A\)-module and let \(I,J\subseteq A\) be coprime ideals with \(I\cap J=0\).
Then \(B\cong (B/BI)\times (B/BJ)\) and \(B/BI\) is a free \(A/I\)-module.
Show that \(\Tr\) and \(\det\) commute with this isomorphism when \(B\) has finite rank.
\end{exercise}

\begin{proposition}\label{prop:trad_is_rad}
Let \(m\) be a square-free integer and \(B\) a commutative algebra over \(A=\Z/m\Z\) that is free of finite rank as a module.
If \(p>\textup{rk}_A(B)\) for all primes \(p\) dividing \(m\), then \(\Trad(B/A)=\nil(B)\).
\end{proposition}
\begin{proof}
We have \(A\cong \prod_{p\mid m} (\Z/p\Z)\) and \(B\cong \prod_{p\mid m}(B/p B)\).
Clearly, under these isomorphisms we have \(\nil(B)\cong \prod_{p\mid m} \nil(B/pB)\), and from Exercise~\ref{ex:canon_trace_det} it follows that also \(\Trad(B/A)\cong \prod_{p\mid m}\Trad((B/pB)/(A/pA))\).
Note that \(p>\rk_A(B) = \rk_{A/pA}(B/pB)\) for all \(p\mid m\).
Hence it suffices to prove the proposition for \(m\) prime, which is Exercise~\ref{ex:field_nilpotent}.
\end{proof}

With this proposition, we may prove the following.

\begin{theorem}\label{thm:rad_to_nil}
There exists a polynomial-time algorithm that, given a finite commutative ring $A$ and the integer \(\rad(\#A)\), computes \(\nil(A)\).
\end{theorem}

\begin{Algorithm}\label{alg:comp_ring_rad}
Let \(A\) be a finite commutative ring and let \(r=\rad(\#A)\) and compute \(l=\lfloor\log_2(\#A)\rfloor\).
\begin{enumerate}[topsep=0pt,itemsep=-1ex,partopsep=1ex,parsep=1ex]
\item Apply Algorithm~\ref{alg:graph_algorithm} to \(\{r,\#(A/rA)\}\cup\{p\mid p\leq l,\, p \text{ prime}\}\) to compute a coprime basis \(c_1,\dotsc,c_n\).
\item Factor \(\#(A/rA)=\prod_{i\in I} c_i^{k_i}\) with \(I\) such that \(k_i\geq 1\) for all \(i\in I\).
\item For each \(i\in I\) compute \(\nil(A/c_iA)\) as follows:
\begin{enumerate}[topsep=-10pt,itemsep=-1ex,partopsep=1ex,parsep=1ex]
\item If \(c_i\leq l\), then \(c_i\) is prime and apply Algorithm~\ref{alg:nilradical_prime_power}.
\item If \(c_i>l\), then \(A/c_i A\) is free over \(\Z/c_i\Z\) of rank \(k_i\) and compute \(\nil(A/c_iA)\) as \(\Trad((A/c_iA)/(\Z/c_i\Z))\) using linear algebra.
\end{enumerate}
\item Return the inverse image of \(\prod_{i\in I} \nil(A/c_iA)\) under the natural map \(A\to\prod_{i\in I} (A/c_i A)\).
\end{enumerate}
\end{Algorithm}

\begin{proof}[Proof of Theorem~\ref{thm:rad_to_nil}]
Clearly Algorithm~\ref{alg:comp_ring_rad} runs in polynomial time if it is correct.
For \(i\in I\) we have \(\gcd(c_i,r) > 1\), so \(c_i\mid r\) and \(c_i\) is square-free.
If \(c_i\leq l\), then it is prime, so Algorithm~\ref{alg:nilradical_prime_power} correctly computes \(\nil(A/c_iA)\) by Proposition~\ref{prop:nilradical_prime_power} in step (3a).
Now suppose \(c_i>l\).
As \(c_i\) is square-free we have \(A/c_i A\cong \prod_{p\mid c_i} A/pA\).
Clearly \(A/pA\) is free over \(\Z/p\Z\) since the latter is a field, and its rank must be \(k_i\).
Hence \(A/c_i A\) is free over \(\Z/c_i\Z\) of rank \(k_i\).
Thus the conditions to Proposition~\ref{prop:trad_is_rad} are satisfied and step (3b) correctly computes \(\textup{nil}(A/c_iA)\).
By Lemma~\ref{lem:reduce_powers_radical} we compute \(\textup{nil}(A)\) in step (4).
\end{proof}

\subsection{Non-invertible ideals}\label{ssec:singular_ideals}
By Theorem~\ref{thm:blowup}, finding larger orders within the same number field can be done by finding non-invertible ideals.

\begin{lemma}\label{lem:singular_discriminant}
Let \(R\) be a full-rank order in a number field \(K\) and let \(p\) be prime. Then \(p \mid \#(\mathcal{O}_K/R)\) if and only if \(R\) has a non-invertible prime ideal containing \(p\).
\end{lemma}
\begin{proof}
Every prime of \(R\) containing \(p\) is invertible if and only if every prime of \(R_p\) is invertible if and only if \(R_p\) is integrally closed if and only if \(R_p=(\mathcal{O}_K)_p\) if and only if \((\mathcal{O}_K/R)_p=1\) if and only if \(p\nmid \#(\mathcal{O}_K/R)\).
\end{proof}

\begin{lemma}[cf.\ Theorem~1.2 in \cite{Buchman-Lenstra}]\label{lem:compute_OK}
There is a polynomial-time algorithm that, given a full-rank order $R$ in a number field $K$ and square-free integer $d$ contained in all non-invertible prime ideals of \(R\), computes $\mathcal{O}_K$. 
\end{lemma}
\begin{proof}
For all \(R\subseteq S\subseteq\mathcal{O}_K\) it holds that \(d\) is contained in all non-invertible prime ideals of \(S\).
Namely, this is equivalent to \(p\mid d\) for all primes \(p\mid\#(\mathcal{O}_K/S)\) by Lemma~\ref{lem:singular_discriminant}, and we have that \(\#(\mathcal{O}_K/S) \mid \#(\mathcal{O}_K/R)\).

Using Algorithm~\ref{alg:comp_ring_rad} we may compute \(\rad(R/dR)\) and in turn its inverse image \(\mathfrak{a}\) in \(R\).
We compute \(\textup{Bl}(\mathfrak{a})\), and repeat the procedure with \(R\) replaced by \(\textup{Bl}(\mathfrak{a})\) with the same \(d\) until \(R=\textup{Bl}(\mathfrak{a})\).
Then \(\mathfrak{a}\) is the product of all prime ideals containing \(d\), so by assumption every non-invertible prime ideal of \(R\) is a factor of \(\mathfrak{a}\).
Because \(\mathfrak{a}\) is invertible over \(R\) by Theorem~\ref{thm:blowup}, as is each factor, we have \(R=\mathcal{O}_K\).
\end{proof}

\begin{proposition}\label{prop:rad_to_OK}
Computing the maximal order of a number field reduces to computing radicals of integers.
\end{proposition}
\begin{proof}
Let \(R\) be any order of \(K\). 
Then we may compute \(d=\rad|\Delta(R)|\).
As \(\#(\mathcal{O}_K/R) \mid \Delta(R)\) by Exercise~\ref{ex:discriminant_subring} we get that \(\rad(\#(\mathcal{O}_K/R))\mid d\).
Then by Lemma~\ref{lem:singular_discriminant} every prime \(p\) for which there exists a non-invertible prime of \(R\) above \(p\) we have \(p\mid d\).
Since \(d\) is square-free we obtain \(\mathcal{O}_K\) from Lemma~\ref{lem:compute_OK}.
\end{proof}

We can now prove the main theorem.

\begin{proof}[Proof of Theorem~\ref{thm:equivalent_problems}]
Example~\ref{ex:nilrad_to_rad}, Lemma~\ref{lem:OK_to_rad}, Theorem~\ref{thm:rad_to_nil} and Proposition~\ref{prop:rad_to_OK} combined  prove Theorem~\ref{thm:equivalent_problems}.
\end{proof}

\subsection{Reduced discriminants}\label{ssec:reduced_discriminants}
Lemma~\ref{lem:singular_discriminant} states that we can find the non-inverti\-ble prime ideals of an order \(R\) of a number field \(K\) precisely above rational primes dividing \(\#(\mathcal{O}_K/R)\). However, algorithmically we do not have access to this index. 
In the search of these rational primes \(p\), we could consider \(p\) for which \(p^2 \mid \Delta(R)\), as \(\Delta(R)=\#(\mathcal{O}_K/R)^2 \cdot \Delta(\mathcal{O}_K)\) by Exercise~\ref{ex:discriminant_subring}.
This generally gives `false positives', as \(\Delta(\mathcal{O}_K)\) could have independent square divisors.
To recover an equally strong result as Lemma~\ref{lem:singular_discriminant} that is algorithmically accessible, we turn to the reduced discriminant.

We equip a number field \(K\) with the symmetric \(\Q\)-bilinear map \(K\times K\to \Q\) given by \((x,y)\mapsto \Tr_{K/\Q}(xy)\). By Exercise~\ref{ex:trace_iff_det} the hypothesis of Exercise~\ref{ex:general_dagger} is satisfied and thus we have a duality as follows.

\begin{definition}[cf.\ Exercise~\ref{ex:general_dagger}]\label{def:trace_dual}
Let \(I\) be a subgroup of a number field \(K\).
We define the \emph{trace dual} \(I^\dagger\) of \(I\) as 
\[I^\dagger = \{ x\in K \,|\, \Tr_{K/\Q}(xI) \subseteq \Z \}.\]
We define the \emph{reduced discriminant} \(\delta(R)\) of a full-rank order \(R\) in \(K\) to be the exponent of the abelian group \(R^\dagger / R\).
\end{definition}

Since \(\Tr_{K/\Q}(R)\subseteq \Z\) for any order \(R\), we have \(R\subseteq R^\dagger\) and thus \(R^\dagger/R\) is a well-defined group.
Moreover, \(R^\dagger\) is a fractional \(R\)-ideal, so the quotient is in fact finite, and \(\delta(R)\) is well-defined.
Note that we can compute \(R^\dagger\) and \(\delta(R)\) in polynomial time. 
As is the case for the trace radical in its relation to the nilradical (Proposition~\ref{prop:trad_is_rad}), the reduced discriminant has good properties only at the `large primes'.

\begin{exercise}\label{ex:trace_dual_properties}
Let \(K\) be a number field and let \(0\subsetneq I,J\subseteq K\) be full-rank finitely generated additive subgroups. 
Show that if \(I\subseteq J\), then \(J^\dagger \subseteq I^\dagger\).
Conclude that \(S \subseteq R^\dagger\) for all full-rank orders \(R\) and \(S\) in \(K\).
Show that \((I^\dagger)^\dagger=I\), \((I+J)^\dagger=I^\dagger \cap J^\dagger\) and \((I\cap J)^\dagger=I^\dagger+J^\dagger\). 
\end{exercise}

The following proposition motivates the name `reduced discriminant'.

\begin{proposition}
Let \(R\) be an order. Then \(|\Delta(R)|=\#(R^\dagger/R)\).
\end{proposition}
\begin{proof}
Let \(b_1,\dotsc,b_n\) be a basis of \(R\) and consider its dual basis \(b_1^\dagger,\dotsc,b_1^\dagger\) from Exercise~\ref{ex:general_dagger}.
Consider then the linear map \(f:\textup{Q}(R)\to\textup{Q}(R)\) that maps \(b_i^\dagger\) to \(b_i\) for all \(i\). Then \(\#(R^\dagger/R)=|\det(f)|\) and
\[\Delta(R) = \det( (\Tr(b_i b_j))_{i,j}) = \det(f) \cdot \det( (\Tr(b_i f^{-1} (b_j)))_{i,j}) = \det(f), \]
as was to be shown.
\end{proof}

It follows that \(\textup{rad}|\Delta(R)|=\textup{rad}(\delta(R))\).

\begin{lemma}\label{lem:better_reduced_discriminant}
Let \(R\subseteq S\) be full-rank orders of a number field of degree \(n\) and let \(p>n\) be prime. Then \(S/R\subseteq p(R^\dagger /R) \).
\end{lemma}
\begin{proof}
By Proposition~\ref{prop:trad_is_rad} we have that 
\[ \textup{nil}(R/pR) = \textup{Trad}((R/pR)/\F_p) = \{x\in R \,|\,\Tr(xR)\subseteq p\Z \}/pR=(pR^\dagger \cap R)/pR, \]
and similarly for \(S\).
The ring homomorphism \(R/pR\to S/pS\) preserves nilpotents, hence \(pR^\dagger \cap R \subseteq pS^\dagger \cap S\).
We then apply \((-)^\dagger\) to the inclusion \(R^\dagger \cap \frac{1}{p} R \subseteq S^\dagger\) to obtain \(R+pR^\dagger \supseteq S\) from Exercise~\ref{ex:trace_dual_properties}. Hence \(S/R\subseteq (R+pR^\dagger)/ R=p(R^\dagger/R)\).
\end{proof}

\begin{lemma}\label{lem:local_principal_comparable}
Let \(R\) be a local ring with ideals \(I,J\subseteq R\). If \(I+J\) is principal, then \(I\subseteq J\) or \(J\subseteq I\).
\end{lemma}
\begin{proof}
Write \(\fm\subseteq R\) for the maximal ideal and \(I+J=xR\) for some \(x\in R\).
Consider the \(R/\fm\)-vector space \(A=xR/x\fm\).
Note that the image of \(I+J\) in \(A\) generates \(A\). 
Since \(A\) is at most one-dimensional, being generated by \(x\) over \(R/\fm\), we may assume without loss of generality that the image of \(I\) generates \(A\).
Hence \(I+x\fm=xR\), so by Lemma~\ref{lem:local_nakayama}.1 we have \(I=xR\) and thus \(J\subseteq I\).
\end{proof}

\begin{definition}
Let \(R\) be an order and \(p\) a rational prime. 
We say a fractional ideal \(I\) of \(R\) is {\em invertible at \(p\)} if \(IR_p\) is an invertible ideal of \(R_p\), where \(R_p\) is the localization of \(R\) at \(p\) as a \(\Z\)-algebra.
\end{definition}

\begin{exercise}
Show that a fractional ideal is invertible if and only if it is invertible at all rational primes.
\end{exercise}

\begin{proposition}\label{prop:delta_singular_ideal}
Let \(R\) be an order of rank \(n\), let \(p\) be a prime and let \(a\in\Z_{>0}\) be such that \(0<\ord_p(a)<\ord_p(\delta(R))\).
If \(\delta(R) R^\dagger + a R \subseteq R\) is invertible at \(p\), then \(p\leq n\).
\end{proposition}
\begin{proof}
We will write \((-)_p\) for the localization at the prime \(p\).
Let \(k=(\Z/\delta(R)\Z)_p\) and consider the \(k\)-algebra
\[A=(R/\delta(R)R)_p=\prod_{\fp\mid p}(R/\delta(R) R)_\fp =: \prod_{\fp\mid p} A_\fp,\]
where the product ranges over all prime ideals of \(R\) above \(p\).
Note that \(A\) is free over \(k\) of rank \(n\). Since \(k\) is local, the \(A_\fp\) are also free.
We will show that \(p\mid \rk_k(A_\fp)\) for some \(\fp\), from which it follows that \(p\leq n\).

First we identify \(\fp\).
Let \(I=\delta(R) (R^\dagger)_p\) and \(J=a R_p\) and suppose \(I\subseteq J\).
Then \(\tfrac{\delta(R)}{a} (R^\dagger)_p \subseteq R_p\) and \(\tfrac{\delta(R)}{a}\) annihilates the \(\Z_p\)-module \((R^\dagger/R)_p\).
Note that \((R^\dagger/R)_p\) is the \(p\)-power torsion subgroup of \(R^\dagger/R\).
Thus the exponent \(\delta(R)\) of \(R^\dagger/R\) divides \(\tfrac{\delta(R)}{a}\) in \(\Z_p\).
But \(a\not\in\Z_p^*\) by assumption on \(p\), so this is a contradiction.
Thus \(I\not\subseteq J\).
By Lemma~\ref{lem:intersect_ideals} we may choose some prime \(\fp\mid p\) such that \(I_\fp\not\subseteq J_\fp\).
If we now assume that \(\delta(R) R^\dagger + a R \subseteq R\) is invertible at \(p\), then \(I_\fp+J_\fp\) is invertible.
Hence \(J_\fp\subseteq I_\fp\) by Corollary~\ref{cor:locally_principal} and Lemma~\ref{lem:local_principal_comparable}.

Recall that \(A_\fq\) is free over \(k\) for all \(\fq\mid p\).
By Exercise~\ref{ex:trace_det_product} the trace map \(\Tr_{A/k} : \prod_{\fq\mid p} A_\fq \to k\) is given by \((x_\fq)_\fq \mapsto \sum_{\fq\mid p} \Tr_{A_\fq/k}(x_\fq)\). 
Note that the image of \(I\) in \(A\) is simply \(\prod_{\fq \mid p} A I_\fq\) and similarly for \(J\).
With Exercise~\ref{ex:trace_reductions} we compute

\begin{align*} 
a\Tr_{A_\fp/k}(A_\fp) 
&= a \Tr_{A/k}(A_\fp) 
= \Tr_{A/k}(AJ_\fp) 
\subseteq \Tr_{A/k}(AI_\fp) \\
&\subseteq \Tr_{A/k}(AI) 
\subseteq \Tr_{R_p/\Z_p}(I) + \delta(R) \Z_p \\
&= \delta(R) \Tr_{\textup{Q}(R)/\Q}((R^\dagger)_p)+\delta(R)\Z_p
\subseteq \delta(R) \Z_p.
\end{align*}
It follows that \(p \mid \tfrac{\delta(R)}{a} \mid \Tr_{A_\fp/k}(1) \equiv \rk_k(A_\fp) \ (\textup{mod }p)\), as was to be shown.
\end{proof}

\begin{proposition}\label{prop:red_disc}
Let \(R\) be a full-rank order of a number field \(K\) of degree \(n\) and let \(p>n\) be prime.
Then \(p \mid \#(\mathcal{O}_K/R)\) if and only if \(p^2 \mid \delta(R)\).
\end{proposition}
\begin{proof}
If \(p \mid \#(\mathcal{O}_K/R)\), then \(\mathcal{O}_K/R\) has non-trivial \(p\)-torsion, and so does \(p(R^\dagger/R)\) by Lemma~\ref{lem:better_reduced_discriminant}. It follows that \(p^2 \mid \exp(R^\dagger/R)=\delta(R)\). Conversely, if \(p^2\mid\delta(R)\), then we obtain from Proposition~\ref{prop:delta_singular_ideal} an ideal that is not invertible at \(p\), hence \(p\mid \#(\mathcal{O}_K/R)\) by Lemma~\ref{lem:singular_discriminant}.
\end{proof}

\begin{exercise}\label{ex:small_prime_counter}
Show that there exist:
\begin{enumex}
\item an order $R$ with \(\rk(R)=2\) and \(2\mid \#(\mathcal{O}_{\Q R}/R)\) and $2^2 \nmid \delta(R)$;
\item an order $R$ with \(\rk(R)=2\) and $2 \nmid \#(\mathcal{O}_{\Q R}/R)$ and $2^2 \mid \delta(R)$.
\end{enumex}
\end{exercise}

Exercise~\ref{ex:small_prime_counter} shows that the equivalence of Proposition~\ref{prop:red_disc} fails to hold when we drop the assumption that \(p>n\).
The proposition also has an algorithmic version.

\begin{exercise}[cf.\ Theorem 6.9 in \cite{Buchman-Lenstra}]
Show that the following problems are computationally equally hard:
Given an order \(R\) of rank \(n\), decide if there exist and if so compute 
\begin{enumerate}
\item an order \(R\subsetneq S \subseteq \Q R\) such that \(p \mid \#(S/R)\) for some prime \(p > n\);
\item an \(a\in\Z_{>0}\) such that \(0<\ord_p(a)<\ord_p(\delta(R))\) for some prime \(p > n\).
\end{enumerate}
\end{exercise}

We can now prove the equivalence of several decision problems.

\begin{proof}[Proof of Theorem~\ref{thm:equivalent_decision_problems}]
By Exercise~\ref{ex:equivalent_decision_problems} it suffices to reduce the problems to decide (i) whether a finite commutative ring is reduced and (ii)~whether an order is maximal, to the problem of testing whether a positive integer is square-free.

(i) Suppose \(A\) is a finite commutative ring. 
We compute the exponent \(d=\exp(A)\) and note that if \(\textup{nil}(A)\) is to be trivial we need \(d\) to be square-free.
We test this and proceed under the assumption that \(d\) is square-free.
In this case \(d=\rad(\# A)\), so we may apply Theorem~\ref{thm:rad_to_nil} to compute \(\nil(A)\) and verify whether it is trivial.

(ii) We first verify whether the order \(R\) has full rank.
Then we may simply check whether \(\delta(R)\) is square-free to conclude that either \(R\neq \mathcal{O}_K\) or that the non-invertible primes of \(R\) must lie over primes \(p\leq [K:\Q]\) by Proposition~\ref{prop:red_disc}. 
We then apply Lemma~\ref{lem:compute_OK} with \(d\) the product of all such small primes.
\end{proof}

\begin{exercise}
Recall the definition of a projective module from Exercise~\ref{ex:projective}.
Show that there exist polynomial-time algorithms for the following problems: 
\begin{enumex}
\item Given an integer \(q>0\) and a non-zero \(\Z/q\Z\)-module \(M\), computes the maximal integer \(e\) such that \(M/eM\) is a projective \(\Z/e\Z\)-module.
\item Given an order \(R\) of degree \(n\), compute an order \(R\subseteq S\subseteq \Q R\) such that \(S^\dagger\) is invertible and \(S^\dagger/S\) is projective over \(\Z/\delta(S)\Z\) at all primes \(p>n\).
\end{enumex}
\end{exercise}

\begin{exercise}[Proposition~2.5 in \cite{Buchman-Lenstra}]\label{ex:BL}
Let \(R\) be an order and \(I\) a fractional ideal of \(R\). Show that the following are equivalent: (1) \(I\) is invertible; (2) \((R:I):(R:I)=R\); (3) \(R:(R:I)=I\) and \(I:I=R\).
\end{exercise}

\begin{exercise}[Proposition 2.7 in \cite{Buchman-Lenstra}, difficult]\label{ex:gorenstein}
Let \(R\) be an order in a number field. Show that the following are equivalent: (1) \(R^\dagger\) is invertible; (2) for all fractional ideals \(I\) of \(R\) it holds that \(I:I=R\) if and only \(I\) is invertible. We call an order with the above equivalent properties a {\em Gorenstein order}. \\
\emph{Hint:} Show that \(I^\dagger = R^\dagger : I\) for all \(I\). For (1 \(\Rightarrow\) 2), use Exercise~\ref{ex:BL}. 
\end{exercise}

\begin{exercise}
Let the notation and hypotheses be as in Theorem~\ref{thm:MM}, denote by \(f'\) the derivative of \(f\), and write \(A=\textup{Bl}(\Z+\Z\gamma)\).
\begin{enumex}[label={\textbf{\alph*}\hspace{.24em}}]
\item \hspace{-.4em}(Euler; see \cite{Weiss}*{Proposition 3-7-12})\textbf{.} Prove that the dual (Exercise~\ref{ex:general_dagger}) of the \(\Q\)-basis \(1, \gamma, \ldots, \gamma^{n-1}\) of \(K\) is \(p_{n-1}/f'(\gamma), p_{n-2}/f'(\gamma), \ldots, p_0/f'(\gamma)\).
\end{enumex}\begin{enumex}[resume]
\item Prove that \(M=D^{1-n}\) and  \(M^\dagger=D/f'(\gamma)\). 
\item Prove that \(A^\dagger=D^{1-n}\cdot M^\dagger=D\cdot M/f'(\gamma)\) and that \(A\) is a Gorenstein order (Exercise~\ref{ex:gorenstein}).
\item Prove that \(1/f'(\gamma), \gamma/f'(\gamma), \ldots, \gamma^{n-2}/f'(\gamma), a_n\gamma^{n-1}/f'(\gamma)\) is a \Z-basis for \(A^\dagger\). 
\item Show that \(f'(\gamma)=\sum_{i=0}^{n-1}p_i\gamma^{n-1-i}\in D\cdot M\).
\end{enumex}
\end{exercise}

\begin{exercise}
Let $A$ be a finite commutative ring, and let \(m\in\Z_{>0}\) such that $m A = 0$.
\begin{enumex}
\item Prove that for some $t, b\in \Z_{>0}$ the following is true: for all sequences $(d_i)_{i = 1}^t$ consisting of $t$ integers $d_i\ge b$ there exists a sequence $(e_i)_{i = 1}^t$ of integers satisfying $0\le e_i<d_i$ such that for all sequences $(f_i)_{i = 1}^t$ of monic polynomials $f_i \in \Z[X]$ satisfying $f_i \equiv X^{d_i} - X^{e_i} \bmod m\Z[X]$ there exists a surjective ring homomorphism
\[ \Z[X_1,\ldots,X_t]/(f_1(X_1),\ldots,f_t(X_t))\to A. \]
\item Prove that there exist an order $R$ in some number field, an ideal \(I\subseteq R\), and a ring isomorphism $R/I\to A$.
\item \textbf{(open)} Can we compute a triple as in (b) in polynomial time given \(A\)?
\end{enumex}
\end{exercise}

\section{Computing symbols}\label{sec:symbols}

This section will be an application of the theory from Section~\ref{sec:abgp} which is independent of the rest of this document.
In algebraic number theory we find lots of `symbols'.
It is not well-defined what a symbol is, but generally they are maps which encode algebraic properties of its parameters which also satisfy some reciprocity law.
In this section we will define and give algorithms for computing some of these symbols.
The father of all symbols is the Legendre symbol.
We will use \cite{ComputeJacobi} as a reference.

\begin{definition}
Let \(p\) be an odd prime and let \(a\) be an integer coprime to \(p\).
We define the \emph{Legendre symbol} 
\[\leg{a}{p} = \begin{cases} +1 & a\text{ is a square in } \Z/p\Z \\ -1 & \text{otherwise} \end{cases} \]
so that \(\leg{a}{p}\equiv a^{\frac{p-1}{2}} \bmod p\).
\end{definition}

We can easily compute the Legendre symbol directly from the second definition in polynomial time using a square-and-multiply algorithm modulo~\(p\).

\begin{theorem}[Quadratic reciprocity]\label{thm:quad_rec_legendre}
Suppose \(p\) and \(q\) are distinct odd primes. Then
\[ \leg{p}{q} \leg{q}{p} = (-1)^{\frac{p-1}{2}\cdot\frac{q-1}{2}}. \qed \]
\end{theorem} \leavevmode

\begin{definition}
Let \(b\) be a positive odd integer and \(a\) an integer coprime to \(b\).
In terms of the prime factorization \(b=p_1^{k_1} \dotsm p_n^{k_n}\) of \(b\) we define the \emph{Jacobi symbol}
\[ \leg{a}{b} = \leg{a}{p_1}^{k_1} \dotsm\ \  \leg{a}{p_n}^{k_n}, \]
where \(\leg{a}{p_i}\) is the Legendre symbol defined previously.
\end{definition}

Note that the Jacobi symbol extends the Legendre symbol, which justifies using the same notation for both.
To compute the Jacobi symbol directly from the definition we need to be able to factor \(b\), which is unfeasible.
Quadratic reciprocity for the Jacobi symbol gives us a better method.

\begin{proposition}[Quadratic reciprocity]\label{prop:quad_rec_jacobi}
Suppose \(a\) and \(b\) are coprime positive odd integers. Then
\[ \leg{a}{b}\leg{b}{a} = (-1)^{\frac{a-1}{2}\cdot\frac{b-1}{2}}. \] 
\end{proposition}
\begin{proof}
Note that for fixed \(b\) the maps \(a\mapsto \leg{a}{b}\leg{b}{a}\) and \(a\mapsto (-1)^{\frac{a-1}{2} \cdot \frac{b-1}{2}}\)  are multiplicative. Hence it suffices to prove the proposition for \(a\) prime.
Applying the same reasoning to \(b\) we have reduced to Theorem~\ref{thm:quad_rec_legendre}.
\end{proof}

\begin{exercise}
Let \(p\) be an odd prime. 
Show that \(2\) is a square in \(\F_p\) if and only if \(p\equiv \pm 1\bmod 8\). 
Conclude that \(\leg{2}{b}=(-1)^{(b^2-1)/8}\) for all odd \(b>0\). \\
\emph{Hint:} Write \(\sqrt{2}\) in terms of 8-th roots of unity.
\end{exercise}

\begin{exercise}\label{ex:jacobi_induction}
Let \(x_0,x_1,x_2,x_3\in\Z_{\geq 0}\) such that \(x_1\) is odd, \(x_0\) and \(x_1\) are coprime, and \(x_{n+2}\equiv x_{n} \bmod x_{n+1}\) for \(n\in\{0,1\}\). Write \(x_2=2^uv\) for \(u\in\Z_{\geq0}\) and \(v\) odd.
Show that:
\vspace{.5em}
\begin{enumex}
\item \(\displaystyle \leg{x_0}{x_1} = (-1)^{(x_1-1)(x_2-1)/4} \cdot \leg{x_1}{x_2} \quad\text{if } u=0\).
\item \(\displaystyle \leg{x_0}{x_1} = (-1)^{u\cdot\frac{x_1^2-x_3^2}{8}}(-1)^{\frac{v-1}{2}\cdot\frac{x_1-x_3}{2}} \cdot \leg{x_2}{x_3} \quad\text{if }u>0\).
\end{enumex} 
\end{exercise}\leavevmode

\begin{theorem}\label{thm:compute_jacobi}
There exists a polynomial-time algorithm that, given coprime positive odd integers \(a\) and \(b\), computes the Jacobi symbol \(\leg{a}{b}\). 
\end{theorem}
\begin{proof}
We define the \emph{gcd sequence} of positive integers \(x_0\) and \(x_1\) to be the sequence \((x_0,\dotsc,x_N)\) where 
\(x_{n+2}\) is the unique integer such that \(0\leq x_{n+2} < x_{n+1}\) and \(x_{n+2} \equiv x_{n} \bmod x_{n+1}\) and with \(x_N=0\). The proof of the Euclidean algorithm shows that this sequence contains only linearly many elements and can be computed in polynomial time. Moreover, \(x_{N-1} = \gcd(x_0,x_1)\).

Compute the gcd sequence of \(a\) and \(b\). 
Note that \(\gcd(x_{n},x_{n+1})=\gcd(a,b)=1\) for all \(n<N\).
Using Exercise~\ref{ex:jacobi_induction} we may express \(\leg{a}{b}=s_n\cdot \leg{x_n}{x_{n+1}}\) for some \(s_n\in\{\pm 1\}\) iteratively for \(n\) with \(x_{n+1}\) odd.
Clearly we may compute these \(s_n\) in polynomial time.
As \(\leg{x_{N-2}}{x_{N-1}}=\leg{x_{N-2}}{1}=1\), we simply return \(s_{N-1}\).
\end{proof}

The Jacobi symbol is defined on a subset of \(\Z^2\).
As is the theme in this document, we will `extend' the Jacobi symbol to number rings.

\begin{remark}
One is sometimes interested in a more general Jacobi symbol \(\leg{a}{b}\) for coprime integers \(a\) and \(b\), where \(b\) is not necessarily odd or positive, by defining 
\[\leg{a}{2}=(-1)^{\frac{a^2-1}{8}} \quad\text{and}\quad \leg{a}{-1} = \frac{a}{|a|} = \begin{cases} +1 & \text{if } a > 0 \\ -1 & \text{if } a < 0 \end{cases}.\]
This is called the \emph{Kronecker symbol}.
The Kronecker symbol can be computed in polynomial time by writing \(b=uc 2^k\) where \(u=\pm 1\) and \(c\) is odd and positive, and applying Theorem~\ref{thm:compute_jacobi} to the factor \(\leg{a}{c}\) in \(\leg{a}{b}=\leg{a}{u} \leg{a}{c} \leg{a}{2}{}^k\).
\end{remark}

\begin{exercise}
The Euclidean algorithm implied by Exercise~\ref{ex:mod_gcd} produces a different type of `gcd sequence' than those used in Theorem~\ref{thm:compute_jacobi}, namely those where \(|x_{n+2}|\leq |x_{n+1}|/2\) and \(x_{n+2} \equiv x_n \bmod x_{n+1}\) for all \(n\).
Give a proof of Theorem~\ref{thm:compute_jacobi} for Kronecker symbols using such gcd sequences.
\end{exercise}

\subsection{Jacobi symbols in number rings}

First we define the Legendre symbol for a general number ring.

\begin{definition}
Let \(\fp\) be a prime ideal in a number ring \(R\) of odd index \(n_\fp=\#(R/\fp)\) and let \(a\in R\) such that \(aR+\fp=R\). We define the \emph{Legendre symbol}
\[\leg{a}{\fp} = \begin{cases} +1 & \text{if \(a\) is a square in \(R/\fp\)} \\ -1 & \text{otherwise} \end{cases},\]
so that \(\leg{a}{\fp}=a^{\frac{n_\fp-1}{2}}\bmod\fp\).
\end{definition}

Extending the definition to general ideals as for the Jacobi symbol cannot be done similarly, unless \(R\) is a Dedekind domain, because it would require prime factorization of ideals.
Instead, we consider the following.

\begin{definition}
Suppose \(\fb\) is an ideal of a number ring \(R\).
For a prime \(\fp\) we define \(l_\fp(\fb)\in\Z_{\geq0}\) such that \(\#(R_\fp/\fb_\fp)=\#(R/\fp)^{l_\fp(\fb)}\).
For \(\#(R/\fb)\) odd and \(a\in R\) with \(aR+\fb=R\) we define the \emph{Jacobi symbol} by
\[\leg{a}{\fb} = \prod_{\textup{max. }\fp\subset R} \leg{a}{\fp}^{l_\fp(\fb)}.\]
\end{definition}

\begin{theorem}[Theorem in \cite{ComputeJacobi}]\label{thm:compute_jacobi_number_field}
There exists a polynomial-time algorithm that, given an order \(R\), an ideal \(\fb\) of \(R\) such that \(\#(R/\fb)\) is odd and \(a\in R\) such that \(aR+\fb=R\), computes the Jacobi symbol \(\leg{a}{\fb}\).
\end{theorem}

We will prove this theorem by expressing the Jacobi symbol in terms of yet another symbol.

\begin{exercise}\label{ex:composition_series}
Let \(\fp\) and \(\fb\) be ideals in a number ring \(R\) with \(\fp\) prime and take any composition series \(0=M_0\subsetneq M_1 \subsetneq \dotsm \subsetneq M_n= R/\fb\) of \(R/\fb\) as an \(R\)-module.
Show that \(l_\fp(\fb)\) equals the number of quotients \(M_{i+1}/M_i\) that are isomorphic to \(R/\fp\) as an \(R\)-module.
\end{exercise}

\subsection{Signs of automorphisms}

In this section we will consider the following symbol.

\begin{definition}
Let \(B\) be a finite abelian group and let \(\sigma\in\Aut(B)\). 
We define \((\sigma, B)\) to be the sign of \(\sigma\) as an element of the permutation group on \(B\).
\end{definition}

\begin{lemma}\label{lem:set_permutation}
Suppose \(A\) and \(C\) are sets and \(\alpha\in\Aut(A)\) and \(\gamma\in\Aut(C)\) are permutations.
Write \(\alpha \sqcup \gamma\) and \(\alpha\times\gamma\) respectively for the induced permutation on the disjoint union \(A\sqcup C\) and product \(A\times C\).
Then \(\sgn(\alpha\sqcup\gamma)=\sgn(\alpha)\cdot \sgn(\gamma)\) and \(\sgn(\alpha\times \gamma) = \sgn(\alpha)^{\# C} \cdot \sgn(\gamma)^{\#A}\).
\end{lemma}
\begin{proof}
That \(\sgn(\alpha\sqcup\gamma)=\sgn(\alpha)\cdot \sgn(\gamma)\) follows from the fact that \(\alpha\sqcup \gamma = \alpha \gamma\) when \(\Aut(A)\) and \(\Aut(C)\) are naturally embedded in \(\Aut(A\sqcup C)\).

For the second part write \(\alpha'=\alpha\times\id_C\) and \(\gamma'=\id_A\times \gamma\) and note that \(\alpha\times\gamma=\alpha'\cdot\gamma'\).
Now \(\alpha'\) acts as \(\alpha\) on \(\#C\) disjoint copies of \(A\), hence by the previous \(\sgn(\alpha')=\sgn(\alpha)^{\#C}\). Mutatis mutandis we obtain the same for \(\gamma'\), and the lemma follows from multiplicativity of the sign.
\end{proof}

\begin{proposition}\label{prop:exact_sign}
Suppose \(B\) is a finite abelian group and \(\beta\in\Aut(B)\).  
Suppose we have an exact sequence \(0\to A\to B \to C\to 0\) such that \(\beta\) restricts to an automorphism \(\alpha\) of \(A\).
Then \(\beta\) induces an automorphism \(\gamma\) of \(C\) such that the following diagram commutes
\begin{center}
\begin{tikzcd}[row sep = 1.4em, column sep = 2em]
0 \arrow{r}& A \arrow{r}{f}\arrow{d}{\alpha} & B \arrow{r}{g} \arrow{d}{\beta} & C \arrow{r} \arrow{d}{\gamma} & 0 \\
0 \arrow{r}& A \arrow{r}{f} & B \arrow{r}{g} & C \arrow{r} & 0
\end{tikzcd}
\end{center}
and if \(\#C\) is odd we have \((\beta,B)=(\alpha,A)\cdot(\gamma,C)^{\#A}\).
\end{proposition}
\begin{proof}
The map \(\gamma\) exists by a diagram chasing argument.
Since \(\#C\) is odd we may write \(C=\{0\}\sqcup D \sqcup(-D)\) for some subset \(D\subseteq C\).
Choosing any section \(D\to B\) of \(g\) (not necessarily a group homomorphism!), we may extend it to a section \(h:C\to B\) in such a way that \(h(-c)=-h(c)\) and \(h(0)=0\). 
Now the maps \(f\) and \(h\) together give a bijection of sets \(A\times C\to B\). 
Let \(\beta'\) be the induced action of \(\beta\) on \(A\times C\).
By Lemma~\ref{lem:set_permutation} we have that \((\alpha\times\gamma,A\times C)=(\alpha,A)^{\#C}\cdot (\gamma,C)^{\#A}\), hence to prove the proposition it suffices to show that \(\sigma = (\alpha\times\gamma)^{-1} \cdot \beta'\) is an even permutation. 
For all \(d\in D\) the action of \(\sigma\) restricts to \(A\times\{d\}\).
Note that \(\sigma\) commutes with \(-1\), hence the action of \(\sigma\) on \(A\times\{-d\}\) is isomorphic to the action on \(A\times\{d\}\).
Hence the restriction of \(\sigma\) to \(A\times (C\setminus\{0\})\) is even.
Finally, \(\sigma\) is the identity on \(A\times\{0\}\), so we conclude \(\sigma\) is even, as was to be shown.
\end{proof}

The exact sequence \(0\to 2\Z/4\Z \to \Z/4\Z \to \Z/2\Z \to 0\) resists application of Proposition~\ref{prop:exact_sign}.
If we choose the non-trivial automorphism \(\sigma\) given by \(x\mapsto -x\) on \(\Z/4\Z\) we see that its sign is \(-1\), while the induced maps on the other terms are trivial.
Hence \((\sigma,\Z/4\Z)\) is not an \(\F_2\)-linear combination of \((\sigma,2\Z/4\Z)\) and \((\sigma,\Z/2\Z)\).

\begin{exercise}\label{ex:2_power_cyclic}
Show that for \(k\in\Z_{\geq 2}\) and \(a\in(\Z/2^k\Z)^*\) we have 
\[(x\mapsto ax, \Z/2^k\Z)=(-1)^{\frac{a-1}{2}}.\] 
\emph{Hint:} Write \(\Z/2^k\Z = (\Z/2^k\Z)^* \sqcup (2\Z/2^k\Z)\) and show \((x\mapsto ax, (\Z/2^k\Z)^*)=-1\) if and only if \(a\) generates \((\Z/2^k\Z)^*\).
\end{exercise}

\begin{exercise}\label{ex:2_groups}
Suppose \(B\) is a non-trivial finite abelian group. 
For \(b\in B\) write \(\lambda_b:B\to B\) for the map \(x\mapsto b+x\). 
\begin{enumex}
\item  Show that \(\sgn(\lambda_b)=-1\) if and only if \(\#(B/\langle b\rangle)\) is odd and \(\# \langle b\rangle\) is even.
\end{enumex}
Let \(\beta\in\Aut(B)\) and suppose both \(B\) and \(\beta\) have \(2\)-power order.
\begin{enumex}[resume]
\item Show that there exists a subgroup \(C\subseteq B\) such that \(\beta(C) = C\) and \(\#(B/C)=2\).
\item Let \(b\in B\) such that \(B=C \cup (b+C)\) with \(C\) as in (b). Show that \((\beta,B)=(\lambda_{\beta(b)-b},C)\), and that \((\beta,B)=-1\) if and only if \(C=\langle \beta(b) -b\rangle\) and \(4 \mid \# B\).
\end{enumex}
We now drop the assumption that \(\beta\) has \(2\)-power order.
\begin{enumex}[resume]
\item Suppose \((\beta,B)=-1\). Show that \(B\cong\Z/2^k\Z\) for some \(k\geq 2\) or \(B\cong(\Z/2\Z)^2\).
\item Exhibit a polynomial-time algorithm that, given a finite abelian group \(A\) and \(\alpha\in\Aut(A)\) such that \(2\mid \#A\), computes \((\alpha,A)\).
\end{enumex}
\end{exercise}

\subsection{Computing signs of group automorphisms}

In this section we will prove we can compute the sign of automorphisms of finite abelian groups in polynomial time.
We will need an elementary lemma about determinants that mirrors Proposition~\ref{prop:exact_sign}.

\begin{lemma}\label{lem:mult_det}
Let \(\F\) be a field and let \(0\to A \to B \to C \to 0\) be an exact sequence of finite dimensional \(\F\)-vector spaces together with automorphism \(\alpha\), \(\beta\) and \(\gamma\) such that the diagram \begin{center}
\begin{tikzcd}[row sep = 1.4em, column sep = 2em]
0 \arrow{r}& A \arrow{r}\arrow{d}{\alpha} & B \arrow{r} \arrow{d}{\beta} & C \arrow{r} \arrow{d}{\gamma} & 0 \\
0 \arrow{r}& A \arrow{r} & B \arrow{r} & C \arrow{r} & 0
\end{tikzcd}
\end{center}
commutes. Then \(\det(\beta)=\det(\alpha)\cdot\det(\gamma)\). \qed
\end{lemma}

\noindent In terms of matrices, the above lemma simply states that for square matrices \(A\) and \(C\) and a matrix \(P\) that fits, the block matrix \(B=\left(\begin{smallmatrix} A & P \\ 0 & C \end{smallmatrix}\right)\) satisfies \(\det(B)=\det(A)\cdot\det(C)\).

\begin{theorem}[Proposition~2.3 in \cite{ComputeJacobi}]\label{thm:as_det}
Suppose \(b\in\Z_{>0}\) is odd and \(B\) is a free \((\Z/b\Z)\)-module of finite rank. Then for all \(\sigma\in\Aut(B)\) we have \((\sigma,B)=\leg{\det(\sigma)}{b}\).
\end{theorem}
\begin{proof}
For \(B=0\) the theorem clearly holds, so assume \(b\neq 1\) and that \(B\) has rank at least 1.

First suppose \(b\) is prime and \(B\) has rank \(1\). 
Then \(\sigma\) is given by multiplication with \(a\in(\Z/b\Z)^*\).
If \(a\) generates \((\Z/b\Z)^*\), then the corresponding permutation fixes \(0\) and acts transitively on the \(b-1\) remaining elements of \(B\), so that \((\sigma,B)=-1=\leg{a}{b}=\leg{\det(\sigma)}{b}\).
By multiplicativity of both symbols in \(\sigma\), this also proves the case for prime \(b\) where \(a\) is not a generator.

Now we prove using induction the case for \(\rk B \geq 2\).
Suppose \(\sigma\) is given by an upper or lower triangular matrix. 
Then there exists a subspace \(0\subsetneq A \subsetneq B\) such that \(\sigma\) restricts to \(A\).
Hence we have a split exact sequence \(0\to A\to B\to C\to0\) with \(C=B/A\) and let \(\alpha\) and \(\gamma\) be the induced maps on \(A\) and \(C\) respectively.
Then by Proposition~\ref{prop:exact_sign}, the induction hypothesis and Lemma~\ref{lem:mult_det} we get 
\[(\sigma,B)=(\alpha,A)\cdot(\gamma,C) = \leg{\det(\alpha)}{b}\cdot \leg{\det(\gamma)}{b} = \leg{\det(\alpha)\det(\gamma)}{b} = \leg{\det(\sigma)}{b}.\]
Since every matrix can be written as a product of upper and lower triangular matrices, the case for general \(\alpha\) follows.

Now we prove the theorem for general \(b\) with induction on the number of divisors of \(b\).
We have just proven the induction base with \(b\) prime.
For \(b\) not prime we may take a divisor \(1 < d < b\) of \(b\).
Let \(A=dB\) and \(C=B/A\) and note that they are free modules over \(\Z/\frac{b}{d}\Z\) and \(\Z/d\Z\) respectively.
Moreover, \(\sigma\) induces maps \(\alpha\) and \(\gamma\) on \(A\) and \(C\) respectively that make the usual diagram commute.
It follows from the definition of the determinant that \(\det(\alpha)\equiv \det(\sigma) \bmod \frac{b}{d}\) and \(\det(\gamma)\equiv\det(\sigma) \bmod d\).
Then
\[(\sigma,B)=(\alpha,A)(\gamma,C) = \leg{\det(\alpha)}{b/d} \leg{\det(\gamma)}{d} = \leg{\det(\sigma)}{b/d} \leg{\det(\sigma)}{d}=\leg{\det(\sigma)}{b}.\]
The theorem now follows by induction.
\end{proof}

\begin{theorem}[cf.\ Proposition~2.4 in \cite{ComputeJacobi}]\label{thm:compute_sign}
There exists a polynomial-time algorithm that, given an finite abelian group \(B\) and an automorphism \(\sigma\) of \(B\), computes the symbol \((\sigma, B)\).
\end{theorem}
\begin{proof}
If \(2\mid \#B\) we have Exercise~\ref{ex:2_groups}, so suppose \(B\) has odd order.
Using Theorem~\ref{thm:fundament_fingen}, write \(B\) as a product \(\prod_{k=1}^m (\Z/n_k\Z)\) of non-trivial cyclic groups such that \(n_{j} \mid n_k \) for all \(j > k\). 
Note that \(B\) fits in an exact sequence \(0\to A\to B\to C\to 0\) with \(A=n_m B\) and \(C=B/A\).
We have \(\sigma(A)=A\) and hence \(\sigma\) induces maps \(A\to A\) and \(C\to C\).
Then
\[ A=\prod_{k=1}^m  (n_m\Z/n_k\Z) \cong \prod_{k=1}^{m-1}\bigg(\Z\bigg/\frac{n_k}{n_m}\Z\bigg) \quad\text{and}\quad C\cong (\Z/n_m\Z)^{m}.\]
Since \(C\) is a free \((\Z/n_m\Z)\)-module, we may compute \((\sigma,C)\) using Theorem~\ref{thm:as_det} in polynomial time.
Note that \(A\) is a product of strictly fewer cyclic groups than \(B\), as well as having smaller order.
While \(A\neq 0\) we compute \((\sigma,A)\) recursively and apply Proposition~\ref{prop:exact_sign} to compute \((\sigma,B)\).
Since \(m\) is polynomially bounded in the length of the input, there are only polynomially many recursive steps and the algorithm runs in polynomial time.
\end{proof}

\subsection{Computing Jacobi symbols in number rings}

To compute Jacobi symbols in polynomial time it now suffices to reduce to Theorem~\ref{thm:compute_sign}.

\begin{lemma}\label{lem:reduction_jacobi_to_permutation}
Suppose \(\fb\) is an ideal in a number ring \(R\) odd index \((R:\fb)\), and suppose \(a\in R\) satisfies \(aR+\fb=R\).
Then \(\leg{a}{\fb}=( \alpha, R/\fb )\) where the map \(\alpha\) is multiplication by \(a\). 
\end{lemma}
\begin{proof}
First suppose \(\fb\) is a prime ideal and choose a generator \(a\) of  \((R/\fb)^*\).
Then \(\leg{a}{\fb}=-1\).
As \(a\) acts transitively on an even number of elements \((R/\fb)\setminus\{0\}\), we conclude that \((x\mapsto ax,R/\fb)=-1=\leg{a}{\fb}\).
The case for general \(a\) follows from multiplicativity of both symbols.

Now consider the case of general \(\fb\).
Choose some composition series \(0=M_0\subsetneq M_1 \subsetneq \dotsm \subsetneq M_n = R/\fb\) of \(R/\fb\) as an \(R\)-module.
Consider the exact sequence \(0\to M_i \to M_{i+1} \to M_{i+1}/M_i \to 0\) and note that \(M_{i+1}/M_{i} \cong R/\fp\) as \(R\)-modules for some prime ideal \(\fp_i\) of \(R\).
By applying Proposition~\ref{prop:exact_sign} inductively we obtain \((\alpha,B) = \prod_{i=1}^n (\alpha,M_{i+1}/M_i) = \prod_{i=1}^n \leg{\alpha}{\fp_i}\).
It follows from Exercise~\ref{ex:composition_series} that the latter equals \(\smash{\leg{\alpha}{\fb}}\).
\end{proof}

\begin{proof}[Proof of Theorem~\ref{thm:compute_jacobi_number_field}]
Compute \(R/\fb\) and the map \(\alpha:R/\fb\to R/\fb\) given by \(x\mapsto a x\).
Using Theorem~\ref{thm:compute_sign} compute \((\alpha,R/\fb)\), which equals \(\leg{a}{\fb}\) by Lemma~\ref{lem:reduction_jacobi_to_permutation}.
\end{proof}

\bibspread
\bibliography{citations}

\end{document}